\documentclass{aic}

\aicAUTHORdetails{%
  title = {The Tur\'an Density of 4-Uniform Tight Cycles},
  author = {Maya Sankar},
  plaintextauthor = {Maya Sankar},
  plaintexttitle = {The Turan Density of 4-Uniform Tight Cycles}, 
  keywords = {extremal numbers, hypergraphs, tight cycles},
}   

\aicEDITORdetails{%
   year={2026},
   number={4},
   received={21 November 2024},   
   published={22 May 2026}, 
   doi={10.19086/aic.2026.4},     
}   

\pdfoutput=1

\usepackage{amsmath,amssymb,amsthm}
\usepackage{thmtools}
\usepackage[capitalize, nameinlink]{cleveref}
\usepackage{enumitem}
	\setlist[enumerate,1]{label={(\arabic*)}}
\usepackage[usenames,dvipsnames]{xcolor}
\usepackage{tikz}
\usepackage{bm} 

\colorlet{mygreen}{ForestGreen!60!LimeGreen}
\colorlet{graphyellow}{yellow}
\colorlet{graphblue}{blue!60!cyan}
\colorlet{graphred}{red}
\colorlet{graphgreen}{mygreen}
\colorlet{graphpurple}{violet!90!black}

\newtheorem{theorem}{Theorem}[section]
\newtheorem{lemma}[theorem]{Lemma}
\newtheorem{proposition}[theorem]{Proposition}
\newtheorem{corollary}[theorem]{Corollary}
\newtheorem{claim}[theorem]{Claim}
\newtheorem{subclaim}[theorem]{Subclaim}
\newtheorem{obstruction}{Obstruction}
\newtheorem{step}{Step}
\newtheorem*{claim*}{Claim}
\newtheorem*{milne}{Milne's Inequality}

\theoremstyle{definition}
\newtheorem{definition}[theorem]{Definition}

\newtheorem{question}[theorem]{Question}
\newtheorem{problem}[theorem]{Problem}
\newtheorem{conjecture}[theorem]{Conjecture}

\theoremstyle{remark}
\newtheorem{remark}[theorem]{Remark}

\crefname{claim}{Claim}{Claims}
\newcommand{\proofofclaim}{\renewcommand*{\qedsymbol}{$\blacksquare$}}

\newcommand{\al}{\alpha}
\newcommand{\bet}{\beta}
\newcommand{\del}{\delta}
\newcommand{\eps}{\epsilon}
\newcommand{\ga}{\gamma}
\newcommand{\sg}{\sigma}

\newcommand{\Ga}{\Gamma}
\newcommand{\Del}{\Delta}
\newcommand{\Delbf}{{\boldsymbol{\Delta}}}

\newcommand{\texthom}{\textnormal{-hom}}
\newcommand{\ent}{{\mathrm H}}
\newcommand{\epsref}[1]{\eps_{\ref{#1}}}

\newcommand{\CC}{{\mathcal C}}
\newcommand{\EE}{{\mathcal E}}
\newcommand{\LL}{{\mathcal L}}
\newcommand{\TT}{{\mathcal T}}

\newcommand{\C}[2]{{\CC_{#1}^{(#2)}}}
\newcommand{\Vbad}{V_{\mathrm{bad}}}
\newcommand{\Ebad}{{\mathcal E}_{\mathrm{bad}}}
\newcommand{\Tbad}{{\mathcal T}_{\mathrm{bad}}}

\newcommand{\Gbar}{\overline{G}}
\newcommand{\degbar}{\overline{\deg}}

\newcommand{\Godd}{G^{\mathrm{odd}}}
\newcommand{\Gopt}{G^{\mathrm{opt}}}
\newcommand{\eopt}{e^{\mathrm{opt}}}
\newcommand{\ebar}{{\bar e}}

\DeclareMathOperator{\id}{id}

\DeclareMathOperator*{\E}{{\mathbb E}}
\DeclareMathOperator{\ex}{ex}
\DeclareMathOperator{\stab}{Stab}
\DeclareMathOperator{\tc}{tc}
\DeclareMathOperator{\indeg}{indeg}
\DeclareMathOperator{\outdeg}{outdeg}

\newcommand{\x}{{\mathbf x}}
\newcommand{\y}{{\mathbf y}}
\newcommand{\z}{{\mathbf z}}
\newcommand{\cyc}{\mathrm{cyc}}

\newcommand{\Edir}{{\vec E}}
\newcommand{\Fdir}{{\vec F}}

\newcommand{\abc}{\al,\bet,\ga}
\newcommand{\abd}{\al,\bet,\del}
\newcommand{\abcd}{\al,\bet,\ga,\del}
\newcommand{\abcdopt}{\left(\frac 14,\frac 14,\frac 14,0\right)}

\usetikzlibrary{calc, arrows.meta, positioning, shapes.geometric, bending, decorations.markings}

\newcommand{\resettikzparameters}{
	\def\pinwheelopacity{0}
	\def\circletriangleangle{240}
}
\resettikzparameters

\tikzset{vtx/.style={circle, fill, inner sep=1.5pt}}
\tikzset{
	circle arrow/.style={},
	purple opacity/.style={fill opacity=0.1},
	edge color/.style={},
}

\newcounter{uid}
\tikzset{unlabeled triangle/.style={baseline=0.8ex, scale=0.6}}
\tikzset{labeled triangle/.style={baseline=1ex,scale=0.6,every node/.style={scale=0.7}}}

\newcommand{\tikzmaketriangle}[4][]
{\path (0,0) coordinate (#1#3) -- (1,0) coordinate (#1#4) -- (60:1) coordinate (#1#2);
}
\newcommand{\tikzlabeltriangle}[7][]
{\path (#1#2) node[above=-1pt] {$#5$}
		(#1#3) node[left=-1pt] {$#6$}
		(#1#4) node[right=-1pt] {$#7$};
}

\newcommand{\tikzmakebluetriangle}[4][]
{\draw[graphblue, line width=0.45pt, fill=graphblue, fill opacity=0.2] 
	(#1#2) -- (#1#3) -- (#1#4) -- cycle;
}

\newcommand{\tikzmakepointedtriangle}[4][] {
	\path (#1#2) -- (#1#3) coordinate[pos=0.4] (#1#2#3);
	\path (#1#2) -- (#1#4) coordinate[pos=0.4] (#1#2#4);
	\fill[red, draw=graphblue] (#1#2) -- (#1#2#3) -- (#1#2#4) -- cycle;
	\fill[graphblue, draw=none, opacity=0.2] (#1#3) -- (#1#2#3) -- (#1#2#4) -- (#1#4);
	\draw[graphblue, line width=0.45pt] (#1#2) -- (#1#3) -- (#1#4) -- cycle;
}

\newcommand{\tikzmakepinwheel}[4][] {
	\path (#1#2) -- (#1#3) coordinate[pos=0.5] (#1#2#3)
			-- (#1#4) coordinate[pos=0.5] (#1#3#4)
			-- (#1#2) coordinate[pos=0.5] (#1#2#4)
			-- (#1#3#4) coordinate[pos=2/3] (#1#2#3#4);
	\if\pinwheelopacity0
	\else
	\fill[draw=none,fill=white,fill opacity=\pinwheelopacity]
		(#1#2) -- (#1#2#3#4) -- (#1#2#4) -- (#1#2)
		(#1#3) -- (#1#2#3#4) -- (#1#2#3) -- (#1#3)
		(#1#4) -- (#1#2#3#4) -- (#1#3#4) -- (#1#4);
	\fi
	\def\pinwheelopacity{0}
	\draw[graphpurple, fill=graphpurple, fill opacity=0.8]
		(#1#2) -- (#1#2#3#4) -- (#1#2#3) -- cycle
		(#1#3) -- (#1#2#3#4) -- (#1#3#4) -- cycle
		(#1#4) -- (#1#2#3#4) -- (#1#2#4) -- cycle;
	\draw[graphpurple, line width=0.45pt] (#1#2) -- (#1#3) -- (#1#4) -- cycle;
}

\newcommand{\tikzmakecircletriangle}[4][] {
	\draw[graphpurple, line width=0.45, fill=graphpurple, purple opacity] (#1#2) -- (#1#3) -- (#1#4) -- cycle;
	\coordinate (#1#2#3#4) at ($1/3*(#1#2) + 1/3*(#1#3) + 1/3*(#1#4)$);
	\coordinate (#1#2#2) at ($0.5*(#1#2)+0.5*(#1#4)$);
	\coordinate (#1#3#3) at ($0.5773502*(#1#3) - 0.288675*(#1#2) + 0.711325*(#1#4)$);
	\begin{scope}[shift=(#1#4),x={(#1#2#2)},y={(#1#3#3)}]
	\draw[->,graphpurple,>={Latex[angle=50:3pt]},circle arrow] ($(#1#2#3#4) + (\circletriangleangle:0.42)$) arc (\circletriangleangle:\circletriangleangle-340:0.42);
	\end{scope}
\def\circletriangleangle{240}
}

\newcommand{\tikzmakeshiftedvtxs}[5]{
	\stepcounter{uid}
	\coordinate (\theuid med) at ($1/3*(#1)+1/3*(#2)+1/3*(#3)$); 
	\coordinate (\theuid x) at ($(#1)!#5!(\theuid med)+#4$);
	\coordinate (\theuid y) at ($(#2)!#5!(\theuid med)+#4$);
	\coordinate (\theuid z) at ($(#3)!#5!(\theuid med)+#4$);
}
\newcommand{\trianglecommands}[4][]{
\expandafter\newcommand\csname #2\endcsname{
	\tikz[unlabeled triangle, #1] {
		\tikzmaketriangle #4;
		\csname tikzmake#3\endcsname xyz;
}}
\expandafter\newcommand\csname #2labeled\endcsname[3]{
	\tikz[labeled triangle, #1] {
		\tikzmaketriangle #4;
		\tikzlabeltriangle #4{##1}{##2}{##3}
		\csname tikzmake#3\endcsname xyz;
}}
\expandafter\newcommand\csname tikzshifted#2\endcsname[5]{
	\begin{scope}[#1]
	\tikzmakeshiftedvtxs{##1}{##2}{##3}{##4}{##5};
	\csname tikzmake#3\endcsname[\theuid]xyz;
	\end{scope}
}
}

\trianglecommands{plaintriangle}{plaintriangle}{xyz}
\trianglecommands{bluetriangle}{bluetriangle}{xyz}
\trianglecommands{graytriangle}{graytriangle}{xyz}

\trianglecommands{pointedtriangle}{pointedtriangle}{xyz}
\trianglecommands{pointedtriangleleft}{pointedtriangle}{yxz}
\trianglecommands{pointedtriangleright}{pointedtriangle}{zyx}

\trianglecommands[edge color/.prefix style={graphred},vtx color/.prefix style={graphyellow}]
{rededgetriangle}{edgetriangle}{xyz}
\trianglecommands[edge color/.prefix style={graphyellow},vtx color/.prefix style={graphred}]
{yellowedgetriangle}{edgetriangle}{xyz}
\trianglecommands[edge color/.prefix style={graphred},vtx color/.prefix style={graphyellow}]
{rededgetriangleright}{edgetriangle}{yxz}
\trianglecommands[edge color/.prefix style={graphyellow},vtx color/.prefix style={graphred}]
{yellowedgetriangleright}{edgetriangle}{yxz}

\trianglecommands{pinwheel}{pinwheel}{xyz}
\trianglecommands{pinwheelrev}{pinwheel}{yxz}
\trianglecommands{circletriangle}{circletriangle}{xyz}
\trianglecommands{circletrianglerev}{circletriangle}{xzy}

\newcommand{\circlediamond}{
\tikz[labeled triangle, baseline=-2.2pt] { 
	\path (0,0) coordinate (w) -- (-30:1) coordinate (x) -- (30:1) coordinate (y) --++ (-30:1) coordinate (z);
	\tikzmakecircletriangle ywx;
	\def\circletriangleangle{60}
	\tikzmakecircletriangle xzy;
}}
\newcommand{\circlediamondrev}{
\tikz[labeled triangle, baseline=-2.2pt] { 
	\path (0,0) coordinate (w) -- (-30:1) coordinate (x) -- (30:1) coordinate (y) --++ (-30:1) coordinate (z);
	\tikzmakecircletriangle xwy;
	\def\circletriangleangle{60}
	\tikzmakecircletriangle yzx;
}}

\tikzset{unlabeled tetrahedron/.style={baseline=0.4ex,scale=0.6}}
\tikzset{labeled tetrahedron/.style={baseline=0.42ex,scale=0.6,every node/.style={scale=0.7}}}

\newcommand{\tikzmaketetrahedron}[5][]
{\path (0,0.05) coordinate (#1#3)
	(0.75,-0.2) coordinate (#1#4)
	(1,0.33) coordinate (#1#5)
	(0.5,0.85) coordinate (#1#2);
}
\newcommand{\tikzlabeltetrahedron}[9][]
{\path (#1#2) node[above=-1.3pt] {$#6$}
	(#1#3) node[left=-1pt] {$#7$}
	(#1#4) node[below right=-3pt] {$#8$}
	(#1#5) node[right=-1pt] {$#9$};
}

\newcommand{\tikzmakepointedtetrahedron}[5][] {
	\path (#2) -- (#3) coordinate[pos=0.4] (#2#3)
		(#2) -- (#4) coordinate[pos=0.4] (#2#4)
		(#2) -- (#5) coordinate[pos=0.4] (#2#5);
	\foreach \i/\j/\op in {#3/#4/0.15,#4/#5/0.35} {
		\draw[draw=none,fill=graphred, fill opacity=0.9] (#2) -- (#2\i) -- (#2\j) -- (#2);
		\fill[draw=none,graphblue, opacity=\op] (\i) -- (#2\i) -- (#2\j) -- (\j);
	}
	\fill[graphblue, opacity=0.1, draw=none] (#3) -- (#4) -- (#5);
	\draw[graphblue,line width=0.45pt] (#2) -- (#3) -- (#4) -- (#5) -- (#2)
		(#2) -- (#4);
	\draw[graphblue, line width=0.45pt, opacity=0.45] (#3) -- (#5);
	\draw[graphblue, line width=0.3pt] (#2#3) -- (#2#4) -- (#2#5);
	\draw[graphblue, opacity=0.3, line width=0.27pt] (#2#3) -- (#2#5);
}

\newcommand{\tetrahedroncommands}[4][]{
\expandafter\newcommand\csname #2\endcsname{
	\tikz[unlabeled tetrahedron, #1] {
		\tikzmaketetrahedron wxyz;
		\csname tikzmake#3\endcsname #4;
}}
\expandafter\newcommand\csname #2labeled\endcsname[4]{
	\tikz[labeled tetrahedron, #1] {
		\tikzmaketetrahedron wxyz;
		\tikzlabeltetrahedron wxyz{##1}{##2}{##3}{##4};
		\csname tikzmake#3\endcsname #4;
}}
}
\tetrahedroncommands{pointedtetrahedron}{pointedtetrahedron}{wxyz}
\tetrahedroncommands{pointedtetrahedrontwisted}{pointedtetrahedron}{yxwz}
\tetrahedroncommands[edge1/.style={graphyellow}, edge2/.style=graphred]
	{edgetetrahedron}{edgetetrahedron}{wxyz}
\tetrahedroncommands[edge1/.style={graphred}, edge2/.style=graphyellow]
	{edgetetrahedronflipped}{edgetetrahedron}{wxyz}
\tetrahedroncommands{circletetrahedron}{circletetrahedron}{wxyz}

\newcommand{\tikztetrahedronvtxs}[1]{
	\path (0,0.05) coordinate (x)
		(0.75,-0.2) coordinate (z)
		(0.95,0.33) coordinate (y)
		(0.5,0.85) coordinate (w);
	\path (w) -- (z) node[pos=0.007,#1]{} node[pos=0.98,#1]{}
		  (x) -- (y) node[pos=0.03,#1]{} node[pos=0.99,#1]{};

}

\tikzset{tetrahedron boundary/.style={
	inline/.style={baseline=0.48ex,scale=0.9,vtx/.append style={scale=0.6}, line width=0.4pt},
	big/.style={scale=2.7, line width=0.6pt,
		circle arrow/.append style={>={Latex[angle=50:3pt,scale=1.5]}}},
}}

\newcommand{\pointedtetrahedronboundary}[1][inline]{
\tikz[tetrahedron boundary,#1] {
	\tikztetrahedronvtxs{vtx};
	\tikzshiftedpointedtriangle wxy{(0,0.01)}{0.14};
	\tikzshiftedbluetriangle xyz{(-0.03,0.01)}{0.22};
	\tikzshiftedpointedtriangle wyz{(-0.008,-0.05)}{0.28};
	\tikzshiftedpointedtriangle wxz{(0.008,0.02)}{0.27};
}}
\newcommand{\edgetetrahedronboundary}{
\tikz[tetrahedron boundary, big] {
	\tikztetrahedronvtxs{vtx};
	\begin{scope}[edge color/.append style={fill opacity=0.55}, vtx color/.append style={fill opacity=0.55}]
		\tikzshiftedyellowedgetriangle ywx{(0,0.01)}{0.14};
		\tikzshiftedrededgetriangle xyz{(-0.03,0.01)}{0.22};
	\end{scope}
	\tikzshiftedrededgetriangle wyz{(-0.008,-0.05)}{0.28};
	\tikzshiftedyellowedgetriangle zwx{(0.008,0.02)}{0.27};
}}
\newcommand{\circletetrahedronboundary}{
\tikz[tetrahedron boundary, big] {
	\tikztetrahedronvtxs{vtx};
	\begin{scope}[graphpurple/.style={graphpurple!0.8!gray},
				circle arrow/.append style={opacity=0.5}, 
				purple opacity/.style={fill opacity=0.12}]
		\def\circletriangleangle{280}
		\tikzshiftedcircletriangle wyx{(0,0.01)}{0.14};
		\tikzshiftedcircletriangle xyz{(-0.03,0.01)}{0.22};
	\end{scope}
	\begin{scope}[circle arrow/.append style={line width=0.7pt}, 
				purple opacity/.style={fill opacity=0.15}]
		\def\circletriangleangle{165}
		\tikzshiftedcircletriangle wzy{(-0.008,-0.05)}{0.28};
		\def\circletriangleangle{230}
		\tikzshiftedcircletriangle wxz{(0.01,0.02)}{0.27};
	\end{scope}
}}

\tikzset{hash1/.style={postaction={decorate}, 
        decoration={markings, mark=at position 0.5 with {\draw[-, line width=0.5mm] (0,-2pt) -- (0,2pt);}}}}
\tikzset{hash2/.style={postaction={decorate}, 
        decoration={markings,
        mark=at position 0.4 with {\draw[-, line width=0.5mm] (0,-2pt) -- (0,2pt);},
        mark=at position 0.6 with {\draw[-, line width=0.5mm] (0,-2pt) -- (0,2pt);}
        }}}
\tikzset{hash3/.style={postaction={decorate}, 
        decoration={markings,
        mark=at position 0.32 with {\draw[-, line width=0.47mm] (0,-2pt) -- (0,2pt);},
        mark=at position 0.5 with {\draw[-, line width=0.47mm] (0,-2pt) -- (0,2pt);},
        mark=at position 0.68 with {\draw[-, line width=0.47mm] (0,-2pt) -- (0,2pt);}
        }}}
\tikzset{arrow1/.style={>={
	Stealth[width=1.84mm,length=1.2mm]},
		postaction={decorate},
		decoration={markings, mark=at position 0.5 with
			{\draw[->] (-0.6mm,0) -- (0.6mm,0);}}}}
\tikzset{arrow2/.style={>={
	Stealth[width=1.84mm,length=1.2mm]},
		postaction={decorate},
		decoration={markings, mark=at position 0.5 with
			{\draw[->>] (-1.2mm,0) -- (1.2mm,0);}}}}
\tikzset{arrow3/.style={>={
	Stealth[width=1.84mm,length=1mm]},
		postaction={decorate},
		decoration={markings, mark=at position 0.5 with
			{\draw[->>>] (-1.5mm,0) -- (1.5mm,0);}}}}

\tikzset{tikz arrows/.style={
	line width=0.9pt,
	every node/.style={scale=0.8,inner sep=2pt},
	>={Stealth[width=1.84mm, length=2mm]},
	e0/.style={},e1/.style={},e2/.style={},e3/.style={},
red arrow/.style={
	graphred,
	e1/.style={arrow1},
	e2/.style={arrow2},
	e3/.style={arrow3},
	e0/.style={->}
},
blue edge/.style={
	graphblue,
	e1/.style={hash1},
	e2/.style={hash2},
	e3/.style={hash3},
	e0/.style={}
},
bipartite edge/.style={graphblue, >={Stealth[scale=1, red]}, ->}
}}

\newcommand{\arrowcommands}[3][]{
\expandafter\newcommand\csname#2\endcsname[1][#1]{
	\tikz[baseline=-0.6ex, tikz arrows]{
	\draw (0,0) coordinate[vtx] (1);
	\draw (0.75,0) coordinate[vtx] (2);
	\draw[#3,##1] (1) -- (2);
}}
\expandafter\newcommand\csname#2labeled\endcsname[3][#1]{
	\tikz[baseline=-0.6ex, tikz arrows]{
	\path (0,0) coordinate[vtx] (1) node[above] {$##2$};
	\draw (0.75,0) coordinate[vtx] (2) node[above] {$##3$};
	\draw[#3,##1] (1) -- (2);
}}
\expandafter\newcommand\csname#2xy\endcsname[1][#1]
{\csname#2labeled\endcsname[##1]{\vphantom{y}x}{y}}
}

\arrowcommands[e0]{redto}{red arrow}
\arrowcommands[e0]{blueedge}{blue edge}
\arrowcommands{greenedge}{graphgreen}
\arrowcommands{purpleto}{graphpurple, ->}
\arrowcommands{purpleot}{graphpurple, <-}
\arrowcommands{bipedge}{bipartite edge} 

\tikzset{unlabeled vtx triangle/.style={baseline=1ex, scale=0.7}} 
\tikzset{labeled vtx triangle/.style={baseline=1ex, scale=0.7}}

\newcommand{\tikzmakevtxtriangle}[4][]
{\path (0,0) coordinate[vtx] (#3) -- (1,0) coordinate[vtx] (#4) 
	-- (60:1) coordinate[vtx] (#2);
}
\newcommand{\tikzlabelvtxtriangle}[7][]
{\path (#1#2) node[left=0.5pt] {$#5$}
		(#1#3) node[left=0pt] {$#6$}
		(#1#4) node[right=0pt] {$#7$};
}

\newcommand{\triangleboundarycommands}[5][]{
\expandafter\newcommand\csname #2\endcsname[1][#1]{
	\tikz[unlabeled vtx triangle, tikz arrows] {
	\tikzmakevtxtriangle xyz;
	\draw[#3,##1] (y) -- (x);
	\draw[#4,##1] (z) -- (x);
	\draw[#5,##1] (y) -- (z);
}}
\expandafter\newcommand\csname #2labeled\endcsname[4][#1]{
	\tikz[labeled vtx triangle, tikz arrows, ##1] {
		\tikzmakevtxtriangle xyz;
		\tikzlabelvtxtriangle xyz{##2}{##3}{##4};
		\draw[#3,##1] (y) -- (x);
		\draw[#4,##1] (z) -- (x);
		\draw[#5,##1] (y) -- (z);
}}
}

\triangleboundarycommands[e0]{cherry}{red arrow}{red arrow}{blue edge}
\triangleboundarycommands[e0]{cherrypath}{red arrow}{draw=none,}{blue edge}
\triangleboundarycommands{greentri}{graphgreen}{graphgreen}{graphgreen}
\triangleboundarycommands{purpletri}{graphpurple, ->}{graphpurple, <-}{graphpurple, <-}
\triangleboundarycommands{purpletranstri}{graphpurple, ->}{graphpurple, ->}{graphpurple, ->}
\triangleboundarycommands{dirboundary}{bipartite edge}{bipartite edge}{blue edge}

\begin{document}

\begin{frontmatter}[classification=text]

\author[maya]{Maya Sankar\thanks{Supported by NSF GRFP Grant DGE-1656518 and a Hertz Fellowship.}}

\begin{abstract}
For any uniformity $r$ and residue $k$ modulo $r$, we give an exact characterization of the $r$-uniform hypergraphs that homomorphically avoid tight cycles of length $k$ modulo $r$, in terms of colorings of $(r-1)$-tuples of vertices. This generalizes the result that a graph avoids all odd closed walks if and only if it is bipartite, as well as a result of Kam\v cev, Letzter, and Pokrovskiy in uniformity 3. In fact, our characterization applies to a much larger class of families than those of the form $\C kr=\{\text{$r$-uniform tight cycles of length $k$ modulo $r$}\}$.

We also outline a general strategy to prove that, if $\CC$ is a family of tight-cycle-like hypergraphs (including but not limited to the families $\C kr$) for which the above characterization applies, then all sufficiently long  $C\in \CC$ will have the same Tur\'an density.
We demonstrate an application of this framework, proving that there exists an integer $L_0$ such that for every $L>L_0$ not divisible by 4, the tight cycle $C^{(4)}_L$ has Tur\'an density $1/2$. 
\end{abstract}
\end{frontmatter}

\section{Introduction}

One of the most fundamental problems in extremal combinatorics is understanding those graphs avoiding a fixed forbidden graph $F$. The first results in this area were for graphs forbidding a clique $K_s$: Mantel \cite{Man07} in 1907 (for $s=3)$ and Tur\'an \cite{Tu41} in 1941 (for all $s$) showed that the densest $K_s$-free graph on $n$ vertices is a balanced blowup of $K_{s-1}$, i.e., a complete $(s-1)$-partite graph. In general, given a graph $F$, its \emph{Tur\'an number} $\ex(n,F)$ is defined to be the maximum number of edges in any $F$-free graph on $n$ vertices. 
In many cases, it is impractical to determine the Tur\'an number exactly --- instead, we usually study the \emph{Tur\'an density} $\pi(F)=\lim_{n\to\infty}\ex(n,F)/\binom n2$, which limit is known to exist.

A celebrated theorem of Erd\H os, Stone, and Simonovits \cite{ErSt46,ErSi66} shows that $\pi(F)=1-\frac 1{\chi(F)-1}$ for any graph $F$. Note that this is the density of a balanced blowup of the clique $K_{\chi(F)-1}$. Shortly thereafter, Erd\H os and Simonovits \cite{Er67,Er68,Si68} proved a stronger structural result: any construction attaining $\ex(n,F)$ differs from a complete $(\chi(F)-1)$-partite graph in at most $o(n^2)$ edges as $n\to\infty$.

We may define Tur\'an numbers of $r$-uniform hypergraphs, henceforth called \emph{$r$-graphs}, analogously. If $F$ is an $r$-graph, then $\ex(n,F)$ is defined to be the maximum number of edges in an $F$-free $r$-graph on $n$ vertices, and its Tur\'an density is $\pi(F)=\lim_{n\to\infty}\ex(n,F)/\binom nr$. The $F$-free $r$-graphs attaining $\ex(n,F)$ edges are called the \emph{extremal} constructions.
In contrast to the graph case, determining the Tur\'an density of hypergraphs remains a major research area. There is no unique parameter, like the chromatic number, that determines the Tur\'an density. In some cases, the extremal $F$-free construction is the densest $r$-graph with chromatic or rainbow chromatic number smaller than $F$ \cite{MuPi07,fano-1,fano-2}. However, the extremal constructions for other $F$ are mystifying, including fractal-like iterated blowups \cite{BaLu22,KLP24,LiMaPf24}, algebraic constructions \cite{KeSu05}, and exponentially numerous families of examples \cite{BaClLu22}.

Even determining the Tur\'an densities of complete hypergraphs remains widely open. In the simplest case, Tur\'an \cite{Tu-tetrahedron} conjectured that the Tur\'an density of the \emph{tetrahedron} $K_4^{(3)}$ is $5/9$. The lower bound is attained by a large family of constructions \cite{Fo88,Fr08}, but proving a matching upper bound remains open.

This paper has two main contributions. 
The first is to introduce a new algebraic parameter that should control the Tur\'an density of various ``cycle-like'' hypergraphs, together with a general framework (given in \cref{sec:framework}) to show that these correctly determine the Tur\'an density of sufficiently long ``cycle''s. This is elaborated on later in the introduction. Secondly, we demonstrate an application of this framework by determining the Tur\'an density of 4-uniform tight cycles, defined as follows. The \emph{tight $r$-uniform cycle of length $\ell$}, denoted $C_{\ell}^{(r)}$, is an $r$-graph with $\ell> r$ vertices $v_1,\dots, v_\ell$ such that $v_i\cdots v_{i+r-1}$ is an edge for each $1\leq i\leq \ell$, with subscripts taken modulo $\ell$. When $\ell$ is a multiple of $r$, the tight cycle $C_\ell^{(r)}$ is an $r$-partite $r$-graph, and an old result of Erd\H os \cite{Er64} states that $\pi(C_{\ell}^{(r)})=0$. Recently, Kam\v cev, Letzter, and Pokrovskiy \cite{KLP24} determined $\pi(C_{\ell}^{(3)})=2\sqrt 3-3$ for all sufficiently large $\ell\not\equiv 0\pmod 3$. Partially inspired by their ideas, Balogh and Luo \cite{KLP24} showed that $\pi(C_{\ell}^{(3)-})=1/4$, where $C_\ell^{(3)-}$ is obtained by removing an edge from $C_\ell^{(3)}$. In both cases, the asymptotically optimal examples have a fractal-like structure. (The lower bounds on $\ell$ have since been improved using computer-assisted flag algebra computations, to $\ell\geq 7$ for $C_\ell^{(3)}$ by Bodn\'ar--Le\'on--Liu--Pikhurko \cite{BLLP25} and $\ell\geq 5$ for $C_\ell^{(3)-}$ independently by the same group and Lidick\'y--Mattes--Pfender \cite{BLLP24,LiMaPf24}.)

The key new component in the framework, our algebraic parameter allows us to unify key structural lemmas in \cite{BaLu22,KLP24} in a very general result that also generalizes the folklore equivalence between bipartiteness and avoiding odd cycles (see \cref{intro:generalizing-bipartite}).
Using this framework, we completely determine the Tur\'an density of all sufficiently long 4-uniform tight cycles of any other residue modulo 4.

\begin{theorem}\label{thm:unif4-density}
	There exists an integer $L_0$ such that the following holds. For any $L>L_0$ not divisible by 4, we have $\pi(C_L^{(4)})=\frac 12$.
\end{theorem}

The lower bound is given by the \emph{complete oddly bipartite} $4$-graph, which is constructed as follows. 
Let $V$ be a set of vertices partitioned as $A\cup B$. The complete oddly bipartite 4-graph $\Godd(A,B)$ has vertex set $V$ whose edge set comprises those 4-tuples intersecting $A$ in an odd number of vertices. This hypergraph is $C_\ell^{(4)}$-free for all $4\nmid\ell$ (see \cref{prop:Godd-C-free}) and has edge density approaching $\frac 12$ as $|V|\to\infty$ if the two parts have roughly equal sizes. 
Actually, $e(\Godd(A,B))$ is maximized if $|A|$ and $|B|$ differ from $|V|/2$ by $\Theta(\sqrt{|V|})$ vertices, but it is unknown which explicit choices of $|A|$ and $|B|$ maximize the part sizes \cite{KeSu05}.
For this reason, our results are stated in terms of the parameter $\eopt(n)$, defined to be the maximum number of edges in a complete oddly bipartite 4-graph on $n$ vertices, and phrased in accordance with the fact that there may be multiple nonisomorphic complete oddly bipartite 4-graphs with $\eopt(n)$ edges (although we do not know of any $n$ for which this is the case).

In fact, our full result (see \cref{cor:unif4-cycles}) shows that the extremal 4-graph(s) avoiding $C_L^{(4)}$ \emph{homomorphically} --- a slightly stronger notion of avoidance defined in \cref{subsec:homomorphisms} --- are exactly the complete oddly bipartite 4-graphs with the maximum edges. As a consequence of our techniques, the extremal $C_L^{(4)}$-free 4-graph differs from a complete oddly bipartite 4-graph in at most $o(n^4)$ edges, although it seems likely that this difference is much smaller.

\subsection{Hypergraph homomorphisms} \label{subsec:homomorphisms}
An alternate perspective on Tur\'an numbers comes through the lens of hypergraph homomorphisms. Given $r$-graphs $F$ and $G$, a \emph{homomorphism} $F\to G$ is a map $\phi:V(F)\to V(G)$ such that, for any edge $\{v_1,\ldots,v_r\}\in E(F)$, its image $\{\phi(v_1),\ldots,\phi(v_r)\}$ is an edge of $G$ (and in particular, $\phi(v_1),\ldots,\phi(v_r)$ must be distinct, as edges of $G$ have size $r$). Observe that $G$ contains $F$ as a sub-hypergraph if and only if there is a homomorphism $F\to G$ which maps every vertex of $F$ to a distinct vertex of $G$. A \emph{homomorphic copy} of $F$ in $G$ is the image of a homomorphism $F\to G$ in which vertices of $F$ are not necessarily mapped to distinct vertices of $G$. We say $G$ is \emph{$F$-hom-free} if $G$ contains no homomorphic copy of $F$. Define $\ex(n,F\texthom)$ to be the maximum number of edges in an $F$-hom-free $r$-graph on $n$ vertices, and define $\pi(F\texthom)$ analogously.

Understanding $\ex(n,F\texthom)$ provides an alternate route to computing the Tur\'an density $\pi(F)$. Indeed, a now-standard result in Tur\'an theory (see \cref{thm:density-hom}) states that, if $F$ is an $r$-graph, then $\ex(n,F)=\ex(n,F\text{-hom})+o(n^r)$. 
For 2-graphs, one observes that the clique $K_{\chi(F)}$ is a homomorphic copy of $F$ and deduces that $\ex(n,F\text{-hom})\leq\ex(n,K_{\chi(F)})$. In fact, this is an equality: by Tur\'an's theorem \cite{Tu41}, the extremal $K_{\chi(F)}$-free graph has chromatic number $\chi(F)-1$, and is thus $F$-hom-free. Thus, the homomorphism approach provides a relatively straightforward derivation of the Erd\H os--Stone--Simonovits theorem.

To further underscore this connection, fix a 2-graph $F$ and let $\mathcal G_\chi(F)$ be the set of graphs $G$ with $\chi(G)<\chi(F)$. Every graph in $\mathcal G_\chi(F)$ is $F$-hom-free; what is surprising is that, from the previous paragraph, the extremal $F$-hom-free graph on any number of vertices is always in $\mathcal G_\chi(F)$.

For hypergraphs, there are several na\"ive analogues of $\mathcal G_\chi$, but for many $F$ these families do not contain the extremal or conjecturally extremal $F$-hom-free constructions. We contribute a new \emph{tight connectivity} parameter $\tc(F)$, which is a collection of subgroups of $S_r$ measuring how a given $r$-edge can be connected to itself via tight walks. (More specifically, we define a subgroup $\tc(x_1\dots x_r)\subset S_r$ for each ordered $r$-tuple $x_1\cdots x_r$ whose support $\{x_1,\ldots,x_r\}$ is an edge of $F$, and we take $\tc(F)=\{\tc(x_1\cdots x_r):\{x_1,\ldots,x_r\}\in E(F)\}$.) Any homomorphism $F\to G$ induces an inclusion from each $\Gamma\in\tc(F)$ to some $\Gamma'\in\tc(G)$. Abbreviating by $\tc(F)\leq\tc(G)$ the condition that each $\Gamma\in\tc(F)$ is contained in some $\Gamma'\in\tc(G)$, it is natural to consider the family $\mathcal G_{\tc}(F)=\{G:\tc(F)\not\leq\tc(G)\}$. This paper shows that the extremal $C_L^{(4)}$-hom-free 4-graphs are contained in $\mathcal G_{\tc}(C_L^{(4)})$ for all sufficiently large $L\not\equiv 0\pmod 4$. Analogous results in uniformities 2 and 3 can be derived from \cite{Man07,KLP24}.

\subsection{Generalizing bipartite graphs}\label{intro:generalizing-bipartite}

Our framework to understand $\pi(C_L^{(4)})$ further develops the approach used by Kam\v cev, Letzter, and Pokrovskiy \cite{KLP24} to understand the Tur\'an density of tight cycles in uniformity 3. Let $\C kr$ be the family of $r$-uniform tight cycles of length $k$ modulo $r$; we say an $r$-graph is $\C kr$-hom-free if it is $C$-hom-free for each $C\in\C kr$. Let $\ex(n,\C kr\texthom)$ denote the maximum number of edges in an $n$-vertex $\C kr$-hom-free $r$-graph, and define $\pi(\C kr)=\pi(\C kr\texthom)$ analogously.

In \cref{sec:framework}, we describe a general approach to show that, if $\CC$ is a family of cycle-like hypergraphs (including but not restricted to the families of the form $\C kr$), then $\ex(n,C\texthom)=\ex(n,\CC\texthom)$ for any sufficiently large $C\in\CC$. Observe that $\ex(n,C\texthom)\geq\ex(n,\CC\texthom)$; the challenge lies in showing that the extremal $C$-hom-free $r$-graph is actually $\CC$-hom-free. This mirrors the ideas used in \cite{KLP24} for uniformity 3, but the tight connectivity parameter we introduce is key to deriving results in higher uniformities, and additionally allows us to extend the ideas beyond tight cycles.

Key to the framework is a coloring characterization of $\C kr$-hom-free hypergraphs, which is an interesting result in its own right. This generalizes the equivalence between avoiding odd closed walks and 2-colorability for graphs, as well as a more complicated result in uniformity 3 due to Kam\v cev, Letzter, and Pokrovskiy \cite{KLP24}. Our generalization characterizes $\C kr$-hom-free $r$-graphs in terms of colorings of $(r-1)$-tuples of vertices, making $\ex(n,\C kr\texthom)$ a tractable quantity to study.
Let us remark that we in fact characterize $\CC$-hom-free hypergraphs for a much broader class of families $\CC$ of \emph{twisted tight cycles} (discussed in \cref{subsec:twisted}) which is what allows us to generalize our framework beyond tight cycles. Notably, the specialization to 3-uniform twisted tight cycles was proven by Balogh and Luo \cite{BaLu22} when analyzing $\pi(C_\ell^{(3)-})$.

The full statement of our characterization is given in two equivalent forms in \cref{thm:visual-applicable,thm:alg-provable}, with explicit corollaries in uniformities 2, 3, and 4 stated in \cref{sec:coloring-corollaries}. To give a flavor of the result, we state one of these corollaries here, which was not known prior to this paper.

\begin{theorem}\label{thm:coloring-2mod4}
Let $G$ be a 4-graph. Then $G$ is $\C 24$-hom-free if and only if we can place one of the pictograms $\pointedtriangle$ or $\bluetriangle$ on each triple of vertices of $G$ such that, for each edge $x_1\cdots x_4\in E(G)$, the coloring restricted to those four vertices is isomorphic to $\pointedtetrahedronboundary$. 
In particular, three of the four triples $x_ix_jx_k$ will be colored with $\pointedtriangle$, with the red triangle always located at the same vertex, and the fourth triple will be colored with $\bluetriangle$.
\end{theorem}

The general characterization of $\CC$-hom-free $r$-graphs prescribes a set of pictograms (or ``oriented colors'') with which we can color each $(r-1)$-tuple of vertices --- this set is specific to each family $\CC$ --- and requires that, for each $r$-edge of $G$, the oriented colors of the sub-$(r-1)$-tuples are mutually consistent in an appropriate sense.

\paragraph{Outline.} In \cref{sec:framework}, we outline our general framework, describing the approach taken to prove \cref{thm:unif4-density}. This framework is separated into four steps, which are tackled in Sections 3--6. \cref{sec:bipartite} is devoted to our coloring characterization of $\C kr$-hom-free $r$-graphs. This section, together with the very short \cref{sec:delete-long-cycles}, constitutes the primary application of the tight connectivity parameter and handles tight cycles in arbitrary uniformities. \cref{sec:density-stability,sec:cleaning} are devoted to understanding the extremal $\C 14$-hom-free and $C_L^{(4)}$-hom-free 4-graphs, respectively. We conclude in \cref{sec:concluding} with a discussion of the implications of our approach for other classes of cycle-like hypergraphs.

\paragraph{Accessibility.} This paper includes a number of figures where color is crucial. Most notably, \cref{sec:density-stability} works with four (directed and undirected) edge colors: red \redto, blue \blueedge, green \greenedge, and purple \purpleto. Readers are encouraged to contact the author if the choice of colors poses an issue to readability. It is intentionally possible to redefine the colors used with minimal overhead (the interested reader can also do this themselves by downloading the \TeX\ source from arXiv).

\section{Framework and Preliminaries}\label{sec:framework}

In this section, we outline a general proof strategy to show that all sufficiently long tight cycles in the family $\CC=\C kr$ have the same Tur\'an density. This framework builds on the approach followed in \cite{KLP24} to understand tight cycles in uniformity 3, and it gives an outline of our proof of \cref{thm:unif4-density}. We are able to handle Steps 1 and 3 below for any $r$ and $k$ --- indeed, the characterization described in Step 1 is one of the key contributions of our paper --- and in fact, this framework is more broadly applicable to any family $\CC_\pi$ of \emph{twisted tight cycles}, as defined in \cref{subsec:twisted}.

\begin{step}[\cref{sec:bipartite}]
Show that $\CC$-hom-free $r$-graphs have an equivalent characterization in terms of (oriented) colorings of $(r-1)$-tuples of vertices.
\end{step}
\vspace{-\smallskipamount}

This step generalizes the statement that a 2-graph contains no odd closed walks if and only if it is bipartite (i.e., properly 2-colorable). In \cref{sec:bipartite}, we give a generalization of this statement that characterizes $\C kr$-hom-free $r$-graphs for all $r$ and $k$, deriving \cref{thm:coloring-2mod4} as a special case as well as corollaries in uniformities 3 and 4 in \cref{sec:coloring-corollaries}.
This result is one of our key contributions --- it enables us to reframe understanding $\ex(n,\CC\texthom)$ as a coloring problem and paves the way to understanding the extremal hypergraphs avoiding tight cycles in higher uniformities. Indeed, Kam\v cev, Letzter, and Pokrovskiy \cite{KLP24} pointed out that understanding this step in higher uniformities would be requisite to applying their strategy to longer tight cycles.

\begin{step}[\cref{sec:density-stability}]
Leveraging the classification in Step 1, identify the extremal $n$-vertex $\CC$-hom-free hypergraph(s) for all sufficiently large $n$.
\\Additionally, prove a stability result: if $G$ is $\CC$-hom-free with $n$ vertices and edge density at least $\pi(\CC)-\del$, show that $G$ differs from an extremal $\CC$-hom-free $r$-graph in at most $\eps n^r$ edges for $\eps,\del\to 0$.
\end{step} 
\vspace{-\smallskipamount}

This step must be handled separately for each family $\CC$, as the extremal structure depends heavily on the specific coloring characterization derived in Step 1. In \cref{sec:density-stability}, we show that $\pi(\C k4)=1/2$ for $k=1,2,3$, with the lower bound given by the complete oddly bipartite 4-graph. Additionally, with far more work, we show (\cref{thm:stability}) that for every $\eps$ there exists a $\del$ such that if $G$ is a $\C 14$-hom-free 4-graph with edge density at least $\frac 12-\del$, then $G$ differs from a complete oddly bipartite 4-graph on at most $\eps n^4$ edges.

\begin{step}[\cref{sec:delete-long-cycles}]
Let $C\in\CC$. Show that one may delete $O\big(|V(C)|^{-1/r}n^r\big)$ edges from $G$ so that the resulting $r$-graph is $\CC$-hom-free.
\end{step}
\vspace{-\smallskipamount}

Using the \emph{tight connectivity parameter} defined in \cref{sec:bipartite}, we are able to handle Step 3 for any family $\CC$ for which the $\CC$-hom-free $r$-graphs can be characterized by Step 1. This is a quick step, with the main ideas encapsulated in \cref{prop:delete-long-cycles}. The explicit deletion result, \cref{delete-long-cycles}, is only stated for $\CC=\C kr$ for simplicity, but it is not difficult to derive a similar result for other $\CC$.

\begin{step}[\cref{sec:cleaning}]
Suppose $G$ is a $C$-hom-free $r$-graph with $n$ vertices and $\ex(n,C\texthom)$ edges, where $C\in\CC$ is a sufficiently long cycle-like $r$-graph. 
From steps 2 and 3, conclude that $G$ differs from an extremal $\CC$-hom-free $r$-graph $\Gopt$ on an $O(|V(C)|^{-1/r})$-fraction of $r$-tuples.
By carefully analyzing the structure of $G$, show that $|E(G)\setminus E(\Gopt)|\leq 0.99|E(\Gopt)\setminus E(G)|$.
\end{step}
\vspace{-\smallskipamount}

This step must also be handled separately for each family $\CC$, as it involves  arguments that depend heavily on the structure of $\Gopt$. For the family $\C14$, this is done in \cref{subsec:prove-one-cycle}. Nevertheless, our general approach --- summarized in the next paragraph --- seems applicable in many cases where $\Gopt$ has a blowup-like structure.

Suppose a $C_L^{(4)}$-hom-free 4-graph $G$ with $n$ vertices differs from a complete oddly bipartite 4-graph $\Godd(A_1,B_1)$ on $\eps n^4$ edges. First, we slightly alter the sets $A_1$ and $B_1$ to locate a partition $V(G)=A\cup B$, so that $G$ and $\Godd(A,B)$ ``look similar'' on all but a $(1-\eps')$-fraction of triples in $\binom{V(G)}3$, collected in a set $\TT\subset\binom{V(G)}3$. The precise idea of similarity we introduce --- which generalizes naturally to higher uniformities --- allows us to easily locate homomorphic copies of forbidden tight cycles if edges of $G-\Godd(A,B)$ are placed ``near'' edges that contain a triple in $\TT$.
Second, we leverage our structural understanding to show $e(G)\leq e(\Godd(A,B))$, with equality if and only if these 4-graphs are identical.

\subsection{Notation}

Let $G$ be an $r$-uniform hypergraph. Its vertex set and edge set are denoted by $V(G)$ and by $E(G)\subseteq\binom{V(G)}r$, respectively, and we write $e(G):=|E(G)|$ for the cardinality of the latter. Its \emph{complement} $\Gbar$ is the $r$-graph on $V(G)$ with edge set $E(\Gbar)=\binom{V(G)}r-E(G)$.

Given $0<s<r$ vertices $v_1\cdots v_s\in\binom{V(G)}{s}$, their \emph{neighborhood} $N_G(v_1\cdots v_s)$ is the set of vertices $v_{s+1}\cdots v_r\in\binom{V(G)}{r-s}$ such that $v_1\cdots v_s\in E(G)$. Given a set $W\subseteq\binom{V(G)}{r-s}$, we write $N_G(v_1\cdots v_s;W):=N_G(v_1\cdots v_s)\cap W$. The \emph{degree} of $v_1\cdots v_s$ in $G$ is $\deg_G(v_1\dots v_s):=|N_G(v_1\cdots v_s)|$; similarly, we define $\deg_G(v_1\dots v_s;W):=|N_G(v_1\cdots v_s;W)|$.

In some cases, it is helpful to discuss neighborhoods and degrees with regard to some set $\EE\subseteq\binom Vr$ of $r$-subsets of a given ground set $V$, particularly in the cases $r=2$ and $r=3$. In this case, $N_\EE$ and $\deg_\EE$ implicitly refer to neighborhoods and degrees taken in the $r$-graph $G$ with $V(G)=V$ and $E(G)=\EE$.

In an $r$-graph $G$, the \emph{link} $\LL(v)$ of a vertex $v\in V(G)$ is an auxiliary $(r-1)$-graph on vertex set $V(G)-\{v\}$, whose $(r-1)$-edges $v_1\cdots v_{r-1}\in E(\LL(v))$ are in direct correspondence with $r$-edges $vv_1\cdots v_{r-1}\in E(G)$. That is, $E(\LL(v))=N_G(v)$ and $e(\LL(v))=\deg_G(v)$.

Given an integer $n$, let 
$\eopt(n)=\max_{a+b=n}\left(\binom a3\times b+a\times\binom b3\right)$
denote the maximum number of edges in a complete oddly bipartite 4-graph on $n$ vertices. The exact value of $\eopt(n)$ is unknown, as the optimal values $a,b$ will differ from $n/2$ by $\Theta(\sqrt n)$ \cite{KeSu05}. However, we note $\eopt(n)=(1-o(1))n^4/48$ in general, by taking $a=\lfloor n/2\rfloor$ and $b=\lceil n/2\rceil$.

\begin{proposition}\label{prop:eopt}
For any $n\geq 1$, we have $(n^4-6n^3)/48\leq\eopt(n)\leq n^4/48$.
\end{proposition}

\begin{proof}
For any $a+b=n$ we have
\[
6\eopt(n)\leq a^3b+ab^3=ab((a+b)^2-2ab)=ab(n^2-2ab).
\]
This quantity is maximized when $2ab=n^2/2$, in which case it equals $n^4/8$. Thus, $\eopt(n)\leq n^4/48$.

If $n=2m$ is even, we have
\[
\eopt(n)\geq 2m\binom m3=\frac{m^4-3m^3+2m^2}3\geq\frac{m^4-3m^3}{3}=\frac{n^4-6n^3}{48}.
\]
If $n=2m+1$ is odd, we have
\begin{align*}
\eopt(n)&\geq m\binom{m+1}3+(m+1)\binom m3
=\frac{(m+1)m(m-1)^2}3
=\frac{(n+1)(n-1)(n-3)^2}{48}
\\&=\frac{n^{4}-6n^{3}+8n^{2}+6n-9}{48}\geq\frac{n^4-6n^3}{48}.\qedhere
\end{align*}
\end{proof}

\subsection{Preliminaries}

We state several preliminary results that will be useful to us later.

\begin{lemma}\label{thm:density-hom}
Let $F$ be an $r$-graph. We have $\pi(F)=\pi(F\texthom)$.
\end{lemma}

The proof of \cref{thm:density-hom} is alluded to in \cite[\S2]{Ke11}, but we give an explicit proof here, leveraging the equivalence between the Tur\'an density of a hypergraph and any blowups.

\begin{proof}
Let $\mathcal F$ be the family of homomorphic copies of $F$, i.e., the family of $r$-graphs $F'$ for which there exists a surjective (on vertices and on edges) $r$-graph homomorphism $F\twoheadrightarrow F'$. By definition, $\pi(F\texthom)=\pi(\mathcal F)$. Moreover, $\pi(F)\geq\pi(\mathcal F)$ because $F\in\mathcal F$.

Let $t=v(F)$. For each $F'\in\mathcal F$, let $F'[t]$ denote the \emph{$t$-blowup} of $F'$, which is constructed by replacing each vertex of $F$ with an independent set of size $t$ and placing a complete $r$-partite $r$-graph $K^{(r)}_{t,\ldots,t}$ between any $r$ size-$t$ sets corresponding to an $r$-edge in $F'$. It is clear that $F'[t]$ contains $F$ for each homomorphic copy $F'$ of $F$.

Let $\mathcal F[t]=\{F'[t]:F'\in\mathcal F\}$. A standard result (proven for single hypergraphs in \cite[Theorem 2.2]{Ke11}, but the same proof applies to any finite family $\mathcal F$) states that $\pi(\mathcal F)=\pi(\mathcal F[t])$. However, because each $F'[t]\in\mathcal F$ contains $F$, we have $\pi(\mathcal F[t])\geq\pi(F)$.

We conclude that $\pi(F)\geq\pi(\mathcal F)=\pi(\mathcal F[t])\geq \pi(F)$.
\end{proof}

\begin{proposition}\label{prop:hom-free-chain}
Fix a uniformity $r$, positive integers $\ell,s$ and a nonnegative integer $t$. Any $C_{s\ell+rt}^{(r)}$-hom-free $r$-graph is also $C_\ell^{(r)}$-hom-free.
\end{proposition}

\begin{proof}
We prove the contrapositive. Set $m=s\ell+rt$.
Suppose that $G$ is an $r$-graph that is not $C_\ell^{(r)}$-hom-free, i.e., there is a homomorphism $C_\ell^{(r)}\to G$. 

Write $C_{m}^{(r)}=v_1\cdots v_{m}$ and $C_\ell^{(r)}=w_1\cdots w_\ell$, where edges are of the form $v_i\cdots v_{i+r-1}$ or $w_j\cdots w_{j+r-1}$ for each $i$ modulo $m$ or $j$ modulo $\ell$.
There is a homomorphism $C^{(r)}_{m}\to C_\ell^{(r)}$ given by mapping
\begin{align*}
(v_1,\ldots,v_{s\ell})&\mapsto(w_1,\ldots,w_\ell,w_1,\ldots,w_\ell,\ldots,w_1,\ldots,w_\ell)\ \text{and}
\\(v_{s\ell+1},\ldots,v_{s\ell+tr})&\mapsto(w_1,\ldots,w_r,w_1,\ldots,w_r,\ldots,w_1,\ldots,w_r).
\end{align*}
Precomposing the homomorphism $C_\ell^{(r)}\to G$ with the homomorphism $C_m^{(r)}\to C_\ell^{(r)}$ yields a homomorphism $C_m^{(r)}\to G$, showing that $G$ is not $C_m^{(r)}$-hom-free.
\end{proof}

\begin{lemma}\label{prelim:turan-regular}
Let $\mathcal F$ be a family of $r$-graphs and suppose $G$ is an $\mathcal F$-hom-free $r$-graph with $\ex(n,\mathcal F\texthom)$ edges. Then for any $v,v'\in E(G)$, we have $|\deg_G(v)-\deg_G(v')|\leq\binom n{r-2}$.
\end{lemma}

\begin{proof}
Let $v,w\in V(G)$. Let $G'=G-w$ and let $G''$ be the $r$-graph obtained from $G'$ by adding a vertex $v'$ whose neighborhood is exactly $N_{G'}(v)$. Identifying $v$ and $v'$ yields a homomorphism $G''\to G$, so if $G''$ contains a homomorphic copy of some $F\in\mathcal F$, then the composition $F\to G''\to G$ shows that $G$ does as well. Because $G$ is $\mathcal F$-hom-free, we conclude that $G''$ is $\mathcal F$-hom-free. 

Observe that $e(G')=e(G)-\deg_G(w)$ and $e(G'')= e(G')+\deg_G(v)-\deg_G(vw)$. Because $e(G)=\ex(n,\mathcal F\texthom)$ and $G''$ is $\mathcal F$-hom-free on the same number of vertices, it holds that
\[
e(G)\geq e(G'')=e(G)+\deg_G(v)-\deg_G(w)-\deg_G(vw).
\]
Thus $\deg_G(v)-\deg_G(w)\leq\deg_G(vw)\leq\binom n{r-2}$.
\end{proof}

\begin{lemma}\label{prelim:large-min-degree}
For every uniformity $r\geq 2$ and constant $\eps>0$, there exists some $\eps_0=\eps_{\thetheorem}(r,\eps)$ such that the following holds for any $0<c\leq 1$.

Suppose $G$ is an $r$-graph with $n$ vertices such that, for all $W\subseteq V(G)$, the induced subgraph $G[W]$ has at most $c|W|^r/r!$ edges. If $e(G)\geq(c-\eps_0)n^r/r!$, then there is a set $V'\subseteq V(G)$ containing at least $(1-\eps)n$ vertices such that, in the induced subgraph $G'=G[V']$, we have $\deg_{G'}(v)\geq(c-\eps)(|V'|-1)^{r-1}/(r-1)!$ for each $v\in V'$.
\end{lemma}

\begin{proof}
Suppose $\eps_0$ is sufficiently small in terms of $r$ and $\eps$, and let $G$ be an $r$-graph as described with $e(G)\geq(c-\eps_0)n^r/r!$.
Consider the following greedy process: identify vertices $v_1,v_2,\ldots$ such that $\deg_G(v_i;V-\{v_1,\ldots,v_i\})\leq(c-\eps)(n-i)^{r-1}/(r-1)!$. Suppose this process proceeds for $s=\eps n$ steps and let $V_s=V-\{v_1,\ldots,v_s\}$.  Then,
\begin{align*}
e(G)&\leq e(G[V_s])+\sum_{i=1}^{s}\frac{(c-\eps)(n-i)^{r-1}}{(r-1)!}
\\&\leq	\frac{c(n-s)^r}{r!}+\frac{c\times\left(n^r-(n-s)^r\right)}{r!}
	-\eps s\times\frac{\left(n-s\right)^{r-1}}{(r-1)!}
\\&= \frac{cn^r}{r!}
-r\eps\times\frac sn\left(1-\frac sn\right)^{r-1}\times\frac{n^r}{r!}
\leq\left(c-r\eps^2(1-\eps)^{r-1}\right)\frac{n^r}{r!}
<(c-\eps_0)\frac{n^r}{r!}
\end{align*}
if $\eps_0$ is sufficiently small in terms of $\eps$ and $r$, contradicting the lower bound on $e(G)$.

Thus, this process terminates in $t<\eps n$ steps. The resulting set $V'=V_t=V-\{v_1,\ldots,v_t\}$ has $m\geq n-\eps n$ vertices, each of which has degree at least $(c-\eps)(m-1)^{r-1}/(r-1)!$.
\end{proof}

\section{Characterizing Hypergraphs Avoiding Certain Walks}
\label{sec:bipartite}

In this section, we develop a hypergraph analogue of the following well-known statement.
\begin{proposition}\label{prop:bipartite}
	Let $G$ be a graph. There is a proper 2-coloring of $V(G)$ if and only if $G$ contains no odd closed walks.
\end{proposition}
Our result (stated pictorially in \cref{thm:visual-applicable} and more formally in \cref{thm:alg-provable} below) is phrased in terms of \emph{oriented colorings} of the edge set of an $r$-graph. These can be described either pictorially or symbolically.

We first state our main result pictorially as \cref{thm:visual-applicable}, characterizing $\C kr$-hom-free $r$-graphs as those colorable by a specific set of pictograms in an \emph{accordant} manner. The pictorial perspective given in \cref{sec:pictorial-oriented-colorings} is primarily intended to motivate the formal symbolic definitions given in \cref{sec:algebraic-oriented-colorings}, using which we state an equivalent non-pictorial formulation of \cref{thm:visual-applicable}. This equivalent formulation, which is \cref{thm:alg-provable}, is proven in \cref{sec:prove-coloring-result}. Lastly, in \cref{sec:coloring-corollaries} we apply \cref{thm:visual-applicable} to derive explicit pictorial characterizations of $\C kr$-hom-free $r$-graphs (including \cref{thm:coloring-2mod4}) for all uniformities $r\leq 4$ and all residues $k$ modulo $r$. The characterizations in uniformities 2 and 3 were known prior to this paper --- indeed, the 2-uniform result is precisely \cref{prop:bipartite} --- but we feel it is instructive to illustrate how they are both special cases of this more general characterization.

\subsection{Pictorial perspective}\label{sec:pictorial-oriented-colorings}
The goal of this subsection is to phrase our characterization in an informal but visually intuitive manner in \cref{thm:visual-applicable}. 
Fix a uniformity $r$, and suppose $\Delbf=\{\Delta_1,\ldots,\Delta_m\}$ is some set of pictograms, each depicting a coloring of the $r$-vertex geometric simplex stabilized by some symmetries. As an example, if $r=3$, the three-vertex simplex is a triangle, and two possible pictograms are $\Delta_{\mathrm{refl}}=\pointedtriangle$ and $\Delta_{\mathrm{rot}}=\pinwheel$. The former is stabilized by reflection across the vertical axis, while the latter is stabilized by rotation by 120 or 240 degrees; both are subgroups of the dihedral group $D_3=S_3$. In general, the full symmetry group of the $r$-vertex simplex is the symmetric group $S_r$, where each symmetry of the simplex corresponds to a permutation of its $r$ vertices. Each pictogram $\Delta$ is stabilized by some subgroup of $S_r$; this group is denoted by $\stab(\Delta)$. When formalizing our setup in \cref{sec:algebraic-oriented-colorings}, the subgroup $\stab(\Delta)\subseteq S_r$ is the only combinatorial data about the pictogram $\Delta$ that will be used; the geometric coloring is simply for visual intuition.

Given an $r$-graph $G$, an \emph{oriented coloring} of $E(G)$ by the set $\Delbf=\{\Del_1,\ldots,\Del_m\}$ involves placing one of the pictograms $\Delta_i$ on each $r$-edge of $G$. An example oriented coloring of a 3-graph by $\Delbf=\{\Del_{\mathrm{refl}},\Del_{\mathrm{rot}}\}$ is given in \Cref{fig:accordantoricoloring}. An oriented coloring is said to be \emph{accordant} if, for any two edges $e=x_1\cdots x_{r-1}y$ and $e'=x_1\cdots x_{r-1}y'$ in $E(G)$ differing at a single vertex, the edges $e$ and $e'$ are colored with the same orientation of the same pictogram, where we identify $e$ with $e'$ by mapping $y$ to $y'$. Visually, this means the coloring should be reflected across the subsimplex $x_1\cdots x_{r-1}$.
 The oriented coloring pictured in \Cref{fig:accordantoricoloring} is accordant; we remark that, two intersecting edges that differ on two or more vertices need not be colored in the same fashion. Non-accordant oriented colorings are pictured in \Cref{fig:nonaccordantcoloring}.

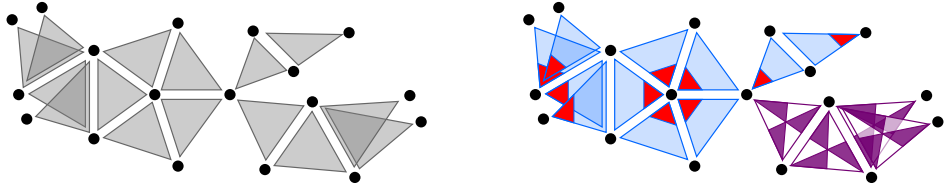
\begin{figure}[t]
\centering
\begin{tikzpicture}
	\coordinate[vtx] (x) at (0,0);
	\foreach \i in {1,...,5} {
		\coordinate[vtx] (x\i) at (72*\i:1);
	}
	\coordinate[vtx] (y) at ($(x2|-x)+(180:1)$);
	\coordinate[vtx] (y2) at ($(x2|-x)+(200:0.95)$);
	\path (y) --+ (95:1) coordinate[vtx] (z) --+ (75:1.2) coordinate[vtx] (z2);
	\path[rotate=45] (x5) --+ (25:0.9) coordinate[vtx] (a1) 
		--++ (-25:0.9) coordinate[vtx] (a2) 
		--++ (-10:0.9) coordinate[vtx] (a3);
	\path[rotate=-35] (x5) --+ (30:1.1) coordinate[vtx] (b1) 
		--++ (-30:1.1) coordinate[vtx] (b2) 
		--++ (30:1.2) coordinate[vtx] (b3) 
		--++ (75:1.15) coordinate[vtx] (b4);
	\path (b4) --++ (-0.15,0.35) coordinate[vtx] (b5);

	\tikzshiftedgraytriangle{x}{x1}{x2}{(0,0)}{0.2}
	\tikzshiftedgraytriangle{x}{x2}{x3}{(0,0)}{0.2}
	\tikzshiftedgraytriangle{x}{x3}{x4}{(0,0)}{0.2}
	\tikzshiftedgraytriangle{x}{x4}{x5}{(0,0)}{0.2}
	\tikzshiftedgraytriangle{x}{x5}{x1}{(0,0)}{0.2}
	\tikzshiftedgraytriangle{y}{x2}{x3}{(-0.03,0.01)}{0.25}
	\tikzshiftedgraytriangle{y2}{x2}{x3}{(0,-0.01)}{0.2}
	\tikzshiftedgraytriangle{y}{z}{x2}{(0,-0.01)}{0.2}
	\tikzshiftedgraytriangle{y}{z2}{x2}{(0,0.04)}{0.25}
	\tikzshiftedgraytriangle{x5}{a1}{a2}{(0,0)}{0.2}
	\tikzshiftedgraytriangle{a3}{a1}{a2}{(0.05,0.01)}{0.21}
	\tikzshiftedgraytriangle{x5}{b1}{b2}{(0,0)}{0.2}
	\tikzshiftedgraytriangle{b3}{b1}{b2}{(0,0)}{0.2}
	\tikzshiftedgraytriangle{b3}{b1}{b4}{(0.04,0)}{0.2}
	\tikzshiftedgraytriangle{b3}{b1}{b5}{(-0.02,0)}{0.2}
\end{tikzpicture}	
\qquad\quad
\begin{tikzpicture}
	\coordinate[vtx] (x) at (0,0);
	\foreach \i in {1,...,5} {
		\coordinate[vtx] (x\i) at (72*\i:1);
	}
	\coordinate[vtx] (y) at ($(x2|-x)+(180:1)$);
	\coordinate[vtx] (y2) at ($(x2|-x)+(200:0.95)$);
	\path (y) --+ (95:1) coordinate[vtx] (z) --+ (75:1.2) coordinate[vtx] (z2);
	\path[rotate=45] (x5) --+ (25:0.9) coordinate[vtx] (a1) 
		--++ (-25:0.9) coordinate[vtx] (a2) 
		--++ (-10:0.9) coordinate[vtx] (a3);
	\path[rotate=-35] (x5) --+ (30:1.1) coordinate[vtx] (b1) 
		--++ (-30:1.1) coordinate[vtx] (b2) 
		--++ (30:1.2) coordinate[vtx] (b3) 
		--++ (75:1.15) coordinate[vtx] (b4);
	\path (b4) --++ (-0.15,0.35) coordinate[vtx] (b5);
	
	\tikzshiftedpointedtriangle{x}{x1}{x2}{(0,0)}{0.2}
	\tikzshiftedpointedtriangle{x}{x2}{x3}{(0,0)}{0.2}
	\tikzshiftedpointedtriangle{x}{x3}{x4}{(0,0)}{0.2}
	\tikzshiftedpointedtriangle{x}{x4}{x5}{(0,0)}{0.2}
	\tikzshiftedpointedtriangle{x}{x5}{x1}{(0,0)}{0.2}
	\tikzshiftedpointedtriangle{y}{x2}{x3}{(-0.03,0.01)}{0.25}
	\tikzshiftedpointedtriangle{y2}{x2}{x3}{(0,-0.01)}{0.2}
	\tikzshiftedpointedtriangle{y}{z}{x2}{(0,-0.01)}{0.2}
	\tikzshiftedpointedtriangle{y}{z2}{x2}{(0,0.04)}{0.25}
	\tikzshiftedpointedtriangle{x5}{a1}{a2}{(0,0)}{0.2}
	\tikzshiftedpointedtriangle{a3}{a1}{a2}{(0.05,0.01)}{0.21}
	\tikzshiftedpinwheel{x5}{b1}{b2}{(0,0)}{0.2}
	\tikzshiftedpinwheel{b3}{b1}{b2}{(0,0)}{0.2}
	\tikzshiftedpinwheel{b3}{b1}{b5}{(-0.02,0)}{0.2}
	\def\pinwheelopacity{0.6}
	\tikzshiftedpinwheel{b3}{b1}{b4}{(0.04,0)}{0.2}
\end{tikzpicture}

\caption{At left, a 3-graph $G$. At right, an (accordant) oriented coloring of $E(G)$ by $\Delbf=\{\Delta_{\mathrm{refl}},\Delta_{\mathrm{rot}}\}$.}
\label{fig:accordantoricoloring}
\end{figure}

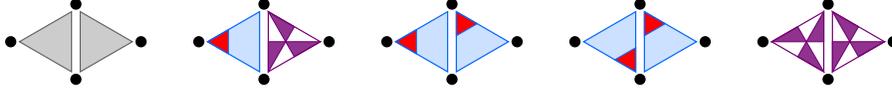
\begin{figure}
\centering
\begin{tikzpicture}
	\path\foreach \x in {v,x,a,b,c} {
		--++ (2.5,0) coordinate[vtx] (\x 0)
	};
	\foreach \x in {v,x,a,b,c} {
		\path (\x 0) --++ (30:1) coordinate[vtx] (\x 1)
			--++ (-90:1) coordinate[vtx] (\x 2)
			--++ (30:1) coordinate[vtx] (\x 3);
	}
	\tikzshiftedgraytriangle{v0}{v1}{v2}{(0,0)}{0.2}
	\tikzshiftedgraytriangle{v3}{v1}{v2}{(0,0)}{0.2}
	
	\tikzshiftedpointedtriangle{x0}{x1}{x2}{(0,0)}{0.2}
	\tikzshiftedpinwheel{x1}{x3}{x2}{(0,0)}{0.2}
	
	\tikzshiftedpointedtriangle{a0}{a1}{a2}{(0,0)}{0.2}
	\tikzshiftedpointedtriangle{a1}{a3}{a2}{(0,0)}{0.2}
	
	\tikzshiftedpointedtriangle{b2}{b1}{b0}{(0,0)}{0.2}
	\tikzshiftedpointedtriangle{b1}{b3}{b2}{(0,0)}{0.2}
	
	\tikzshiftedpinwheel{c0}{c1}{c2}{(0,0)}{0.2}
	\tikzshiftedpinwheel{c1}{c3}{c2}{(0,0)}{0.2}
\end{tikzpicture}

\caption{A 3-graph $G$ followed by four non-accordant oriented colorings of $E(G)$ by $\Delbf=\{\Delta_{\mathrm{refl}},\Delta_{\mathrm{rot}}\}$.}
\label{fig:nonaccordantcoloring}	
\end{figure}

We characterize $\C kr$-hom-free $r$-graphs as exactly those whose edge sets admit an accordant coloring by a specific set $\Delbf_{r,k}$. To state this result, we require a small amount of group theory.
Given a subgroup $\Ga$ of $S_r$, write $[\Ga]=\{\sg\Ga\sg^{-1} : \sg\in S_r\}$ for the family of conjugates of $\Ga$ in $S_r$. Given $\pi\in S_r$, we say $\Ga$ \emph{avoids $\pi$-conjugates} if $\Ga$ contains no element of the form $\sg\pi\sg^{-1}$ --- equivalently, $\pi\notin\bigcup_{\Ga'\in[\Ga]}\Ga'$.
Write $\mathfrak G_\pi$ for the set of maximal $\pi$-conjugate avoiding subgroups of $S_r$. Observe that $\mathfrak G_\pi$ may be partitioned as $\mathfrak G_\pi=\bigsqcup_{i=1}^m[\Ga_i]$ where $\Ga_1,\ldots,\Ga_m$ are subgroups of $S_r$. For $i=1,\ldots,m$, let $\Delta_i$ be any pictogram with $\stab(\Delta_i)=\Ga_i$. (One way to construct such a pictogram, exemplified by $\pinwheel$, is by taking the barycentric subdivision of the $(r-1)$-simplex, noting that $\Ga$ acts on the $r!$ resulting $(r-1)$-simplices, and coloring the union of the $(r-1)$-simplices in one orbit of $\Ga$ purple and the remainder of the simplex white. We do not dwell further on this construction because the pictograms are primarily for visual intuiton.)
Set $\Delbf_\pi=\{\Del_1,\ldots,\Del_m\}$. Using this notation, we may state the pictorial form of our characterization.

\begin{theorem}\label{thm:visual-applicable}
Fix a uniformity $r\geq 2$ and a residue $k$ modulo $r$. Let $\cyc=(1\,2\,\cdots\,r)\in S_r$ denote the cyclic shift permutation, and set $\pi=\cyc^k$. An $r$-graph $G$ is $\C kr$-hom-free if and only if $E(G)$ admits an accordant oriented coloring by the set $\Delbf_\pi$ defined above.
\end{theorem}

\subsection{Formalizing the pictogram perspective}\label{sec:algebraic-oriented-colorings}

Although oriented colorings as defined in the prior section (which we term the \emph{pictorial} definition) are comparatively easy to think about, it is challenging to discuss them formally. In this subsection, we formalize the pictorial perspective by viewing an oriented coloring as an equivariant map between two \emph{$S_r$-sets}, i.e., sets acted on by the symmetric group $S_r$. We then derive an equivalent formulation of \cref{thm:visual-applicable} (this is \cref{thm:alg-provable} below) which completely avoids reference to pictograms --- 
ultimately we will prove this non-pictorial formulation in \cref{sec:prove-coloring-result}.

Fix a uniformity $r$ and let $G$ be an $r$-graph. An \emph{oriented edge} of $G$ is an (ordered) $r$-tuple $\x=x_1\cdots x_r$ whose support $\{x_1,\dots,x_r\}$ is an edge of $G$. We denote by $\Edir(G)$ the set of oriented edges of $G$, which is naturally equipped with an $S_r$-action. Formally, the action of a permutation $\pi:[r]\to[r]$ on an oriented edge $\x=x_1\cdots x_r$ is given by
\[
\pi(x_1\cdots x_r)=x_{\pi^{-1}(1)}\cdots x_{\pi^{-1}(r)};
\]
one may check that $\pi_2(\pi_1(\x))=(\pi_2\circ\pi_1)(\x)$ for any $\pi_1,\pi_2\in S_r$.

Let $\Delbf=\{\Del_1,\ldots,\Del_m\}$ be a set of $r$-vertex pictograms. Let $A_\Delbf$ be the set of all rotations/reflections of any pictogram $\Del_i$. Using the 3-uniform example $\Delbf=\left\{\pointedtriangle,\pinwheel\right\}$, the resulting $S_3$-set would be $A_\Delbf=\left\{\pointedtriangle,\pointedtriangleleft,\pointedtriangleright,\pinwheel,\pinwheelrev\right\}$. Each $\pi\in S_r$ acts on the $r$-vertex simplex by permuting its vertices: $\pi$ sends the $i$th vertex of the simplex to the $\pi(i)$th vertex of the simplex. When $r=3$, for example, we may order the 3 vertices of the triangle as $\plaintrianglelabeled 123{}$, in which case each $\pi\in S_3$ acts on any triangle pictogram by sending the vertex labeled $i$ to the vertex labeled $\pi(i)$.

We claim that an oriented coloring of $E(G)$ by a set $\Delbf$ of pictograms carries the same data as an \emph{$S_r$-equivariant} map $\chi:\Edir(G)\to A_\Delbf$, i.e., a map $\chi$ satisfying $\chi(\pi(\x))=\pi(\chi(\x))$ for all $\x\in\Edir(G)$ and $\pi\in S_r$. Given a pictorial oriented coloring of $E(G)$, construct $\chi:\Edir(G)\to A_\Delbf$ by defining $\chi(x_1\cdots x_r)$ to be the pictogram placed on the simplex $x_1\cdots x_r$, where the simplex is oriented such that $x_i$ is placed at the $i$th vertex of the simplex. Continuing with our 3-uniform example, if a simplex $x_1x_2x_3$ is colored as $\pointedtrianglelabeled{x_1}{x_2}{x_3}$, then we set $\chi(x_1x_2x_3)=\pointedtriangle$ and $\chi(x_2x_1x_3)=\pointedtriangleleft$.

The fact that $\chi(\pi(x_1\cdots x_r))$ must be an appropriate rotation/reflection of $\chi(x_1\cdots x_r)$ for any $\pi\in S_r$ is equivalent to the $S_r$-equivariance of $\chi$. Indeed, $\chi(\pi(x_1\cdots x_r))=\chi(x_{\pi^{-1}(1)}\cdots x_{\pi^{-1}(r)})$ maps $x_1\cdots x_r$ to the pictogram placed on the simplex $x_1\cdots x_r$ oriented such that $x_{\pi^{-1}(i)}$ is the $i$th vertex for each $i$. Furthermore, $\pi(\chi(x_1\cdots x_r))$ maps $x_1\cdots x_r$ to the pictogram placed on the simplex $x_1\cdots x_r$ oriented with $x_i$ at the $i$th vertex, and then rotates/reflects this pictogram by $\pi$. The result is the pictogram placed on the simplex $x_1\cdots x_r$ oriented such that $x_i$ is the $\pi(i)$th vertex for each $i$ --- equivalently, $x_{\pi^{-1}(j)}$ is the $j$th vertex for each $j$. Thus, $\chi(\pi(\x))=\pi(\chi(\x))$ for each oriented edge $\x=x_1\cdots x_r\in\Edir(G)$, i.e., $\chi$ is $S_r$-equivariant.

Lastly, observe that a pictorial oriented coloring of $E(G)$ is accordant if and only if, for any two oriented edges $\x=x_1\cdots x_r$ and $\x'=x_1\cdots x_{i-1}x'_ix_{i+1}\cdots x_r$ in $\Edir(G)$ differing at exactly one vertex, the corresponding $S_r$-equivariant map $\chi$ satisfies $\chi(\x)=\chi(\x')$. Thus, we have derived a symbolic definition of accordant oriented colorings.

\begin{definition}\label{def:ori-coloring-formal}
Let $G$ be an $r$-graph and let $A$ be an $S_r$-set. An \emph{oriented coloring} of $\Edir(G)$ by the set $A$, also called an \emph{$A$-coloring}, is an $S_r$-equivariant map $\chi:\Edir(G)\to A$. We say $\chi$ is \emph{accordant} if, for any two oriented edges $\x=x_1\cdots x_r$ and $\x'=x_1\cdots x_{i-1}x'_ix_{i+1}\cdots x_r$ in $\Edir(G)$ differing in exactly one coordinate, it holds that $\chi(\x)=\chi(\x')$.
\end{definition}

Using oriented colorings in the sense of \cref{def:ori-coloring-formal}, \cref{thm:visual-applicable} states that an $r$-graph is $\C kr$-free if and only if there is an accordant $A_{\Delbf_\pi}$-coloring of $\Edir(G)$, where $\pi=\cyc^k\in S_r$ and $\Delbf_\pi$ is the set defined in \cref{sec:pictorial-oriented-colorings}.
To formulate a result that completely avoids reference to pictograms, we define the following $S_r$-set $A_\pi$, which we show is isomorphic (as an $S_r$-set) to $A_{\Delbf_\pi}$.

 Given a permutation $\pi\in S_r$, recall that $\mathfrak G_\pi$ is defined to be the set of maximal $\pi$-conjugate avoiding subgroups of $S_r$. Recall that $\mathfrak G_\pi$ may be partitioned as $\bigsqcup_{i=1}^m[\Ga_i]$, where $\Ga_1,\ldots,\Ga_m$ are subgroups of $S_r$ and $[\Ga_i]$ denotes the family of conjugates of $\Ga_i$ in $S_r$. Set $A_\pi=\bigsqcup_{i=1}^m(S_r/\Ga_i)$, where $S_r/\Ga_i$ denotes the set of left cosets of $\Ga_i$ acted on by left multiplication by $S_r$.
 
 Recall that $\Delbf_\pi=\{\Del_1,\ldots,\Del_m\}$ is a set of pictograms with each pictogram $\Del_i$ stabilized by $\Ga_i$. Hence the orbit of $\Del_i$ under rotations/reflections from $S_r$ is isomorphic to the $S_r$-set $S_r/\stab(\Del_i)=S_r/\Ga_i$. (This isomorphism can be constructed by following the proof of the orbit--stabilizer theorem.) As $A_{\Delbf_\pi}$ is the disjoint union of these orbits, it follows that $A_{\Delbf_\pi}$ is isomorphic to $A_\pi$. Thus, we have derived the following equivalent formulation of \cref{thm:visual-applicable}.

\begin{theorem}[reformulation of \cref{thm:visual-applicable}]
\label{thm:alg-provable}
Fix a uniformity $r\geq 2$ and a residue $k$ modulo $r$. Let $\cyc=(1\,2\,\cdots\,r)\in S_r$ denote the cyclic shift permutation, and set $\pi=\cyc^k$. An $r$-graph $G$ is $\C kr$-hom-free if and only if there is an accordant $A_\pi$-coloring $\chi:\Edir(G)\to A_\pi$, where $A_\pi$ is the set defined above.
\end{theorem}

\subsection{Proof of \cref{thm:alg-provable}}\label{sec:prove-coloring-result}
Let $G$ be an $r$-uniform hypergraph and fix oriented edges $\x,\y\in\Edir(G)$. We say $\x$ is \emph{tightly connected} to $\y$ if there is a sequence of oriented edges
\[
\x=\z^{(0)},\,\z^{(1)},\,\ldots,\,\z^{(s)}=\y
\]
of $G$ such that, for all $0\leq i<s$, the $r$-tuples $\z^{(i)}=z_1^{(i)}\cdots z_r^{(i)}$ and $\z^{(i+1)}=z^{(i+1)}_1\cdots z_r^{(i+1)}$ differ on at most one coordinate; equivalently $z^{(i)}_j=z^{(i-1)}_j$ for at least $r-1$ indices $j\in[r]$. Define
\[
\tc(\x)=\{\pi\in S_r:\x\text{ is tightly connected to }\pi(\x)\}.
\]
Our proof of \cref{thm:alg-provable} hinges on analyzing the sets $\tc(\x)$. 

It is clear that tight connectedness forms an equivalence relation on $\Edir(G)$. We state several more elementary properties of tight connectedness and the sets $\tc(\x)$ that will prove useful throughout this subsection.

\begin{proposition}\label{prop:tc-properties}
Let $G$ be an $r$-graph and fix $\x,\y\in\Edir(G)$.
\begin{enumerate}
	\item If $\x$ is tightly connected to $\y$, then $\pi(\x)$ is tightly connected to $\pi(\y)$ for all $\pi\in S_r$.
	\item The set $\tc(\x)$ is a subgroup of $S_r$.
	\item It holds that $\tc(\sg(\x))=\sg\tc(\x)\sg^{-1}$ for all $\sg\in S_r$.
\end{enumerate}
\end{proposition}

\begin{proof}
(1) Suppose $\x$ is tightly connected to $\y$ via the sequence
\[
\x=\z^{(0)},\,\z^{(1)},\,\ldots,\,\z^{(s)}=\y,
\]
where $\z^{(i-1)}$ and $\z^{(i)}$ differ on at most coordinate for each $1\leq i\leq s$.
Then $\pi(\z^{(i-1)})$ and $\pi(\z^{(i)})$ also differ on at most one coordinate, so the sequence
\[
\pi(\x)=\pi(\z^{(0)}),\,\pi(\z^{(1)}),\,\ldots,\,\pi(\z^{(s)})=\pi(\y)
\]
proves that $\pi(\x)$ is tightly connected to $\pi(\y)$.

(2) We check that $\tc(\x)$ contains the identity and is closed under inverses and composition. $\x$ is tightly connected to itself via a length-1 sequence $\x=\z^{(0)}=\x$ so $\id\in\tc(\x)$. If $\pi\in\tc(\x)$, i.e., $\x$ is tightly connected to $\pi(\x)$, then $\pi^{-1}(\x)$ is tightly connected to $\x$ by (1). It follows that $\pi^{-1}\in\tc(\x)$. Lastly, suppose $\pi,\pi'\in\tc(\x)$. Because $\x$ is tightly connected to $\pi'(\x)$, it follows by (1) that $\pi(\x)$ is tightly connected to $\pi\pi'(\x)$. Concatenating the sequences of oriented edges tightly connecting $\x$ to $\pi(\x)$ and $\pi(\x)$ to $\pi\pi'(\x)$, we conclude that $\x$ is tightly connected to $\pi\pi'(\x)$, and hence that $\pi\pi'(\x)\in\tc(\x)$. Thus, $\tc(\x)$ contains the identity and is closed under inverses and composition, so $\tc(\x)$ is a subgroup of $S_r$.

(3) Fix $\sg\in S_r$. If $\pi\in\tc(\x)$, i.e., $\x$ is tightly connected to $\pi(\x)$, then (1) implies that $\sg(\x)$ is tightly connected to $\sg\pi(\x)=\sg\pi\sg^{-1}(\sg(\x))$, so $\sg\tc(\x)\sg^{-1}\subseteq\tc(\sg(\x))$. The analogous argument applied to $\sg'=\sg^{-1}$ and $\x'=\sg(\x)$ shows that $\sg^{-1}\tc(\sg(\x))\sg\subseteq\tc(\x)$; equivalently, $\tc(\sg(\x))\subseteq\sg\tc(\x)\sg^{-1}$.
\end{proof}

The next two propositions combine to show that an $r$-graph $G$ admits an accordant $A_\pi$-coloring of $\Edir(G)$ if and only if the groups $\tc(\x)$ avoid $\pi$-conjugates for all $\x\in\Edir(G)$ --- this result is stated in \cref{cor:accordant-stab} below. \cref{prop:accordant-implies-stab} implies that each group $\tc(\x)$ must avoid $\pi$-conjugates if there is an accordant $A_\pi$-coloring of $\Edir(G)$ and \cref{prop:stab-implies-accordant} shows that there is an $A_\pi$-coloring of $\Edir(G)$ if the groups $\tc(\x)$ avoid $\pi$-conjugates.

\begin{proposition}\label{prop:accordant-implies-stab}
Fix a uniformity $r$. Let $G$ be an $r$-graph and let $\chi$ be an accordant oriented coloring of $\Edir(G)$ by an $S_r$-set $A$. We have $\tc(\x)\subseteq\stab(\chi(\x))$ for each $\x\in\Edir(G)$.
\end{proposition}

\begin{proof}
Suppose $\sg\in\tc(\x)$. Then, there is a sequence
\[
\x=\x^{(0)},\,\x^{(1)},\,\ldots,\,\x^{(s)}=\sg(\x)
\]
of oriented edges $\x^{(i)}\in\Edir(G)$ such that $\x^{(i-1)}$ and $\x^{(i)}$ differ on at most one coordinate for each $1\leq i\leq s$. Because $\chi$ is accordant, it holds that $\chi(\x^{(0)})=\chi(\x^{(1)})=\dots=\chi(\x^{(s)})$, i.e., that $\chi(\x)=\chi(\sg(\x))$. Because $\chi$ is $S_r$-equivariant, it follows that $\chi(\x)=\sg(\chi(\x))$, so $\sg\in\stab(\chi(\x))$.
\end{proof}

\begin{proposition}\label{prop:stab-implies-accordant}
Fix a uniformity $r$ and let $G$ be an $r$-graph. Let $\Ga_1,\ldots,\Ga_m$ be subgroups of $S_r$ such that, for each $\x\in\Edir(G)$, the group $\tc(\x)$ is contained in a conjugate $\sg\Ga_i\sg^{-1}$ of one of the groups $\Ga_i$.
Let $A=\bigsqcup_{i=1}^m(S_r/\Ga_i)$ be the set of left cosets of the subgroups $\Ga_i$ acted on by left multiplication. There is an accordant $A$-coloring $\chi$ of $\Edir(G)$.
\end{proposition}

\begin{proof}
Write $\x\sim\y$ if $\x$ is tightly connected to $\pi(\y)$ for some $\pi\in S_r$. This is equivalent to $\y$ being tightly connected to $\pi^{-1}(\x)$, so $\sim$ defines an equivalence relation on $\Edir(G)$. Partition $\Edir(G)$ into equivalence classes $\Fdir_1\cup\cdots\cup\Fdir_t$ under $\sim$ and fix an element $\x_i\in\Fdir_i$ for each $i$.
We define an accordant oriented coloring $\chi:\Edir(G)\to A$ separately on each equivalence class $\Fdir_i$. 

Fix an integer $j\in[t]$. It holds that $\tc(\x_j)\subseteq\sg_j\Ga_{i_j}\sg_j^{-1}$ for some index $i_j\in[m]$ and permutation $\sg_j\in S_r$.
For each $\y\in\Fdir_j$, choose some permutation $\pi$ such that $\x_j$ is tightly connected to $\pi(\y)$, and set $\chi(\y)=\pi^{-1}\sg_j\Ga_{i_j}$.
We claim that any choice of $\pi$ yields the same value $\chi(\y)$. Indeed, suppose $\x_j$ is tightly connected to both $\pi_1(\y)$ and $\pi_2(\y)$. Then $\pi_1(\y)=\pi_1\pi_2^{-1}(\pi_2(\y))$ is also tightly connected to $\pi_1\pi_2^{-1}(\x_j)$, so $\pi_1\pi_2^{-1}\in\tc(\x_j)\subseteq\sg_j\Ga_{i_j}\sg_j^{-1}$. It follows that $\pi_1\pi_2^{-1}$ stabilizes $\sg_j\Ga_{i_j}$ so $\pi_2^{-1}\sg_j\Ga_{i_j}=\pi_1^{-1}\pi_1\pi_2^{-1}\sg_j\Ga_{i_j}=\pi_1^{-1}\sg_j\Ga_{i_j}$ as desired.

To see that $\chi$ is $S_r$-equivariant, fix $\tau\in S_r$. If $\x_j$ is tightly connected to $\pi(\y)=\pi\tau^{-1}(\tau(\y))$, then
\[
\chi(\tau(\y))=\left(\pi\tau^{-1}\right)^{-1}\sg_j\Ga_{i_j}=\tau\pi^{-1}\sg_j\Ga_{i_j}=\tau(\chi(\y)),
\]
so $\chi$ is $S_r$-equivariant.

Lastly, we verify that $\chi$ is accordant. Suppose that $\y=y_1\cdots y_r$ and $\y'=y_1\cdots y_{i-1}y'_iy_{i+1}\cdots y_r$ are oriented edges differing in exactly one coordinate. If $\y\in\Fdir_j$ and $\x_j$ is tightly connected to $\pi(\y)$, then $\x_j$ is also tightly connected to $\pi(\y')$, as $\y$ and $\y'$ are tightly connected to each other. It follows that $\chi(\y')=\pi^{-1}\sg_j\Ga_{i_j}=\chi(\y)$.
\end{proof}

Combining \cref{prop:accordant-implies-stab,prop:stab-implies-accordant}, one obtains the corollary alluded to earlier.

\begin{corollary}\label{cor:accordant-stab}
Fix a uniformity $r$ and a permutation $\pi\in S_r$. Let $A_\pi$ be the set defined in \cref{sec:algebraic-oriented-colorings}. Given an $r$-graph $G$, there is an accordant $A_\pi$-coloring of $\Edir(G)$ if and only if $\tc(\x)$ avoids $\pi$-conjugates for each $\x\in\Edir(G)$.
\end{corollary}

\begin{proof}
Recall that $A_\pi=\bigsqcup_{i=1}^m(S_r/\Ga_i)$	, where the maximal $\pi$-conjugate avoiding subgroups are exactly the conjugates of the groups $\Ga_i$.

First, suppose there is an accordant $A_\pi$-coloring $\chi:\Edir(G)\to A_\pi$. For each $\x\in\Edir(G)$, we have $\chi(\x)=\sg\Ga_i$ for some $\sg\in S_r$ and $i\in[m]$. It follows that $\stab(\chi(\x))=\sg\Ga_i\sg^{-1}$ avoids $\pi$-conjugates, and by \cref{prop:accordant-implies-stab}, $\tc(\x)\subseteq\stab(\chi(\x))$ avoids $\pi$-conjugates as well.

Conversely, suppose $\tc(\x)$ avoids $\pi$-conjugates for each $\x\in\Edir(G)$. It follows that each group $\tc(\x)$ is contained in one of the maximal $\pi$-conjugate avoiding subgroups, i.e., that $\tc(\x)\subseteq\sg\Ga_i\sg^{-1}$ for some $\sg\in S_r$ and $i\in[m]$. By \cref{prop:stab-implies-accordant}, there is an accordant $A_\pi$-coloring of $\Edir(G)$.
\end{proof}

To complete the proof of \cref{thm:alg-provable}, we show that an $r$-graph $G$ contains a homomorphic copy of some cycle in $\C kr$ if and only if $\tc(\x)$ contains $\cyc^k$ for some $\x\in\Edir(G)$. To streamline the proof, we introduce some new notation. Given an $r$-graph $G$, a \emph{tight walk of stretch $\ell$} is a sequence $v_1\cdots v_{\ell+r}$ of $r+\ell$ vertices (not necessarily distinct) such that $v_{i+1}\cdots v_{i+r}\in\Edir(G)$ for each $0\leq i\leq\ell$; we say this is a walk \emph{from} $v_1\cdots v_r$ \emph{to} $v_{\ell+1}\cdots v_{\ell+r}$. We acknowledge that such a walk is standardly considered to have \emph{length} $\ell+1$, as there are $\ell+1$ indices $i$ for which $v_{i+1}\cdots v_{i+r}$ is an edge; the terminology \emph{stretch} is used to forestall potential confusion.

Observe that a homomorphic copy of the tight cycle $C_\ell^{(r)}$ in an $r$-graph $G$ is equivalent to a tight walk $v_1\cdots v_\ell v_1\cdots v_r$ of stretch $\ell$ from some $r$-tuple $v_1\cdots v_r$ to itself. The following proposition gives an equivalent condition for the existence of such a walk.

\begin{proposition}\label{prop:findoddcycle}
Let $G$ be an $r$-graph of any uniformity $r$. Fix a residue $k$ modulo $r$ and two oriented edges $\x,\y\in\Edir(G)$. There is a tight walk from $\x$ to $\y$ of stretch $k$ modulo $r$ if and only if $\x$ is tightly connected to $\cyc^k(\y)$, where $\cyc=(1\,2\,\ldots\,r)\in S_r$ is the cyclic shift whose action on $r$-tuples is given by $\cyc(x_1\cdots x_r)=x_rx_1\cdots x_{r-1}$.
\end{proposition}

\begin{proof}
Write $\y=y_1\cdots y_r$ and set $\y'=\cyc^k(\y)=y_{r-k+1}\cdots y_ry_1\cdots y_{r-k}$. Observe that there are tight walks $y_1\cdots y_ry_1\cdots y_{r-k}$ from $\y$ to $\y'$ and $y_{r-k+1}\cdots y_ry_1\cdots y_r$ from $\y'$ to $\y$ of stretches $r-k$ and $k$, respectively.

We claim that there is a tight walk from $\x$ to $\y$ of stretch $k$ modulo $r$ if and only if there is a tight walk from $\x$ to $\y'$ of stretch 0 mod $r$. Indeed, concatenating a walk from $\x$ to $\y$ with the walk from $\y$ to $\y'$, or conversely concatenating a walk from $\x$ to $\y'$ with the walk from $\y'$ to $\y$, lets us transform walks from $\x$ to $\y$ of stretch $k$ modulo $r$ into walks from $\x$ to $\y'$ of stretch 0 modulo $r$, and vice versa.
Thus, it suffices to show that $\x$ is tightly connected to $\y'$ if and only if there is a tight walk from $\x$ to $\y'$ of stretch 0 modulo $r$.

Suppose $\x$ is tightly connected to $\y'$. Then, there is a sequence
\[
\x=\z^{(0)},\,\z^{(1)},\,\ldots,\,\z^{(s)}=\y'
\]
of oriented edges $\z^{(i)}\in\Edir(G)$ such that $\z^{(i-1)}$ and $\z^{(i)}$ differ on at most one coordinate for each $1\leq i\leq s$. Observe that if $\z=z_1\cdots z_r$ and $\z'=z_1\cdots z_{j-1}z'_jz_{j+1}\cdots z_r$ are oriented edges differing on the $j$th coordinate, then the sequence $\z\z'=z_1\cdots z_rz_1\cdots z_{j-1}z'_jz_{j+1}\cdots z_r$ is a tight walk in $G$ --- the first $j$ length-$r$ subsequences of $\z\z'$ are cyclic shifts of $\z$, and the remaining $r-j+1$ length-$r$ subsequences are cyclic shifts of $\z'$. Thus, writing $\z^{(i)}=z_1^{(i)}\cdots z_r^{(i)}$ for each $i$, the sequence
\[
\z^{(0)}\z^{(1)}\cdots \z^{(s)}=z^{(0)}_1\cdots z^{(0)}_rz_1^{(1)}\cdots z_r^{(1)}\cdots z_1^{(s)}\cdots z_r^{(s)}
\]
is a tight walk of stretch $rs$ from $\x=\z^{(0)}$ to $\y'=\z^{(s)}$.

Conversely, suppose there is a tight walk from $\x$ to $\y'$ of stretch $\ell\equiv 0\pmod r$; we denote this walk by $v_1\cdots v_{\ell+r}$ where $\x=v_1\cdots v_r$ and $\y'=v_{\ell+1}\cdots v_{\ell+r}$. There is a sequence of oriented edges
\begin{align*}
\z^{(0)}&=v_1\cdots v_r,&\z^{(1)}&=v_{r+1}v_2\cdots v_r,&\ldots,&&
	\z^{(r-1)}&=v_{r+1}\cdots v_{2r-1}v_r,
\\\z^{(r)}&=v_{r+1}\cdots v_{2r},&\z^{(r+1)}&=v_{2r+1}v_{r+2}\cdots v_{2r},&\ldots,&&
	\z^{(2r-1)}&=v_{2r+1}\cdots v_{3r-1}v_{2r},
\\&\,\,\,\vdots&&\,\,\,\vdots &
&&&\,\,\,\vdots
\\\z^{(\ell-r)}&=v_{\ell-r+1}\cdots v_{\ell},&\z^{(\ell-r+1)}&=v_{\ell+1}v_{\ell-r+2}\cdots v_{\ell},&\ldots,&&
	\z^{(2r-1)}&=v_{\ell+1}\cdots v_{\ell+r-1}v_{\ell},
\\\z^{(\ell)}&=v_{\ell+1}\cdots v_{\ell+r},
\end{align*}
where $\z^{(i+1)}$ is obtained from $\z^{(i)}$ by replacing the vertex $v_{i+1}$ with $v_{i+r+1}$. Thus, $\x=\z^{(0)}$ is tightly connected to $\y'=\z^{(\ell)}$.
\end{proof}

Using \cref{prop:findoddcycle}, we may prove \cref{thm:alg-provable}.

\begin{proof}[Proof of \cref{thm:alg-provable}]
Let $G$ be an $r$-graph and let $\pi=\cyc^k$ as in the theorem statement. Recall that a homomorphic copy of $C_\ell^{(r)}$ is equivalent to a tight walk of stretch $\ell$ from some oriented edge $\x\in\Edir(G)$ to $\x$. $G$ is $\C kr$-hom-free if and only if, for each $\x\in\Edir(G)$, there is no tight walk from $\x$ to $\x$ of stretch $k$ modulo $r$. By \cref{prop:findoddcycle} this is equivalent to $\tc(\x)$ not containing $\pi$ for each $\x\in\Edir(G)$.

Observe that the conjugates of $\tc(\x)$ may be written as $\sg\tc(\x)\sg^{-1}=\tc(\sg(\x))$ for $\sg\in S_r$. Thus, $\tc(\x)$ avoids $\pi$-conjugates for a given $\x\in\Edir(G)$ if and only if $\pi\notin \tc(\sg(\x))$ for each $\sg\in S_r$. It follows that $G$ is $\C kr$-hom-free if and only if $\tc(\x)$ avoids $\pi$-conjugates for each $\x\in\Edir(G)$. By \cref{cor:accordant-stab}, this occurs if and only if there is an accordant $A_\pi$-coloring of $\Edir(G)$.
\end{proof}

\subsection{Applying \cref{thm:visual-applicable}}\label{sec:coloring-corollaries}

In this subsection, we give explicit corollaries of \cref{thm:visual-applicable} in uniformities 2, 3, and 4. Here and in the remainder of the paper, we rely on the pictorial perspective given in \cref{sec:pictorial-oriented-colorings}, as it would be needlessly cumbersome to phrase our results in terms of equivariant maps.

\begin{proof}[Proof of \cref{prop:bipartite} from \cref{thm:visual-applicable}]

We apply \cref{thm:visual-applicable} with $r=2$ and $k=1$. Set $\pi=\cyc^1=(1\,2)$. The symmetric group $S_2$ has only two subgroups --- $\{\id\}$ and $S_2$ --- and only $\{\id\}$ avoids $\pi$-conjugates. Thus, $\Delbf_\pi=\{\Delta_1\}$ where $\Delta_1$ is a pictogram depicting a coloring of the 2-vertex simplex (i.e., a line segment) with $\stab(\Del_1)=\{\id\}$. We may take $\Delta_1=\bipedge$.
	
\cref{thm:visual-applicable} states that a graph $G$ avoids all odd closed walks --- i.e., $G$ is $\C 12$-hom-free --- if and only if $E(G)$ admits an accordant oriented coloring by $\Delbf_\pi=\{\bipedge\}$. An oriented coloring of $E(G)$ is accordant if, for any two incident edges $xy,xz\in E(G)$, the colorings of $xy$ and $xz$ agree at the vertex $x$. In other words, every vertex $x\in V(G)$ is either a source or a sink; by coloring sources blue and sinks red as in \Cref{fig:coloring-accordant-2uniform}, we observe that this carries the same data as a proper 2-coloring of $V(G)$.
It follows that $G$ avoids all odd closed walks if and only if $V(G)$ is 2-colorable, i.e., $G$ is bipartite.
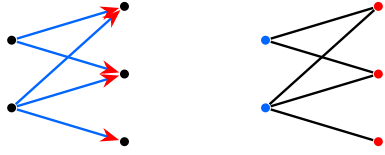
\begin{figure}
\begin{tikzpicture}[yscale=0.9,xscale=0.75,tikz arrows]
\coordinate[vtx] (a1) at (0,1);
\coordinate[vtx] (a2) at (0,0);
\coordinate[vtx] (b1) at (2,1.5);
\coordinate[vtx] (b2) at (2,0.5);
\coordinate[vtx] (b3) at (2,-0.5);
\foreach \i in {1,2}
	\path (a\i) --++ (4.5,0) coordinate[graphblue,vtx] (c\i);
\foreach \i in {1,2,3}
	\path (b\i) --++ (4.5,0) coordinate[graphred,vtx] (d\i);
\foreach \i/\j in {1/1,1/2,2/1,2/2,2/3} {
	\draw[bipartite edge,->] (a\i) -- (b\j);
	\draw (c\i) -- (d\j);
}
\end{tikzpicture}
\centering
\caption{At left, an accordant oriented coloring of the edges of a bipartite graph by $\{\bipedge\}$. At right, the corresponding 2-coloring of the vertices.}
\label{fig:coloring-accordant-2uniform}
\end{figure}
\end{proof}

Next, we derive a result in uniformity 3 due to Kam\v cev, Letzter, and Pokrovskiy \cite{KLP24}. Before deriving this result, we observe that $\C 13$-hom-freeness and $\C 23$-hom-freeness are equivalent conditions; indeed, a more general result holds.

\begin{proposition}\label{prop:homfree-equiv}
Fix an integer $r$ and two residues $s,k$ modulo $r$. Any $\C{sk}r$-hom-free $r$-graph is also $\C kr$-hom-free.
\end{proposition}

\begin{proof}
Suppose $G$ is $\C{ak}r$-hom-free.
For each $\ell\equiv k\pmod r$, it holds that $G$ is $C^{(r)}_{s\ell}$-hom-free, and \cref{prop:hom-free-chain} implies that $G$ is $C_\ell^{(r)}$-hom-free.
\end{proof}

\begin{remark}\label{remark:homfree-equiv}
In the case that $a$ is a multiplicative unit modulo $r$, it is not difficult to check that the permutations $\pi=\cyc^k$ and $\pi'=\cyc^{ak}$ are conjugates in $S_r$. It follows that the families $\mathfrak G_\pi$ and $\mathfrak G_{\pi'}$ of subgroups of $S_r$ avoiding $\pi$-conjugates and $\pi'$-conjugates are identical, and thus we may take $\Delbf_\pi=\Delbf_{\pi'}$. Thus, the characterization in \cref{thm:visual-applicable} gives an alternate proof that $C^{(r)}_k$-hom-freeness is equivalent to $C_{ak}^{(r)}$-hom-freeness whenever $a$ is a multiplicative unit modulo $r$.
\end{remark}

We now derive a 3-uniform corollary of \cref{thm:visual-applicable}. The equivalence between (1), (2), and (4) below is originally due to Kam\v cev, Letzter, and Pokrovskiy \cite{KLP24}; however, their proof is of a very different style and requires several cases. To streamline this statement and our 4-uniform results, we introduce one further piece of notation. Given a set of vertices $V$ and an integer $r\geq 2$, let $V^{(r)}$ denote the $S_r$-set $\Edir(K_V^{(r)})$, where $K_V^{(r)}$ denotes the complete $r$-graph on the vertex set $V$. In this section and the next, we shall frequently work with non-accordant oriented colorings of $V^{(r)}$ that satisfy other consistency conditions.

\begin{corollary}
Let $G$ be a 3-graph with vertex set $V=V(G)$. The following are equivalent.
\begin{enumerate}
	\item $G$ is $\C13$-hom-free.
	\hfil (2) $G$ is $\C23$-hom-free.
	\stepcounter{enumi}
	\item There is an accordant oriented coloring of $E(G)$ by the set $\Delbf=\left\{\pointedtriangle\right\}$.
	\item There is an oriented coloring of $V^{(2)}$ by the set $\Delbf'=\{\bipedge,\blueedge\}$ such that,
	for any 3-edge $xyz\in E(G)$, the coloring restricted to $\{x,y,z\}$ is isomorphic to $\dirboundarylabeled{}{}{}$.
\end{enumerate}
\end{corollary}

\begin{proof}
We have (1)$\iff$(2) by either \cref{prop:homfree-equiv} or \cref{remark:homfree-equiv}.

To show (1)$\iff$(3), we apply \cref{thm:visual-applicable} with $r=3$ and $k=1$. Let $\pi=\cyc^1=(1\,2\,3)$. The symmetric group $S_3$ has four conjugacy classes of subgroups: $\{\id\}$, $S_2$, $A_3$, and $S_3$. No conjugate of $\{\id\}$ or $S_2$ contains $\pi$; however, both $A_3$ and $S_3$ contain all 3-cycles. It follows that the conjugates of $S_2$ are the only maximal subgroups of $S_3$ avoiding $\pi$-conjugates. Thus, $\Delbf_\pi=\{\Del_1\}$, where $\Del_1$ depicts a coloring of the 3-vertex simplex (i.e., the triangle) stabilized by some conjugate of $S_2$ in $S_3$. We may take $\Del_1=\pointedtriangle$; its stabilizer is generated by the symmetry which swaps the bottom two vertices (i.e., reflects across the vertical axis). We conclude by \cref{thm:visual-applicable} that $G$ is $\C13$-hom-free if and only if there is an accordant oriented coloring of $E(G)$ by the set $\Delbf_\pi=\left\{\pointedtriangle\right\}$, i.e., that (1)$\iff$(3).

Lastly, we relate (3) and (4). Each accordant coloring of $E(G)$ by $\Delbf_\pi$ induces an oriented coloring on $V^{(2)}$: if a 3-edge $xyz\in E(G)$ is colored as $\pointedtrianglelabeled xyz$, then we color the three corresponding 2-edges $xy$, $yz$, and $xz$ as $\dirboundarylabeled{x}{y}{z}$, and if $xy\in V^{(2)}$ is contained in no 3-edge $xyz\in E(G)$, then we color $xy$ arbitrarily. The coloring of $V^{(2)}$ is well-defined because the coloring of $E(G)$ is accordant: if $xyz,xyz'\in E(G)$ are two edges containing the pair $xy$ then the edges $xyz$ and $xyz'$ will induce the same coloring of $xy$ in $V^{(2)}$. Thus, (3)$\implies$(4).

Conversely, suppose there is a coloring of $V^{(2)}$ by $\Delbf'$ which restricts to a subgraph isomorphic to $\dirboundarylabeled{}{}{}$ on each edge of $G$. We may recover the associated oriented coloring of $E(G)$ by $\Delbf$ by coloring an edge $xyz\in E(G)$ as $\pointedtrianglelabeled xyz$ if the coloring of $V^{(2)}$ restricted to those three vertices is $\dirboundarylabeled{x}{y}{z}$. It is immediate that this is an accordant coloring of $E(G)$, so (4)$\implies$(3).
\end{proof}

Lastly, we provide two corollaries of \cref{thm:visual-applicable} in uniformity 4 that were not previously known. We handle the $k\equiv 2\pmod 4$ and $k\equiv 1,3\pmod 4$ cases are separately. 

Our first 4-uniform result proves \cref{thm:coloring-2mod4}. Let $\pointedtetrahedron$ denote a 4-vertex pictogram stabilized by any symmetry that fixes the top vertex. This tetrahedron has three faces colored with $\pointedtriangle$ and one face colored with $\bluetriangle$.

\begin{corollary}\label{cor:C24-free}
Let $G$ be a 4-graph with vertex set $V=V(G)$. The following are equivalent.
\begin{enumerate}
	\item $G$ is $\C24$-hom-free.
	\item There is an accordant oriented coloring of $E(G)$ by the set $\Delbf=\left\{\pointedtetrahedron\right\}$.
	\item There is an oriented coloring of $V^{(3)}$ by the set $\Delbf'=\left\{\pointedtriangle,\bluetriangle\right\}$ such that, given any edge $wxyz\in E(G)$, the coloring of $V^{(3)}$ restricted to $\{w,x,y,z\}$ is isomorphic to $\pointedtetrahedronboundary$. In other words, three of the four 3-edges $xwy,xwz,wyz,xyz$ are colored as $\pointedtriangle$, with the red vertex placed at the same vertex in all three 3-edges, and the last 3-edge is colored with $\bluetriangle$.
\end{enumerate}
\end{corollary}

\begin{proof}
To show (1)$\iff$(2), we apply \cref{thm:visual-applicable} with $r=4$ and $k=2$. Let $\pi=\cyc^2=(1\,3)(2\,4)$. There are 11 conjugacy classes of subgroups of $S_4$, of which four ($S_3$ and its three subgroups) avoid $\pi$-conjugates. (A full list of these 11 conjugacy classes is given in \cref{appendix:subgroups}.) Thus, $\Delbf_\pi=\{\Del_1\}$ where $\Del_1$ is a pictogram depicting a coloring of the 4-vertex simplex (i.e., a tetrahedron) stabilized by some conjugate of $S_3$ in $S_4$. We may take $\Del_1=\pointedtetrahedron$; the stabilizer of this pictogram is the group of permutations of the bottom 3 vertices. We conclude by \cref{thm:visual-applicable} that $G$ is $\C24$-hom-free if and only if there is an accordant oriented coloring of $E(G)$ by the set $\Delbf_\pi=\left\{\pointedtetrahedron\right\}$, which implies (1)$\iff$(2).

To show (2)$\implies$(3), fix an accordant oriented coloring of $E(G)$ by $\Delbf_\pi$. There is an associated oriented coloring of $V^{(3)}$ by $\Delbf'=\left\{\pointedtriangle,\bluetriangle\right\}$ which satisfies the following rule on each triple of vertices $xyz\in V^{(3)}$ contained in some edge $wxyz\in E(G)$.
The coloring of $xyz$ in $V^{(3)}$ is 
$\pointedtrianglelabeled xyz$ if every edge $wxyz\in E(G)$ containing $xyz$ is colored as $\pointedtetrahedronlabeled xywz$ and is 
$\bluetrianglelabeled xyz$ if every $wxyz\in E(G)$ is colored as $\pointedtetrahedrontwistedlabeled xywz$. 
This is well defined because the coloring of $G$ is accordant.
Additionally, it is immediate that, given any edge $wxyz\in E(G)$, the coloring of $V^{(3)}$ restricted to the vertex set $\{w,x,y,z\}$ is isomorphic to $\pointedtetrahedronboundary$, which is exactly the condition required by (3).

Lastly, we show (3)$\implies$(2). Suppose there is an oriented coloring of $V^{(3)}$ by $\Delbf'$ which, when restricted to the vertex set of any 4-edge of $G$, is isomorphic to $\pointedtetrahedronboundary$. We may recover an oriented coloring of $E(G)$ by $\Delbf_\pi$ in the ``obvious'' fashion --- an edge $wxyz\in E(G)$ is colored as $\pointedtetrahedronlabeled wxyz$ if and only if the corresponding four 3-edges of $K_V^{(3)}$ are colored as $\pointedtrianglelabeled wxy$, $\pointedtrianglelabeled wxz$, $\pointedtrianglelabeled wyz$, and $\bluetrianglelabeled{x}{y}{z}$. It is immediate that such an oriented coloring of $E(G)$ is accordant.
\end{proof}

Our last result characterizes $\C14$-hom-free and $\C34$-hom-free 4-graphs $G$. Once again, we have a characterizations in terms of accordant oriented colorings of $E(G)$ and an equivalent characterization in terms of oriented colorings of $V^{(3)}$ in which the four sub-triples of any edge $wxyz\in E(G)$ are colored ``consistently''. This result requires three 4-vertex pictograms. Let $\pointedtetrahedron$ be as before. Let $\edgetetrahedron$ be a pictogram that is stabilized by transpositions of the top and left vertices and of the bottom and right vertices; that is, $\stab(\edgetetrahedron)$ when viewed as a subgroup of $S_4$ belongs to the conjugacy class of the non-normal Klein four-subgroup $\langle(1\,2),(3\,4)\rangle$. Lastly, we require a 4-vertex pictogram which is stabilized by rotations but not reflections, i.e., which has stabilizer $A_4$. We render this pictogram as $\circletetrahedron$; it has an orientation $\circletriangle$ or $\circletrianglerev$ placed on each of the four faces of the tetrahedron, so that adjacent faces are always consistently oriented as $\circlediamond$ (or equivalently $\circlediamondrev$). The orientations on the back two faces are not pictured due to space limitations.

\begin{corollary}\label{cor:C14-free}
Let $G$ be a 4-graph with vertex set $V=V(G)$. The following are equivalent.
\begin{enumerate}
	\item $G$ is $\C14$-hom-free.
	\hfil (2) $G$ is $\C34$-hom-free.
	\stepcounter{enumi}
	\item There is an accordant oriented coloring of $E(G)$ by the set $\Delbf=\left\{\pointedtetrahedron,\circletetrahedron,\edgetetrahedron\right\}$.
	\item There is an oriented coloring of $V^{(3)}$ by the set $\Delbf'=\left\{\pointedtriangle,\bluetriangle,\circletriangle,\rededgetriangle,\yellowedgetriangle\right\}$ such that, given any edge $wxyz\in E(G)$, the coloring of $V^{(3)}$ restricted to $\{w,x,y,z\}$ is isomorphic to one of three oriented colorings of $E(K_4^{(3)})$ by $\Delbf'$ given in \Cref{fig:tetrahedron-boundaries}.
	\begin{figure}[h]
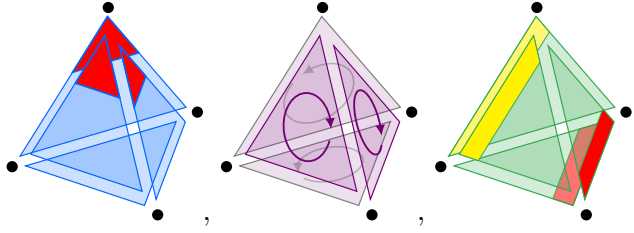

	\centering
	\pointedtetrahedronboundary[big],\circletetrahedronboundary,\edgetetrahedronboundary
	\caption{The three colorings of the four sub-triples of a 4-edge permitted by \cref{cor:C14-free}(4).}
	\label{fig:tetrahedron-boundaries}
	\end{figure}
\end{enumerate}
\end{corollary}

\begin{proof}
We have (1)$\iff$(2) by either \cref{prop:homfree-equiv} or \cref{remark:homfree-equiv}.

To show (1)$\iff$(3), we apply \cref{thm:visual-applicable} with $r=4$ and $k=1$. Let $\pi=\cyc^1=(1\,2\,3\,4)$. There are 11 conjugacy classes of subgroups of $S_4$; a complete list of these subgroups is given in \cref{appendix:subgroups}. The subgroups avoiding $\pi$-conjugates are the conjugates of $S_3$, $A_4$, the non-normal Klein four-subgroup $\Ga=\langle (1\,2),(3\,4)\rangle$, and their subgroups. Thus, $\Delbf_\pi=\{\Del_1,\Del_2,\Del_3\}$ where the pictograms $\Del_1$, $\Del_2$, and $\Del_3$ are stabilized respectively by some conjugate of $S_3$, $A_4$, and $\Ga$. We may take $\Del_1=\pointedtetrahedron$, $\Del_2=\circletetrahedron$, and $\Del_3=\edgetetrahedron$, yielding $\Delbf_\pi=\left\{\pointedtetrahedron,\circletetrahedron,\edgetetrahedron\right\}$. We conclude by \cref{thm:visual-applicable} that $G$ is $\C14$-hom-free if and only if there is an accordant coloring of $E(G)$ by the set $\left\{\pointedtetrahedron,\circletetrahedron,\edgetetrahedron\right\}$, which implies (1)$\iff$(3).

Lastly, we show (3)$\iff$(4). Fix an accordant oriented coloring of $E(G)$ by $\Delbf$. Because the coloring of $E(G)$ is accordant, there is an associated oriented coloring of $V^{(3)}$ by $\Delbf'$ which satisfies the following rule on each triple of vertices $xyz\in V^{(3)}$ contained in some edge $wxyz\in E(G)$:
\[xyz\text{ is colored as}
\begin{cases}
\pointedtrianglelabeled xyz
	\text{if every edge $wxyz\in E(G)$ containing $xyz$ is colored as}
	\pointedtetrahedronlabeled xyzw,
\\\bluetrianglelabeled xyz
	\text{if every edge $wxyz\in E(G)$ containing $xyz$ is colored as}
	\pointedtetrahedrontwistedlabeled xywz,
\\\circletrianglelabeled xyz
	\text{if every edge $wxyz\in E(G)$ containing $xyz$ is colored as}
	\circletetrahedronlabeled xyzw,
\\\rededgetrianglelabeled xyz
	\text{if every edge $wxyz\in E(G)$ containing $xyz$ is colored as}
	\edgetetrahedronlabeled xwyz,
\\\yellowedgetrianglelabeled xyz
	\text{if every edge $wxyz\in E(G)$ containing $xyz$ is colored as}
	\edgetetrahedronflippedlabeled xwyz.
\end{cases}\]
For triples $xyz\in V^{(3)}$ not contained in any edge $wxyz\in E(G)$, we color $xyz$ arbitrarily. It is immediate that this coloring of $V^{(3)}$ satisfies (4), so (3)$\implies$(4).

Conversely, suppose there is an oriented coloring of $V^{(3)}$ by $\Delbf'$ satisfying the condition in (4). There is a natural oriented coloring of $E(G)$ by $\Delbf$ such that the above rule is satisfied. This coloring is accordant because the color of any $xyz\in V^{(3)}$ uniquely determines the colors of all edges $wxyz\in E(G)$ containing $xyz$. Thus, (4)$\implies$(3).
\end{proof}

Using \cref{cor:C24-free,cor:C14-free} yields an alternate proof that the complete oddly bipartite construction is $\C k4$-hom-free for $k=1,2,3$.

\begin{proposition}\label{prop:Godd-C-free}
Given any two disjoint sets of vertices $V_1$ and $V_2$, the complete oddly bipartite $4$-graph $\Godd(V_1,V_2)$ is $\C k4$-hom-free for $k=1,2,3$. In particular, $\ex(n,C_L^{(4)})\geq\eopt(n)$ for each $n$ and each $L\not\equiv 0\pmod 4$.
\end{proposition}

\begin{proof}
Let $V=V_1\cup V_2$. Consider the following oriented coloring of $V^{(3)}$ by $\Delbf'=\left\{\pointedtriangle,\bluetriangle\right\}$, which is pictured in \Cref{fig:Godd-coloring} at left.
If $x\in V_i$ and $y,z\in V_{3-i}$, color the triple $xyz$ as $\pointedtrianglelabeled xyz$. Else, if $x,y,z$ are in the same part $V_i$, then color $xyz$ as $\bluetrianglelabeled xyz$. It is clear that this coloring satisfies \cref{cor:C24-free}(3) and \cref{cor:C14-free}(4), so $\Godd(V_1,V_2)$ is $\C k4$-hom-free for $k=1,2,3$ by \cref{cor:C24-free}(1) and \cref{cor:C14-free}(1)--(2).
\begin{figure}
\centering
\begin{tikzpicture}[scale=0.8]
	\coordinate (o1) at (0,0);
	\coordinate (o2) at (3,0);
	\draw[yscale=1.3] (o1) circle (1) (o2) circle (1);
	\path (o1) --++(-1,0) node[left]{$V_1$};
	\path (o2) --++(1,0) node[right]{$V_2$};
	\path (o1) --+ (-0.8,0) coordinate(x1)
			--+ (0.1,0.8) coordinate (y1)
			--+ (0.1,-0.8) coordinate (z1);
	\path (o2) --+ (0.8,0) coordinate(x2)
			--+ (-0.1,0.8) coordinate (y2)
			--+ (-0.1,-0.8) coordinate (z2);
	\path (y1) --++ (0.2,0) 
			--++ (0,-0.2) coordinate (a1) --++ (0,-0.2) coordinate (a3)
			(z1) --+ (0.2,0) coordinate (a2)
			(y2) --+ (-0.2,0) coordinate (b1)
			(z2) --++ (-0.2,0) 
			--++ (0,0.2) coordinate (b2) --++ (0,0.2) coordinate (b3);

	\tikzmakebluetriangle {x1}{y1}{z1};
	\tikzmakebluetriangle {x2}{y2}{z2};
	\tikzmakepointedtriangle {a1}{b1}{b3};
	\tikzmakepointedtriangle {b2}{a3}{a2};
\end{tikzpicture}
\qquad
\begin{tikzpicture}[scale=0.8]
	\coordinate (o1) at (0,0);
	\coordinate (o2) at (3,0);
	\draw[yscale=1.3] (o1) circle (1) (o2) circle (1);
	\path (o1) --++(-1,0) node[left] {$V_1$};
	\path (o2) --++(1,0) node[right] {$V_2$};
		
	\path (o1) --+ (-0.8,0) coordinate(x1)
			--+ (0.1,0.6) coordinate (y1)
			--+ (0.1,-0.8) coordinate (z1)
			--+ (0.1,0.8) coordinate (w2);
	\path (o2) --+ (0.8,0) coordinate(x2)
			--+ (-0.1,0.8) coordinate (y2)
			--+ (-0.1,-0.6) coordinate (z2)
			--+ (-0.1,-0.8) coordinate (w1);
	\tikzmakepointedtetrahedron{w1}{y1}{x1}{z1}
	\tikzmakepointedtetrahedron{w2}{y2}{x2}{z2}
\end{tikzpicture}
\caption{At left, a coloring of $(V_1\cup V_2)^{(3)}$ that proves $\Godd(V_1,V_2)$ is $\C k4$-hom-free. At right, the corresponding accordant oriented coloring of $E(\Godd(V_1,V_2))$.}
\label{fig:Godd-coloring}
\end{figure}
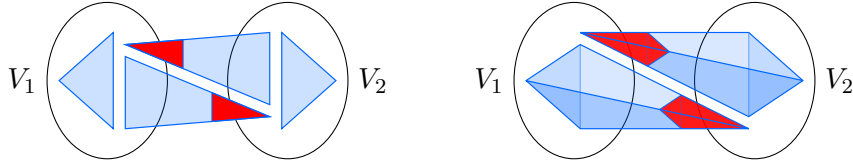
\end{proof}

\section{Removing Long Tight Cycles}\label{sec:delete-long-cycles}

In this section, we describe how to handle Step 3 of the framework described in \cref{sec:framework} in full generality.

\begin{proposition}\label{prop:delete-long-cycles}
Let $G$ be a $n$-vertex $r$-graph. For any $\eps>0$, we may delete at most $\eps n^r$ edges from $G$ such that the resulting subgraph $G'$ has the following property. For any $\x,\y\in\Edir(G')$ that are tightly connected in $G'$, there is a tight walk from $\x$ to $\y$ in $G'$ of stretch $sr$ for some integer $s\leq(2r+1)\eps^{-r}$.
\end{proposition}

\begin{proof}
Iteratively delete edges from $G$ according to the following process: whenever an $(r-1)$-tuple $x_1\dots x_{r-1}$ is contained in at most $\eps n$ edges, delete all edges containing $x_1\dots x_{r-1}$ from $G$. Let $G'$ be the resulting $r$-graph. At most $\eps n\times\binom n{r-1}<\eps n^r$ edges are deleted during this process. Moreover, $\deg_{G'}(x_1\dots x_{r-1})\geq \eps n$ for any $(r-1)$-tuple of vertices $x_1\dots x_{r-1}$ contained in an edge of $G'$.

Suppose $\x,\y\in\Edir(G')$ are tightly connected in $G'$. Let
$\x=\z^{(0)},\z^{(1)},\ldots,\z^{(s)}=\y$
be a shortest sequence of oriented edges tightly connecting $\x$ and $\y$. 
For each $0\leq i\leq s$, set
\[Z^{(i)}=\{\z'\in\Edir(G'):\z^{(i)}\z'=z^{(i)}_1\cdots z^{(i)}_rz'_1\cdots z'_r\text{ is a tight walk in $G'$}\}.\]
We have $|Z^{(i)}|\geq(\eps n)^r$ as there are at least $\eps n$ choices for each vertex $z'_i$ given any choice of $z'_1,\dots,z'_{i-1}$.

We claim that $Z^{(i)}$ and $Z^{(j)}$ are disjoint if $j>i+2r$. Otherwise, given $\z'\in Z^{(i)}\cap Z^{(j)}$, the sequence
\begin{align*}
\z^{(i)}&=z^{(i)}_1\cdots z^{(i)}_r,
\,\,\, z'_1z^{(i)}_2\cdots z^{(i)}_r,
\,\,\, \ldots,
\,\,\, z'_1\cdots z'_{r-1} z^{(i)}_r,
\,\,\, \z'=z'_1\cdots z'_r,
\\&\,z'_1\cdots z'_{r-1}z^{(j)}_r,
\,\,\, \ldots,
\,\,\, z'_1z_2^{(j)}\cdots z_r^{(j)},
\,\,\, \z^{(j)}=z^{(j)}_1,\ldots,z^{(j)}_r
\end{align*}
from $\z^{(i)}$ to $\z^{(j)}$ could shorten the sequence $\z^{(0)},\ldots,\z^{(s)}$. In particular, the sets $Z^{(0)},Z^{(2r+1)},Z^{(4r+2)},\ldots$ are pairwise disjoint. Thus,
\[
\frac s{2r+1}(\eps n)^r\leq \sum_{i=0}^{\lfloor s/(2r+1)\rfloor}|Z^{(i)}|\leq\Edir(G')< n^r,
\]
and it follows that $s\leq(2r+1)\eps^{-r}$. The proof is completed by observing that
\[
\z^{(0)}\z^{(1)}\cdots\z^{(s)}=z^{(0)}_1\cdots z^{(0)}_rz^{(1)}_1\cdots z^{(1)}_r\cdots z_1^{(s)}\cdots z_r^{(s)}
\]
forms a tight walk from $\x$ to $\y$ of stretch $sr$, as discussed in the proof of \cref{prop:findoddcycle}.
\end{proof}

\begin{corollary}\label{delete-long-cycles}
Let $G$ be a $n$-vertex $r$-graph.
Fix an integer $L>r$, and let $0\leq k<r$ be the residue of $L$ modulo $r$.
If $G$ is $C_L^{(r)}$-hom-free, then we may delete at most $2rL^{-1/r}n^r$ edges from $G$ such that the resulting $r$-graph $G'$ is $\C kr$-hom-free.
\end{corollary}

\begin{proof}
Let $\eps=\big(\frac {r(2r+1)}{L-k}\big)^{1/r}$ and let $G'$ be the subgraph of $G$ guaranteed by \cref{prop:delete-long-cycles}. $G'$ is formed by deleting at most $\eps n^r=\big(\frac{r(2r+1)L}{L-k}\big)^{1/r}L^{-1/r}n^r\leq 2rL^{-1/r}n^r$ edges of $G$.

Suppose for the sake of contradiction that $G'$ is not $\C kr$-hom-free, i.e., $G'$ contains a tight walk from some oriented edge $\x=x_1\cdots x_r$ to itself of stretch $k$ modulo $r$. By \cref{prop:findoddcycle}, $\x$ is tightly connected to $\cyc^k(\x)=x_{r-k+1}\cdots x_rx_1\cdots x_{r-k}$. By \cref{prop:delete-long-cycles}, there is a tight walk $z_1\cdots z_{(s+1)r}$ from $\x$ to $\cyc^k(\x)$ of stretch $sr$ with $s\leq (2r+1)\eps^{-r}=(L-k)/r$. Concatenating this walk with the $k$ vertices $x_{r-k+1}\cdots x_r$ on the right, followed by $(L-k-rs)/r$ copies of $\x$, yields a tight walk
\[
(z_1\cdots z_{(s+1)r})(x_{r-k+1}\cdots x_r)(x_1\cdots x_r)\cdots (x_1\cdots x_r)
\]
in $G'$ from $\x$ to $\x$ of stretch exactly $L$, i.e., a homomorphic copy of $C_L^{(r)}$. This contradicts the assumption that $G$ is $C_L^{(r)}$-hom-free.
\end{proof}

\section{The Tur\'an Density of $\C k4$, with Stability}\label{sec:density-stability}

By \cref{cor:C14-free}, the Tur\'an number $\ex(n,\C14\texthom)$ may be reformulated as follows. Let $V$ be a set of $n$ vertices. Then $\ex(n,\C14\texthom)$ is the maximum, over all oriented colorings $\chi$ of $V^{(3)}$ by $\Delbf_3=\left\{\pointedtriangle,\bluetriangle,\circletriangle,\rededgetriangle,\yellowedgetriangle\right\}$, of the number of sets of four vertices whose coloring matches one of the three configurations in \Cref{fig:tetrahedron-boundaries}.

Let $w\in V$. Given an oriented coloring $\chi$ of $V^{(3)}$ by $\Delbf_3$ (i.e., an $S_3$-equivariant map $\chi:V^{(3)}\to A_{\Delbf_3}$), one may define a \emph{link coloring} $\chi_w$ of $(V-\{w\})^{(2)}$, where $\chi_w(xy)$ is exactly determined by $\chi(wxy)$. The easiest way to define $\chi_w(xy)$ symbolically would be to treat $\chi_w$ as a coloring by $A_{\Delbf_3}$, viewed as an $S_2$-set, in which case $\chi_w(xy)$ would be defined equal to $\chi(xyw)$. However, for ease of presentation, we define $\chi_w$ to be an oriented coloring of $(V-\{w\})^{(2)}$ by the $S_2$-set
\[\Delbf_2=\{
\redto[e1],\,\blueedge[e1],\,\redto[e2],\,\blueedge[e2],\,\redto[e3],\,\blueedge[e3],\,
\greenedge,\,\purpleto
\}
\tag{\thesection.1}
\label{eq:del2}
\]
with $\chi_w$ defined in terms of $\chi$ according to \cref{table:link-coloring}.
\begin{table}[h]\centering
\begin{tabular}{|l|cccccccc|}
\hline
$\chi(wxy)$
&\pointedtrianglelabeled {w\vphantom{L}}xy
&\pointedtrianglerightlabeled wxy
&\bluetrianglelabeled wxy
&\circletrianglelabeled wxy
&\rededgetrianglelabeled wxy
&\rededgetrianglerightlabeled wxy
&\yellowedgetrianglelabeled wxy
&\yellowedgetrianglerightlabeled wxy
\\[10pt]$\chi_w(xy)$
&\greenedgexy
&\redtoxy[e1]
&\blueedgexy[e1]
&\purpletoxy
&\blueedgexy[e2]
&\redtoxy[e3]
&\blueedgexy[e3]
&\redtoxy[e2]
\\[3pt]\hline
\end{tabular}
\caption{Colorings of $wxy$ by $\chi$ and the corresponding colorings of $xy$ by $\chi_w$.}
\label{table:link-coloring}
\vspace{-0.1in}
\end{table}

Observe that the coloring of $wxyz$ matches one of the configurations in \Cref{fig:tetrahedron-boundaries} if and only if (after permutation of $xyz$) the colorings of $xyz$ by $\chi$ and $\chi_w$ are given by one of the following pairs:
\[
\begin{array}{c}
\left(\bluetrianglelabeled xyz,\greentrilabeled xyz\right),
\left(\pointedtrianglelabeled xyz,\cherrylabeled[e1]xyz\right),
\left(\circletrianglelabeled xyz,\purpletrilabeled xyz\right),
\\
\left(\rededgetrianglelabeled xyz,\cherrylabeled[e2]xyz\right),
\text{ or }
\left(\yellowedgetrianglelabeled xyz,\cherrylabeled[e3]xyz\right).
\end{array}
\tag{\thesection.2}
\label{eq:chi-and-link-pairs}
\]
We begin by studying the number of triples colored \greentri,\purpletri,\cherry[e1],\cherry[e2], or\cherry[e3] by $\chi_w$.

If one does not distinguish the different types of red or blue edges, one obtains a \emph{simplified link coloring} $\chi'_w$ of $(V-\{w\})^{(2)}$ by 
$\Delbf'_2=\{\redto,\,\blueedge,\,\greenedge,\,\purpleto\}$.  
In \cref{subsection:link}, we prove \cref{lem:color-ineqs}, which provides bounds on the number of copies of \purpletri, \greentri, and \cherry\ in such a coloring, in terms of the density of each edge color. In \cref{subsec:turan-density}, we leverage \cref{lem:color-ineqs} to prove that $\pi(\C14)=\pi(\C34)=1/2$. Lastly, in \cref{subsection:stability}, we use the results of the first two subsections to prove \cref{thm:stability}, a stability result stating that if $G$ is $\C14$-hom-free with edge density close to $1/2$, then $G$ differs from a complete oddly bipartite 4-graph on a small fraction of edges.

\subsection{Counting triangles in a link coloring} 
\label{subsection:link}

Let $V$ be a set of vertices and consider an oriented coloring of $V^{(2)}$ by the set $\Delbf'_2=\{\redto,\,\blueedge,$ $\greenedge,\,\purpleto\}$.
This should be thought of as the simplified link coloring $\chi'_w$, arising from an oriented coloring $\chi$ of a set $(V\cup \{w\})^{(3)}$ by the set $\Delbf_3$.
Let $T(\greentri)$, $T(\purpletri)$, and $T(\cherry)$ denote the number of (unlabeled) copies of each of these three types of triangles in the oriented coloring. We refer to triangles of the third type as \emph{cherries}, in keeping with the language used in \cite{KLP24}. The key lemma of this subsection bounds these triangle counts in terms of the density of each edge color.

\begin{lemma}\label{lem:color-ineqs}
Let $V$ be a set of $n$ vertices. Fix an oriented coloring of $V^{(2)}$ by the set
$\Delbf'_2=\{
\redto,\,
\blueedge,$
$\greenedge,\,
\purpleto
\}$ with $2\al\binom n2$ red directed edges, $\bet\binom n2$ blue undirected edges, $\ga\binom n2$ green undirected edges, and $2\del\binom n2$ purple directed edges. Then the numbers of triangles of each type satisfy the following inequalities.
\begin{enumerate}
	\item $T(\greentri)\leq \ga^{3/2}\times n^3/6$.
	\item $T(\purpletri)\leq 2\del^{3/2}\times n^3/6$.
	\item $T(\cherry)\leq 3\al\sqrt\bet \times n^3/6$.
	\item $T(\cherry)\leq 0.465\times n^3/6$.
	\item $T(\purpletri) + \frac 14\left(T(\greentri) + T(\cherry)\right)\leq 1/4\times n^3/6$
	\item $T(\cherry)\leq 3\al\bet/(\al+\bet)\times n^3/6$.
\end{enumerate}
\end{lemma}

Parts (1)--(3) follow from standard results about entropy (\cref{prop:entropy-bits,shearer}). For proofs of these results and further exposition, we refer the reader to \cite[\S 15]{Alon-PM}. Given a discrete random variable $X$ supported on a finite set $S$, its \emph{entropy} is the quantity
\[
\ent(X)=\sum_{x\in S}\Pr[X=s]\times\log_2\left(\frac 1{\Pr[X=s]}\right).
\]
Intuitively, the entropy of $X$ measures the number of bits of information revealed by observing the value of $X$. In particular, if $X$ is distributed uniformly randomly on $S$, then $\ent(X)=\log_2|S|$, and this is known to be the largest possible entropy of any distribution on $S$.

\begin{proposition}\label{prop:entropy-bits}
	Let $X$ be a random variable supported on a finite set $S$. Then $\ent(X)\leq\log_2|S|$, with equality if and only if $X$ is uniformly distributed on $S$.
\end{proposition}

Given random variables $X_1,\ldots,X_j$ supported on finite sets $S_1,\ldots,S_j$, write $\ent(X_1,\ldots,X_j)$ for the entropy of the random variable $(X_1,\ldots,X_j)$, which is supported on the set $S_1\times\cdots\times S_j$. The following inequality is a special case of a more general inequality due to Shearer (see \cite{Shearer}).

\begin{proposition}[Shearer's inequality, special case]\label{shearer}
	Let $X_1,X_2,X_3$ be three random variables supported on finite sets. Then $2\ent(X_1,X_2,X_3)\leq \ent(X_1,X_2) + \ent(X_1,X_3)+\ent(X_2,X_3)$.
\end{proposition}

Having stated the requisite preliminaries, we may now prove \cref{lem:color-ineqs}.

\begin{proof}[Proof of \cref{lem:color-ineqs}] The proofs of (1)--(3) follow the same entropy-based framework. We choose a uniformly random oriented triangle with the given coloring, apply Shearer's inequality (\cref{shearer}), and then apply \cref{prop:entropy-bits} to both sides.

(1) is a slight loosening of the Kruskal--Katona inequality \cite{kk-1,kk-2} but we give a short proof using the aforementioned entropic framework. Let $(X_1,X_2,X_3)\in V^{(3)}$ be the vertices of a green triangle \greentri, chosen uniformly at random among all orientations of all green triangles. Each green triangle contributes $6=3!$ orientations, so $\ent(X_1,X_2,X_3)=\log_2\left(6T(\greentri)\right)$. Moreover, $(X_1,X_2)$ is supported on the set of ordered pairs $(x,y)$ of vertices forming a green edge; this is a set of size $2\ga\binom n2< \ga n^2$. By \cref{prop:entropy-bits}, it follows that $\ent(X_1,X_2)\leq\log_2(\ga n^2)$, and the same upper bound holds for $\ent(X_1,X_3)$ and $\ent(X_2,X_3)$. Thus, applying Shearer's inequality, we have
\begin{align*}
\log_2\left(6T(\greentri)\right)=\ent(X_1,X_2,X_3)\leq\frac 12\left(\ent(X_1,X_2)+\ent(X_1,X_3)+	\ent(X_2,X_3)\right)
\leq\frac 32\log_2(\ga n^2),
\end{align*}
and it follows that $6T(\greentri)\leq\ga^{3/2}n^3$.

(2) Let $(X_1,X_2,X_3)\in V^{(3)}$ be the vertices of a purple directed triangle $\purpletri$, chosen uniformly among all orientations of all purple directed triangles in which edges are directed from $X_i$ to $X_{i+1 \pmod 3}$. Each purple directed triangle contributes 3 such orientations, so $\ent(X_1,X_2,X_3)=\log_2\left(3 T(\purpletri)\right)$. Moreover, $(X_1,X_2)$  is supported on the set of purple edges $(x,y)$ directed from $x$ to $y$, which is a set of size $2\del\binom n2<\del n^2$. By \cref{prop:entropy-bits}, we have $\ent(X_1,X_2)\leq\log_2(\del n^2)$; the same inequality holds for $\ent(X_1,X_3)=\ent(X_3,X_1)$ and for $\ent(X_2,X_3)$. Thus, applying Shearer's inequality, we have
\begin{align*}
\log_2\left(3T(\purpletri)\right)=\ent(X_1,X_2,X_3)\leq\frac 12\left(\ent(X_1,X_2)+\ent(X_1,X_3)+	\ent(X_2,X_3)\right)
\leq\frac 32\log_2(\del n^2),
\end{align*}
and it follows that $3T(\purpletri)\leq\del^{3/2}n^3$.

(3) Let $(X_1,X_2,X_3)\in V^{(3)}$ be the vertices of a cherry $\cherry$, chosen uniformly at random among all such oriented triangles with apex $X_1$ (i.e., among all orientations where $(X_2,X_3)$ is a blue edge). Each cherry contributes 2 such orientations, so $\ent(X_1,X_2,X_3)=\log_2\left(2 T(\cherry)\right)$. Observe that $(X_1,X_2)$ and $(X_1,X_3)$ are both supported on the set of red edges $(x,y)$ directed from $y$ to $x$, while $(X_2,X_3)$ is supported on the set of ordered pairs $(x,y)$ forming a blue edge. The former set has size $2\al\binom n2<\al n^2$, while the latter set has size $2\bet\binom n2<\bet n^2$. Thus, applying \cref{prop:entropy-bits} and Shearer's inequality yields
\begin{align*}
\log_2\left(2T(\cherry)\right)=\ent(X_1,X_2,X_3)&\leq\frac 12\left(\ent(X_1,X_2)+\ent(X_1,X_3)+	\ent(X_2,X_3)\right)
\\&\leq\frac 12\left(2\log_2(\al n^2)+\log_2(\bet n^2)\right),
\end{align*}
and it follows that $2T(\cherry)\leq\al\sqrt\bet n^3$.

(4) A result of Falgas-Ravry and Vaughan \cite{FaVa12} using flag algebras shown also by by Huang \cite{Hu14} using symmetrization (see also \cite{KLP24}) shows that any oriented coloring of $V^{(2)}$ by $\Delbf_2''=\{\redto,\blueedge\}$ has at most $\left(2\sqrt 3-3\right)n^3/6\approx 0.464\times n^3/6$ cherries. The optimal construction (pictured in \Cref{fig:iterated-blowup}) is a disjoint union of unboundedly many blue cliques whose sizes approximate a geometric sequence, with red edges directed from larger cliques to smaller cliques. The constant $2\sqrt 3-3$ comes from optimizing the ratio between the sizes of consecutive blue cliques.

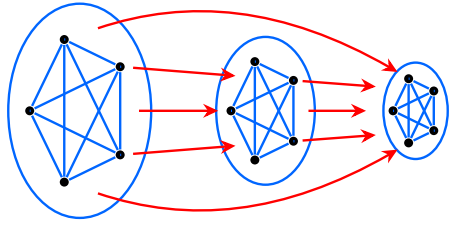
\begin{figure}
\centering
\begin{tikzpicture}[tikz arrows, scale=0.95]
	\draw[yscale=1.5,graphblue] (0,0) circle (1);
	\path [yscale=1.5] 
		(180:0.7) coordinate[vtx] (a1)
		(108:0.7) coordinate[vtx] (a2)
		(36:0.7) coordinate[vtx] (a3)
		(-36:0.7) coordinate[vtx] (a4)
		(-108:0.7) coordinate[vtx] (a5);
	\foreach \i [evaluate=\i as \ii using \i+1] in {1,...,4} {
		\foreach \j in {\ii,...,5} {
			\ifnum \i=\j \else
			\draw[graphblue] (a\i) -- (a\j);
			\fi
		}
	}
\begin{scope}[shift={(2.6,0)},scale=0.69]
	\draw[yscale=1.5,graphblue] (0,0) circle (1);
	\path [yscale=1.5] 
		(180:0.7) coordinate[vtx] (b1)
		(108:0.7) coordinate[vtx] (b2)
		(36:0.7) coordinate[vtx] (b3)
		(-36:0.7) coordinate[vtx] (b4)
		(-108:0.7) coordinate[vtx] (b5);
	\foreach \i [evaluate=\i as \ii using \i+1] in {1,...,4} {
		\foreach \j in {\ii,...,5} {
			\draw[graphblue] (b\i) -- (b\j);
		}
	}
\end{scope}
\begin{scope}[shift={(4.7,0)},scale=0.45]
	\draw[yscale=1.5,graphblue] (0,0) circle (1);
	\path [yscale=1.5] 
		(180:0.7) coordinate[vtx] (c1)
		(108:0.7) coordinate[vtx] (c2)
		(36:0.7) coordinate[vtx] (c3)
		(-36:0.7) coordinate[vtx] (c4)
		(-108:0.7) coordinate[vtx] (c5);
	\foreach \i [evaluate=\i as \ii using \i+1] in {1,...,4} {
		\foreach \j in {\ii,...,5} {
			\draw[graphblue] (c\i) -- (c\j);
		}
	}
\end{scope}
\draw[graphred,->,shorten >=0.7cm, shorten <=0.1cm] (a3) -- (b3);
\draw[graphred,->,shorten >=0.7cm, shorten <=0.1cm] (a4) -- (b4);
\draw[graphred,->,shorten >=1cm, shorten <=0.25cm] (a3|-0,0) -- (b4|-0,0);

\draw[graphred,->,shorten >=0.7cm, shorten <=0.05cm] (b3) -- (c3);
\draw[graphred,->,shorten >=0.7cm, shorten <=0.05cm] (b4) -- (c4);
\draw[graphred,->,shorten >=0.9cm, shorten <=0.2cm] (b3|-0,0) -- (c4|-0,0);

\draw[graphred,->,shorten <=0.4cm, shorten >=0.1cm] (a2) to [bend left=25,] (c2);
\draw[graphred,->,shorten <=0.4cm, shorten >=0.1cm] (a5) to [bend right=25,] (c5);
\end{tikzpicture}
\caption{The coloring of $E(K_n)$ by $\{\blueedge,\redto\}$ that maximizes the number of cherries.}
\label{fig:iterated-blowup}
\end{figure}

(5) We first handle the case that $\al=\bet=\ga=0$, i.e., $V^{(2)}$ is colored only by $\{\purpleto\}$. In this case we show that $T(\purpletri)\leq n^3/24$ by the following well-known variant of Goodman's formula \cite{Goodman}.

If $V^{(2)}$ is solely colored by purple edges, then every triple in $\binom V3$ induces a triangle whose coloring is isomorphic to either $\purpletri$ or $\purpletranstri$. Let $T(\purpletranstri)=\binom n3-T(\purpletri)$ be the number of occurrences of the latter.
Writing $\indeg(v)$ and $\outdeg(v)$ for the purple indegree and purple outdegree of a vertex $v\in V$, we note that
\[
\sum_{v\in V(G)}\indeg(v)\times \outdeg(v)=3T(\purpletri)+T(\purpletranstri)=2T(\purpletri) + \binom n3.
\]
Moreover, $\indeg(v)+\outdeg(v)=n-1$ for each vertex $v$, so $\indeg(v)\times\outdeg(v)\leq\frac 14(n-1)^2$.
It follows that
\[
\frac{n(n-1)^2}4\geq\sum_{v\in V(G)}\indeg(v)\times\outdeg(v)\geq 2T(\purpletri)+\binom n3.
\]
Rearranging this inequality yields
\[
T(\purpletri)\leq \frac{n(n-1)^2}8-\frac{n(n-1)(n-2)}{12}=\frac{n(n-1)(n+1)}{24}
\leq\frac{n^3}{24}.\]

We now handle the general case, in which all four oriented colors in $\Delbf'_2$ may be present. Form a coloring $\chi'$ of $V^{(2)}$ by $\{\purpleto\}$ by recoloring any red, green, or blue edge with a purple directed edge, whose orientation is chosen uniformly at random. Let $T'(\purpletri)$ denote the number of purple directed triangles in $\chi'$; the preceding paragraph gives $T'(\purpletri)\leq n^3/24$. If three vertices $x,y,z$ induced a green triangle or a cherry in the original coloring of $V^{(2)}$, they induce a purple directed triangle in $\chi'$ with probability $\frac 14$. If they induced a purple directed triangle originally, then they also induce such a triangle in $\chi'$. Thus,
\[
\frac{n^3}{24}\geq\E\left[T'(\purpletri)\right]\geq T(\purpletri) + \frac 14\left(T(\greentri)+T(\cherry)\right),
\]
as desired.

(6) We require the following inequality.
\begin{milne}
Let $a_1,\ldots,a_n$ and $b_1,\ldots,b_n$ be sequences of nonnegative real numbers such that $a_i+b_i>0$ for each $i\in[n]$. Then
\[
\left(\sum_{i=1}^n\frac{a_ib_i}{a_i+b_i}\right)\left(\sum_{i=1}^n(a_i+b_i)\right)
\leq\left(\sum_{i=1}^na_i\right)\left(\sum_{i=1}^nb_i\right).
\]	
\end{milne}

\begin{proof}
Observe that $\frac{a_ib_i}{a_i+b_i}=a_i-\frac{a_i^2}{a_i+b_i}$. By the Cauchy--Schwarz inequality,
\[
\left(\sum_{i=1}^n\frac{a_i^2}{a_i+b_i}\right)\left(\sum_{i=1}^n(a_i+b_i)\right)\geq\left(\sum_{i=1}^na_i\right)^2,
\]
so
\begin{align*}
\left(\sum_{i=1}^n\frac{a_ib_i}{a_i+b_i}\right)\left(\sum_{i=1}^n(a_i+b_i)\right)
&\leq\left(\sum_{i=1}^na_i\right)\left(\sum_{i=1}^n(a_i+b_i)\right)-\left(\sum_{i=1}^na_i\right)^2
\\&=\left(\sum_{i=1}^na_i\right)\left(\sum_{i=1}^nb_i\right).\qedhere
\end{align*}
\end{proof}

For each vertex $v\in V$, let $a_v$ be its red outdegree and $b_v$ its blue degree, so $a_v+b_v\leq n-1$. We bound $2T(\cherry)$ by the number of copies of $\cherrypath$, which is exactly $\sum_va_vb_v$. Thus, applying Milne's inequality to the sequences $\{a_v\}$ and $\{b_v\}$, indexed over vertices $v$ with $a_v+b_v>0$, yields
\[
2T(\cherry)\leq\sum_va_vb_v
\leq\sum_v\frac{(n-1)a_vb_v}{a_v+b_v}
\leq\frac{(n-1)\left(\sum_v a_v\right)\left(\sum_v b_v\right)}{\sum_va_v+b_v}.
\]
Substituting $\sum_va_v=2\al\binom n2=\al n(n-1)$ and $\sum_vb_v=2\bet\binom n2=\bet n(n-1)$ yields
\[
T(\cherry)\leq(n-1)\frac{\al\bet n(n-1)}{2(\al+\bet)}\leq\frac{3\al\bet}{\al+\bet}\times\frac{n^3}{6}.\qedhere
\]
\end{proof}

The six inequalities in \cref{lem:color-ineqs} combine to show the following upper bound on the number of triangles colored \purpletri, \greentri, or \cherry. Say an oriented coloring of $V^{(2)}$ by $\Delbf'_2$ has \emph{edge color densities} $(2\al,\bet,\ga,2\del)$ if it has $2\al\binom n2$ red edges, $\bet\binom n2$ blue edges, $\ga\binom n2$ green edges, and $2\del\binom n2$ purple edges.

\begin{corollary}\label{cor:define-f}
Let $g(\al,\bet,\ga)=\ga^{3/2}+\min\left(0.465, \frac{3\al\bet}{\al+\bet},3\al\sqrt\bet \right)$, where $\frac{\al\bet}{\al+\bet}$ is defined to be 0 if $(\al,\bet)=(0,0)$, and set 
$f(\al,\bet,\ga,\del)=\min\big(2\del^{3/2}+g(\al,\bet,\ga), \frac 14+\frac 34 g(\al,\bet,\ga)\big)$.

In any oriented coloring of $E(K_n)$ by $\Delbf'_2$ with edge color densities $(2\al,\bet,\ga,2\del)$, it holds that
\[
T(\purpletri) + T(\greentri) + T(\cherry)\leq f(\al,\bet,\ga,\del)\frac{n^3}{6}.
\]
\end{corollary}

\begin{proof}
Combining \cref{lem:color-ineqs}(1), (3)--(4) and (6) we have 
\[
T(\greentri)+T(\cherry)\leq\left(\ga^{3/2}+\min\left(0.465, \frac{3\al\bet}{\al+\bet},3\al\sqrt\bet \right)\right)\frac{n^3}6=g(\abc)\frac{n^3}6.
\tag{\thetheorem.1}\label{eq:cor-def-f}
\]
Adding \cref{lem:color-ineqs}(2) to (\ref{eq:cor-def-f}) yields 
$T(\purpletri) + T(\greentri) + T(\cherry)\leq \left(2\del^{3/2}+g(\abc)\right)\frac{n^3}6$. 
Adding $\frac 34$ times (\ref{eq:cor-def-f}) to \cref{lem:color-ineqs}(5) yields 
$T(\purpletri) + T(\greentri) + T(\cherry)\leq\left(\frac 14+\frac 34g(\abc)\right)\frac{n^3}6$.
\end{proof}

\subsection{The Tur\'an density of $\C14$}\label{subsec:turan-density}

By studying the function $f(\al,\bet,\ga,\del)$ defined in \cref{cor:define-f}, we may derive an asymptotically tight upper bound on $\ex(n,\C 14\texthom)$. 
As at the beginning of this section, let $V$ be a set of $n$ vertices and $\chi$ an oriented coloring of $V^{(3)}$ by $\Delbf_3=\left\{\pointedtriangle,\bluetriangle,\circletriangle,\rededgetriangle,\yellowedgetriangle\right\}$. 

Our strategy is to bound the number of triangles of the form \purpletri, \greentri, \cherry[e1], \cherry[e2], or \cherry[e3] in a link coloring $\chi_w$ by applying \cref{cor:define-f} to the simplified link coloring $\chi'_w$. However, if $\chi_w$ has too many green edges, this approach is hopeless --- in the worst case, if all edges are green, then all $\binom{n-1}3$ triangles in $V-\{w\}$ will be colored as \greentri. To avoid this case, we observe that, for some $w\in V$, the link coloring $\chi_w$ contains at least twice as many red edges as green edges.

\begin{proposition}\label{prop:redgreenedges}
Suppose $V$ is a set of vertices and $\chi$ an oriented coloring of $V^{(3)}$ by $\Delbf_3$. There is some vertex $w\in V$ such that the link coloring $\chi_w$ contains at least twice as many edges colored \redto[e1] as edges colored \greenedge.
\end{proposition}

\begin{proof}
Let $T(\pointedtriangle)$ be the number of triples colored $\pointedtriangle$ by $\chi$. If $w\in V$ is chosen uniformly at random, then in expectation $\chi_w$ colors $\frac 2nT(\pointedtriangle)$ edges with \redto[e1]\ and $\frac 1nT(\pointedtriangle)$ edges with \greenedge.
\end{proof}

To count triangles of the form \purpletri, \greentri, \cherry[e1], \cherry[e2], or \cherry[e3] in $\chi_w$, we apply \cref{cor:define-f} to $\chi'_w$, where $w$ is a vertex satisfying \cref{prop:redgreenedges}. In other words, the edge color densities $(2\al,\bet,\ga,2\del)$ of $\chi'_w$ must satisfy $\ga\leq\al$. Hence, it suffices to study the function $f(\al,\bet,\ga,\del)$ in the regime that $\ga\leq\al$.

We require the following lemma, which is an exercise in multivariable calculus. A full proof is given in \cref{appendix:calculus}; the lemma may also be checked by computational software such as Mathematica.

\begin{lemma}\label{calculus}
Let $f(\al,\bet,\ga,\del)$ be the function defined in \cref{cor:define-f}. If $\al,\bet,\ga,\del$ are nonnegative real numbers with $2\al+\bet+\ga+2\del=1$ and $\ga\leq\al$ then $f(\al,\bet,\ga,\del)\leq\frac 12$. Equality holds if and only if $(\al,\bet,\ga,\del)=(\frac 14,\frac 14,\frac 14,0)$.
\end{lemma}

Using \cref{calculus}, we prove that the family $\C 14$ has Tur\'an density $1/2$. The lower bound is given by the complete oddly bipartite construction, which is $\C14$-hom-free by \cref{prop:Godd-C-free}.

\begin{theorem}\label{thm:C14-turan-density}
Let $G$ be a $\C 14$-hom-free 4-graph on $n$ vertices. Then $e(G)<n^4/48$.	
\end{theorem}

\begin{proof}
We induct on $n$. For $n<4$, it is clear that $e(G)=0<n^4/48$.

Now, suppose $n\geq 4$ and write $V=V(G)$. Let $\chi$ be the coloring of $V^{(3)}$ by $\Delbf_3$ guaranteed by \cref{cor:C14-free}(4). For each $w\in V$, let $\chi_w$ be the link coloring of $(V-\{w\})^{(2)}$ by $\Delbf_2$ defined by \cref{table:link-coloring}, and let $\chi'_w$ be the simplified link coloring of $(V-\{w\})^{(2)}$ by $\Delbf_2'$. By \cref{prop:redgreenedges}, there is some vertex $w$ such that $\chi'_w$ contains at least twice as many red edges as green edges. Combining \cref{cor:define-f} and \cref{calculus}, at most $(n-1)^3/12$ triples are colored as \purpletri, \greentri, or \cherry\ by $\chi'_w$. 
Hence, $\deg_G(w)\leq(n-1)^3/12$. Applying the inductive hypothesis,
\[e(G)=e(G-w)+\deg_G(w)<\frac{(n-1)^4+4(n-1)^3}{48}<\frac{n^4}{48}.\qedhere\]
\end{proof}

\begin{remark}
Although \cref{thm:C14-turan-density} does not show that the extremal $\C14$-hom-free construction is complete oddly bipartite, we suspect this could be derived from the ideas in this section without too much trouble. We do not prove this statement, as it is a corollary of \cref{thm:one-cycle}.
\end{remark}

\subsection{Stability}
\label{subsection:stability}

We conclude this section by proving \cref{thm:stability}, which shows that, if $G$ is $\C14$-hom-free with edge density close to $1/2$, then $G$ differs from a complete oddly bipartite 4-graph on only a small fraction of 4-edges.

To prove this result, we first show \cref{lem:link-stability}, which states that if a link coloring $\chi_w$ satisfying the conclusion of \cref{prop:redgreenedges} has close to $\frac 12\binom {n-1}3$ triples colored \purpletri, \greentri, \cherry[e1], \cherry[e2], or \cherry[e3], then $\chi_w$ looks close to the link coloring depicted in \Cref{fig:Godd-link-coloring}, which is the link coloring of a vertex in a complete oddly bipartite 4-graph. We begin by deriving a stability version of the inequality used to prove \cref{lem:color-ineqs}(6).

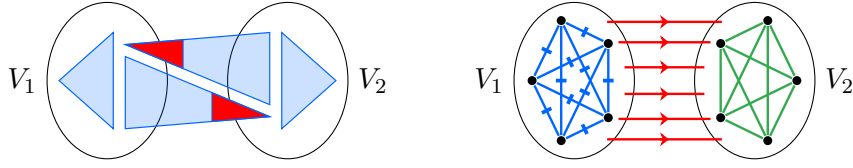
\begin{figure}
\centering
\begin{tikzpicture}[scale=0.8]
        \coordinate (o1) at (0,0);
        \coordinate (o2) at (3,0);
        \draw[yscale=1.3] (o1) circle (1) (o2) circle (1);
        \path (o1) --++(-1,0) node[left]{$V_1$};
        \path (o2) --++(1,0) node[right]{$V_2$};
        \path (o1) --+ (-0.8,0) coordinate(x1)
                        --+ (0.1,0.8) coordinate (y1)
                        --+ (0.1,-0.8) coordinate (z1);
        \path (o2) --+ (0.8,0) coordinate(x2)
                        --+ (-0.1,0.8) coordinate (y2)
                        --+ (-0.1,-0.8) coordinate (z2);
        \path (y1) --++ (0.2,0)
                        --++ (0,-0.2) coordinate (a1) --++ (0,-0.2) coordinate (a3)
                        (z1) --+ (0.2,0) coordinate (a2)
                        (y2) --+ (-0.2,0) coordinate (b1)
                        (z2) --++ (-0.2,0)
                        --++ (0,0.2) coordinate (b2) --++ (0,0.2) coordinate (b3);

        \tikzmakebluetriangle {x1}{y1}{z1};
        \tikzmakebluetriangle {x2}{y2}{z2};
        \tikzmakepointedtriangle {a1}{b1}{b3};
        \tikzmakepointedtriangle {b2}{a3}{a2};
\end{tikzpicture}
\qquad
\begin{tikzpicture}[scale=0.8]
        \coordinate (o1) at (0,0);
        \coordinate (o2) at (3,0);
        \draw[yscale=1.3] (o1) circle (1) (o2) circle (1);
        \path (o1) --++(-1,0) node[left]{$V_1$};
        \path (o2) --++(1,0) node[right]{$V_2$};

\begin{scope}[tikz arrows]
        \path [yscale=1.5]
                (180:0.7) coordinate[vtx] (a1)
                (108:0.7) coordinate[vtx] (a2)
                (36:0.7) coordinate[vtx] (a3)
                (-36:0.7) coordinate[vtx] (a4)
                (-108:0.7) coordinate[vtx] (a5);
        \path [yscale=1.5,shift=(o2)]
                (0:0.7) coordinate[vtx] (b1)
                (72:0.7) coordinate[vtx] (b2)
                (144:0.7) coordinate[vtx] (b3)
                (-144:0.7) coordinate[vtx] (b4)
                (-72:0.7) coordinate[vtx] (b5);
        \foreach \i/\j in {1/2,1/3,1/4,1/5,2/3,2/4,2/5,3/4,3/5,4/5} {
                \draw[blue edge, e1] (a\i) -- (a\j);
                \draw[graphgreen] (b\i) -- (b\j);
        }
        \foreach \ang [count=\i] in {-50,-30,...,50} {
                \path [yscale=1.5] (\ang:0.85) coordinate (r\i);
                \path [yscale=1.5] (180-\ang:0.85) coordinate (s\i);
                \path (s\i) --++ (o2) coordinate (t\i);
                \draw[red arrow, e1] (r\i) -- (t\i);
        }
\end{scope}
\end{tikzpicture}

\caption{At left, the coloring of $(V_1\cup V_2)^{(3)}$ from \Cref{fig:Godd-coloring}. At right, the associated link coloring $\chi_w$ of a vertex $w\in V_1$.}
\label{fig:Godd-link-coloring}
\end{figure}

\begin{proposition}\label{milne-stability}
Suppose $a_1,\ldots,a_n$ and $b_1,\ldots,b_n$ are nonnegative real numbers with $a_i+b_i\leq 1$ for each $i$. Then
\[
\sum_{i=1}^na_ib_i\leq\frac 14\sum_{i=1}^n(a_i+b_i).
\]
Moreover, for each $\eps>0$, if
\[
\sum_{i=1}^na_ib_i\geq\frac 14\sum_{i=1}^n(a_i+b_i)-\eps^3 n
\]
then all but at most $\eps n$ of the pairs $(a_i,b_i)$ satisfy $a_i,b_i\in\big(\frac 12-\eps,\frac 12+\eps\big)$ or $a_i,b_i\in[0,\eps)$.
\end{proposition}

\begin{proof} 
It holds that
\[
\frac 14\sum_{i=1}^n(a_i+b_i)-\sum_{i=1}^na_ib_i
=\frac 14\sum_{i=1}^n(a_i+b_i)(1-a_i-b_i)+(a_i-b_i)^2.
\tag{\theproposition.1}\label{eq:milne-stability}
\]
Because $a_i+b_i\leq 1$, this difference is a sum of nonnegative terms and hence nonnegative.

We count indices $i$ with $|a_i-b_i|\geq\eps$ or $\eps\leq a_i+b_i\leq 1-\eps$. If $|a_i-b_i|\geq\eps$, then $(a_i-b_i)^2\geq\eps^2$. If $\eps\leq a_i+b_i\leq1-\eps$, then $(a_i+b_i)(1-a_i-b_i)\geq\eps(1-\eps)\geq\eps^2$. 
Thus, if the difference in (\ref{eq:milne-stability}) is at most $\eps^3n$, then there are at most $\eps^3n/\eps^2=\eps n$ indices $i$ with either $|a_i-b_i|\geq\eps$ or $\eps\leq a_i+b_i\leq1-\eps$.

Suppose a pair $(a_i,b_i)$ satisfies $|a_i-b_i|<\eps$ and either $a_i+b_i<\eps$ or $a_i+b_i>1-\eps$. 
If $a_i+b_i<\eps$, then clearly $a_i,b_i\in[0,\eps)$. If $a_i+b_i>1-\eps$, then $2a_i\geq a_i+b_i-|b_i-a_i|>1-\eps-\eps$, so $a_i>\frac 12-\eps$. Similarly, $b_i>\frac 12-\eps$, and because $a_i+b_i\leq 1$ it holds that $a_i,b_i\in\left(\frac 12-\eps,\frac 12+\eps\right)$.

It follows that the pairs $(a_i,b_i)$ satisfy the desired conclusion except for the at most $\eps n$ pairs satisfying $|a_i-b_i|\geq\eps$ or $\eps\leq a_i+b_i\leq 1-\eps$.
\end{proof}

\begin{lemma}\label{lem:link-stability}
For each $\eps>0$ there is $\eps_0=\eps_{\theproposition}(\eps)>0$ such that the following holds.

Let $V$ be a set of $n$ vertices, where $n$ is sufficiently large in terms of $\eps$. Fix an oriented coloring of $V^{(2)}$ by the set $\Delbf_2$ defined in (\ref{eq:del2}) in which at least twice as many edges are colored $\redto[e1]$ as $\greenedge$. If the number of triangles colored \purpletri, \greentri, \cherry[e1], \cherry[e2], or \cherry[e3] is at least $(\frac 12-\eps_0)\frac{n^3}{6}$ then there is a partition $V=V_1\cup V_2$ with $(\frac 12-\eps)n\leq |V_i|\leq(\frac 12+\eps)n$ satisfying the following property. The number of edges within $V_1$ not colored $\blueedge[e1]$, edges within $V_2$ not colored $\greenedge$, and ordered pairs $(v_1,v_2)\in V_1\times V_2$ not colored $\redtolabeled[e1]{v_1}{v_2}$ total to at most $\eps n^2$.
\end{lemma}

\begin{proof}
Set $\eps_3=\eps$ and suppose $\eps_2,\eps_1,\eps_0$ are chosen with each $\eps_i$ sufficiently small in terms of $\eps_{i+1}$. Suppose $V$ is a set of size $n$, with $n$ sufficiently large in terms of $\eps_3,\ldots,\eps_0$. Fix a coloring of $V^{(2)}$ as described. Suppose this coloring of $V^{(2)}$ has $2\al_1\binom n2$, $2\al_2\binom n2$, $2\al_3\binom n3$ edges colored $\redto[e1]$, $\redto[e2]$, $\redto[e3]$ respectively; $\bet_1\binom n2$, $\bet_2\binom n2$, $\bet_3\binom n2$ edges colored $\blueedge[e1]$, $\blueedge[e2]$, $\blueedge[e3]$ respectively; $\ga\binom n2$ edges colored $\greenedge$; and $2\del\binom n2$ edges colored $\purpleto$. 

Let $f$ be the function defined in \cref{cor:define-f}.
Applying \cref{cor:define-f} to the simplified coloring of $V^{(2)}$ by $\Delbf'_2$ (in which red and blue edges are recolored with \redto\ and \blueedge, respectively), the number of triangles colored \purpletri, \greentri, \cherry[e1], \cherry[e2], or \cherry[e3] is at most $f(\abcd)n^3/6$, where $\al=\al_1+\al_2+\al_3$ and $\bet=\bet_1+\bet_2+\bet_3$. Thus, $f(\abcd)\geq\frac 12-\eps_0$.

Because $f$ is a continuous function on the region
\[
R=\{(\abcd)\in\mathbb R_{\geq 0}^4:2\al+\bet+\ga+2\del=1,\ga\leq\al\}
\]
which by \cref{calculus} attains its global maximum on $R$ uniquely at $\abcdopt$,
one may choose $\eps_0$ small enough in terms of $\eps_1$ so that
\[|\al'-0.25|+|\bet'-0.25|+|\ga'-0.25|+|\del'-0|\leq\eps_1
\tag{\thetheorem.1}\label{eq:calculus-limit}
\]
for any $(\al',\bet',\ga',\del')\in R$ satisfying $f(\al',\bet',\ga',\del')\geq\frac 12-\eps_0$. In particular, (\ref{eq:calculus-limit}) is satisfied by $(\abcd)$. Because $\al_1\geq\ga$, it holds that $\al_2+\al_3=\al-\al_1\leq\al-\ga\leq \eps_1$. We conclude that at most $2(\al_2+\al_3)\binom n2\leq\eps_1n^2$ pairs are colored $\redto[e2]$ or $\redto[e3]$, and thus at most $\eps_1n^3$ triples are colored as $\cherry[e2]$ or $\cherry[e3]$. Similarly, at most $2\del\binom n2\leq\eps_1 n^2$ pairs are colored $\purpleto$ and thus at most $\eps_1n^3$ triples are colored as $\purpletri$. By \cref{lem:color-ineqs}(1), at most $\ga^{3/2} n^3/6\leq(\frac 14+\eps_1)^{3/2}n^3/6$ triples are colored $\greentri$; for sufficiently small $\eps_1$ this quantity is at most $(\frac 18+2\eps_1)n^3/6$. We conclude that the number of triples colored $\cherry[e1]$ is at least
\[
T(\cherry[e1])\geq\left(\frac 12-\eps_0\right)\frac{n^3}6-\eps_1n^3-\eps_1n^3-\left(\frac 18+2\eps_1\right)\frac{n^3}6\geq\frac{n^3}{16}-3\eps_1n^3.
\]

For each vertex $v$, let $a_v$ and $b_v$ be the  numbers of vertices $w$ with $vw$ colored as $\redtolabeled[e1]{v}{w}$ or $\blueedgelabeled[e1]{v}{w}$, respectively. As in the proof of \cref{lem:color-ineqs}(6), note that $T(\cherry[e1])$ is at most half the number of copies of $\cherrypath[e1]$, i.e., $T(\cherry[e1])\leq\frac 12\sum_va_vb_v$.
Observe also that
\[
\sum_v(a_v+b_v)=2\al_1\binom n2+2\bet_1\binom n2\leq(\al_1+\bet_1)n^2\leq\left(\frac 12+2\eps_1+\bet_1-\bet\right)n^2.
\]
It follows that
\[
\sum_v a_vb_v\geq 2T(\cherry[e1])\geq\frac{n^3}8-6\eps_1n^3\geq\frac n4\sum_v(a_v+b_v)-7\eps_1n^3+\frac{\bet-\bet_1}4n^3.
\tag{\thetheorem.2}\label{eq:bet-and-bet1}
\]

Because $a_v+b_v<n$ for each vertex $v$, we may apply \cref{milne-stability} to the sequences $a_v/n$ and $b_v/n$. Combining the first part of \cref{milne-stability} with (\ref{eq:bet-and-bet1}) yields
\[
\frac 14\sum_v\frac{a_v+b_v}n\geq\sum_v\frac{a_vb_v}{n^2}\geq \frac 14\sum_v\frac{a_v+b_v}n-7\eps_1n+\frac{\bet-\bet_1}4n,
\]
implying that $\bet_1\geq\bet-28\eps_1\geq\frac 14-29\eps_1$. 
Set $V_1=\left\{v\in V:\left(\frac 12-\eps_2\right)n\leq a_v,b_v\leq\left(\frac 12+\eps_2\right)n\right\}$ and set $V_2=V\setminus V_1$. If $7\eps_1\leq\eps_2^3$, the second part of \cref{milne-stability} implies that at most $\eps_2n$ vertices $v\in V_2$ satisfy $a_v\geq\eps_2n$ or $b_v\geq\eps_2n$. Let $\Vbad$ be the set of such vertices and set $V'_2=V_2\setminus\Vbad$. 

Note that $\sum_v(a_v+b_v)=(\al_1+\bet_1)(n^2-n)$. For each $v\in V_1$ we have $(1-2\eps_2)n\leq a_v+b_v\leq n$ and for each $v\in V'_2$ we have $a_v+b_v\leq2\eps_2n$. Thus,
\begin{align*}
\left|(\al_1+\bet_1)n^2-n|V_1|\right|
&\leq(\al_1+\bet_1)n+\Big|\sum_{v\in V}(a_v+b_v)-n|V_1|\Big|
\\&\leq(\al_1+\bet_1)n+2\eps_2n|V_1|+2\eps_2n|V'_2|+n|\Vbad|\leq4\eps_2n^2.
\end{align*}
Previously, we have observed the following bounds on $\al_1$ and $\bet_1$:
\[\frac 14-\eps_1\leq\ga\leq\al_1\leq\al\leq\frac 14+\eps_1\text{ and }\frac 14-29\eps_1\leq\bet_1\leq\bet\leq\frac 14+\eps_1.
\tag{\thetheorem.3}\label{eq:albet}\] 
Thus, $|\frac 12-(\al_1+\bet_1)|\leq 30\eps_1\leq\eps_2$ and
\[\left|\frac n2-|V_1|\right|\leq n\left|\frac 12-(\al_1+\bet_1)\right|+\big|(\al_1+\bet_1)n-|V_1|\big|\leq 5\eps_2n\leq\eps_3 n,
\tag{\thetheorem.4}
\label{eq:V1-half}\]
using that each $\eps_i$ is sufficiently small in terms of $\eps_{i+1}$. Moreover, $\left|\frac n2-|V_2|\right|=\left|\frac n2-|V_1|\right|\leq\eps_3n$.

We conclude the proof by showing that the coloring of $V^{(2)}$ is close to the coloring pictured in \cref{fig:Godd-link-coloring}.
The number of edges colored $\blueedge[e1]$ with at least one endpoint in $V_2$ is at most
\[
\sum_{v\in V_2}b_v\leq|V_2'|\times\eps_2n+|\Vbad|\times n\leq \eps_2n^2+\eps_2n^2=2\eps_2n^2.
\]
Thus, using (\ref{eq:albet}), there are least $\bet_1\binom n2-2\eps_2 n^2\geq \frac{n^2}{8}-3\eps_2n^2$ edges colored $\blueedge[e1]$ with both endpoints in $V_1$. We have $|V_1|^2/2\leq\frac 12\left(\frac n2+5\eps_2 n\right)^2\leq n^2/8+3\eps_2n^2$ by (\ref{eq:V1-half}), so at most $6\eps_2 n^2$ edges with both endpoints in $V_1$ are not colored as $\blueedge[e1]$.

The number of edges colored $\redtolabeled[e1]{v}{w}$ with $v\notin V_1$ is at most
\[
\sum_{v\in V_2}a_v\leq |V_2'|\times\eps_2n+|\Vbad|\times n\leq\eps_2n^2+\eps_2n^2\leq 2\eps_2n^2.
\]
Additionally, at most $6\eps_2n^2$ edges with both endpoints in $V_1$ are colored as $\redtolabeled[e1]{v}{w}$ because all but $6\eps_2n^2$ edges in $V_1$ are colored as $\blueedge[e1]$. Thus, by (\ref{eq:albet}), at least $2\al_1\binom n2-8\eps_2n^2\geq n^2/4-9\eps_2n^2$ pairs $(v_1,v_2)\in V_1\times V_2$ are colored as $\redtolabeled[e1]{v_1}{v_2}$. Because $|V_1\times V_2|\leq\frac{n^2}4$, at most $9\eps_2n^2$ pairs in $V_1\times V_2$ are not colored in this fashion.

Lastly, there are $\ga\binom n2\geq\left(\frac 14-\eps_1\right)\binom n2$ green edges $\greenedge$, of which at most $6\eps_2n^2+9\eps_2n^2$ have both or one endpoint in $V_1$. That is, there are at least $n^2/8-16\eps_2n^2$ green edges with both endpoints within $V_2$. We have $|V_2|^2/2\leq\frac 12\left(\frac n2+5\eps_2 n\right)^2\leq n^2/8+3\eps_2n^2$ by (\ref{eq:V1-half}), so at most $19\eps_2n^2$ edges of $V_2$ are not colored green.

In conclusion, at most $6\eps_2n^2+9\eps_2n^2+19\eps_2n^2=34\eps_2n^2\leq\eps_3 n^2$ edges are not colored as described in the lemma statement, assuming $\eps_2\leq\eps_3/34$.
\end{proof}

We may now prove our full stability result for $\C14$-hom-free 4-graphs.

\begin{theorem}\label{thm:stability}
For each $\eps>0$ there is $\eps_0=\eps_{\thetheorem}(\eps)>0$ such that the following holds.

Let $G$ be a $\C14$-hom-free 4-graph on $n$ vertices, where $n$ is sufficiently large in terms of $\eps$. If $e(G)\geq\left(\frac 12-\eps_0\right)\frac{n^4}{24}$, then there is a partition of $V(G)$ into $V_1\cup V_2$ with $|V_i|\geq\left(\frac 12-\eps\right)n$ so that $|E(G)\setminus E(\Godd(V_1,V_2))|\leq\eps n^4$.
\end{theorem}

\begin{proof}
Set $\eps_3=\eps$ and suppose $\eps_2,\eps_1,\eps_0$ are chosen with each $\eps_i$ sufficiently small in terms of $\eps_{i+1}$. Let $G$ be a $\C14$-hom-free 4-graph on $n$ vertices with $e(G)\geq\left(\frac 12-\eps_0\right)\frac{n^4}{24}$, and suppose $n$ is sufficiently large.

For any $W\subseteq V(G)$, we have $e(G[W])\leq\frac 12|W|^4/24$ by \cref{thm:C14-turan-density}. Thus, assuming $\eps_0\leq\epsref{prelim:large-min-degree}(4,\eps_1)$ we may apply \cref{prelim:large-min-degree} to obtain a subset $V'\subseteq V$ of size $|V'|\geq(1-\eps_1)n$ such that every vertex in the induced subgraph $G'=G[V']$ has degree at least $\left(\frac 12-\eps_1\right)(|V'|-1)^3/6$. Set $m=|V'|-1$.

Let $\chi$ be a coloring of $(V')^{(3)}$ by the set $\Delbf_3=\left\{\pointedtriangle,\bluetriangle,\circletriangle,\rededgetriangle,\yellowedgetriangle\right\}$ agreeing with \Cref{fig:tetrahedron-boundaries} on each edge of $G$; the existence of such $\chi$ is guaranteed by \cref{cor:C14-free}(4). By \cref{prop:redgreenedges}, there is some vertex $w$ such that the link coloring $\chi_w$ associated to $\chi$ by \cref{table:link-coloring} contains at least twice as many edges colored \redto[e1] as \greenedge. Moreover, the number of triples colored \purpletri, \greentri,\cherry[e1], \cherry[e2], or \cherry[e3] by $\chi_w$ is at least  $\deg_{G'}(w)\geq\left(\frac 12-\eps_1\right)m^3/6$. Set $W=V'-\{w\}$, so $|W|=m$.

Assume $\eps_1\leq\epsref{lem:link-stability}(\eps_2)$. Applying \cref{lem:link-stability} to the link coloring $\chi_w$ of $W^{(2)}$, we conclude that there is a partition $W=W_1\cup W_2$ with $(\frac 12-\eps_2)m\leq|W_i|\leq(\frac 12+\eps_2)m$ such that, in $\chi_w$, the number of edges in $\binom{W_1}2$ not colored $\blueedge[e1]$, edges in $\binom{W_2}2$ not colored $\greenedge$, and ordered pairs $(v_1,v_2)\in W_1\times W_2$ not colored as $\redtolabeled[e1]{v_1}{v_2}$ is in total at most $\eps_2m^2$. Let $\Ebad\subset\binom{W}2$ be the set of such pairs.

For $i=0,1,2,3$, we define sets $\TT_i$ that collect the triples $xyz\in\binom{W}3$ with $i$ vertices in $W_1$ whose coloring by $\chi$ does not agree with \Cref{fig:Godd-coloring}. More precisely,
\begin{itemize}
\setlength\itemsep{-5pt}
\item $\TT_3$ is the set of triples $xyz\in\binom{W_1}3$ not colored as \bluetrianglelabeled xyz by $\chi$;
\item $\TT_2$ is the set of triples $v_1v'_1v_2\in\binom{W_1}2\times W_2$ not colored as \pointedtrianglerightlabeled{v_1}{v'_1}{v_2} by $\chi$;
\item $\TT_1$ is the set of triples $v_1v_2v'_2\in W_1\times\binom{W_2}2$ not colored as \pointedtriangleleftlabeled{v_2}{v_1}{v'_2} by $\chi$; and
\item $\TT_0$ is the set of triples $xyz\in\binom{W_2}3$ not colored as \bluetrianglelabeled xyz by $\chi$.
\end{itemize}
Set $\Tbad=\TT_0\cup\TT_1\cup\TT_2\cup\TT_3$. We shall show that $\Tbad$ contains only a small fraction of $\binom W3$.

\begin{claim}
$|\Tbad|\leq 21\eps_2m^3$.	
\end{claim}

\begin{proof}\proofofclaim
Let $\LL(w)$ denote the link of $w$ in $G[V']$.
Recalling (\ref{eq:chi-and-link-pairs}), observe that if $xyz\in E(\LL(w))$ and $xy,xz,yz\notin \Ebad$ then either $xyz$ is in $\binom{W_1}2\times W_2$ and colored \pointedtrianglerightlabeled xyz by $\chi$ or $xyz$ is in $\binom{W_2}3$ and colored \bluetrianglelabeled xyz by $\chi$. That is, $xyz$ is either in $\left(\binom{W_1}2\times W_2\right)-\TT_2$ or $\binom{W_2}3-\TT_0$. Thus,
\[
\left(\frac 12-\eps_1\right)\frac{m^3}6\leq e(\LL(w))\leq m|\Ebad|+\left(\binom{|W_1|}2\times|W_2|+\binom{|W_2|}3\right)-(|\TT_0|+|\TT_2|)
\]
and it follows that
\begin{align*}
|\TT_0|+|\TT_2|&\leq m|\Ebad|
	-\left(\frac 12-\eps_1\right)\frac{m^3}6
	+\frac 12|W_1|^2|W_2|+\frac 16|W_2|^3
\\&=m|\Ebad|-\left(\frac 12-\eps_1\right)\frac{m^3}6
	+\frac{(|W_1|+|W_2|)^3+(|W_2|-|W_1|)^3}{12}
\\&\leq\eps_2m^3-\left(\frac 12-\eps_1\right)\frac{m^3}6
	+\frac{m^3+(2\eps_2m)^3}{12}
<2\eps_2m^3,
\tag{\thetheorem.1}
\label{eq:T0T2}
\end{align*}
assuming $\eps_1$ is sufficiently small in terms of $\eps_2$.

To bound $|\TT_1|$ and $|\TT_3|$, let us first make the following observation. If $x_1x_2x_3x_4\in E(G[W]))$, then either 0, 3, or 4 of the four sub-triples $x_ix_jx_k$ are in $\Tbad$. Indeed, if any two of these sub-triples are colored according to \Cref{fig:Godd-coloring} and the remaining two colored according to \Cref{fig:tetrahedron-boundaries}, then $x_1x_2x_3x_4$ must be in $W_1\times\binom{W_2}3$ or $\binom{W_1}3\times W_2$ with the remaining two sub-triples also colored as in \Cref{fig:Godd-coloring}. Furthermore, if the edge $x_1x_2x_3x_4$ is in $W_1\times\binom{W_2}3$ or $\binom{W_1}3\times W_2$, the same argument shows that either 0 or 4 of the sub-triples $x_ix_jx_k$ are in $\Tbad$.

Fix a vertex $w_2\in W_2$. We first count edges $w_2xyz\in E(G[W])$ with no sub-triple in $\Tbad$. The number of such edges with $xyz\in\binom{W_1}3$ is at most
$\binom{|W_1|}3-|\TT_3|$
and the number with $xyz\in W_1\times\binom{W_2}2$ is at most
$|W_1|\times\binom{|W_2|}2-|\TT_1|$.

We now count edges $w_2xyz\in E(G[W])$ with at least three sub-triples in $\Tbad$. If $xyz\in\binom{W_1}2\times W_2$, then at least one of the triples $xyw_2$ and $xyz$ is in $\TT_2$. If $xyz\in\binom{W_2}3$, then at least one of $w_2xy$ and $w_2xz$ is in $\TT_0$. If $xyz$ is in $\binom{W_1}3$ or $W_1\times\binom{W_2}2$, then all four sub-triples of $w_2xyz$ are in $\Tbad$. In particular, if $xyz\in\binom{W_1}3$, then $xyw_2\in \TT_2$, and if $xyz\in W_1\times\binom{W_2}2$, then $w_2yz\in\TT_0$. It follows that the number of edges $w_2xyz\in E(G[W])$ with at least three sub-triples in $\Tbad$ is at most
\[
\begin{array}{c}
\left(\deg_{\TT_2}(w_2)\times|W_2|+|\TT_2|\right)+\deg_{\TT_0}(w_2)\times|W_2|+\deg_{\TT_2}(w_2)\times|W_1|+\deg_{\TT_0}(w_2)\times |W_1|
\\[5pt]=|\TT_2|+(\deg_{\TT_0}(w_2)+\deg_{\TT_2}(w_2))\times m.
\end{array}
\]

Combining the previous two paragraphs, it holds that
\begin{align*}
\sum_{w_2\in W_2}&\deg_G(w_2;W)
\\&\leq\sum_{w_2\in W_2}\left[\binom{|W_1|}3-|\TT_3|+|W_1|\times\binom{|W_2|}2-|\TT_1|
	+|\TT_2|+(\deg_{\TT_0}(w_2)+\deg_{\TT_2}(w_2))\times m\right]
\\&\leq\frac{|W_1|^3|W_2|}6-|\TT_3||W_2|+\frac{|W_1||W_2|^3}2-|\TT_1||W_2|
	+|\TT_2||W_2|+3m|\TT_0|+m|\TT_2|
\\&=|W_2|\times\left(\frac{(|W_1|+|W_2|)^3+(|W_1|-|W_2|)^3}{12}-|\TT_1|-|\TT_3|
	+|\TT_2|+\frac{m}{|W_2|}(3|\TT_0|+|\TT_2|)\right)
\\&=|W_2|\times\left(\frac{m^3+(2\eps_2m)^3}{12}-|\TT_1|-|\TT_3|
	+|\TT_2|+\frac{m}{|W_2|}(3|\TT_0|+|\TT_2|)\right).
\end{align*}
Observe that $\frac{m}{|W_2|}\leq\left(\frac 12-\eps_2\right)^{-1}\leq 3$ for sufficiently small $\eps_2$. Hence,
\begin{align*}
\frac 1{|W_2|}\sum_{w_2\in W_2}\deg_G(w_2;W)
&\leq\frac{m^3}{12}+\frac{2\eps_2^3m^3}{3}-(|\TT_1|+|\TT_3|)+9|\TT_0|+4|\TT_2|
\\&\leq\frac{m^3}{12}+\frac{2\eps_2m^3}{3}-(|\TT_1|+|\TT_3|)+18\eps_2m^3
\tag{\thetheorem.2}\label{eq:T1T3-upper}
\end{align*}
using (\ref{eq:T0T2}).
In the opposite direction,
\begin{align*}
\sum_{w_2\in W_2}\deg_G(w_2;W)
&=\sum_{w_2\in W_2}\left(\deg_{G'}(w_2)-\deg_{G'}(w_2w)\right)
\geq|W_2|\times\left(\frac 12-\eps_1\right)\frac{m^3}6-3\deg_{G'}(w)
\\&\geq|W_2|\times\left(\frac 12-\eps_1\right)\frac{m^3}6-\frac{m^3}2
\geq|W_2|\times\left(\frac 12-\eps_2\right)\frac{m^3}6
\tag{\thetheorem.3}\label{eq:T1T3-lower}
\end{align*}
if $\eps_1$ is sufficiently small in terms of $\eps_2$ and $m$ is sufficiently large. Combining (\ref{eq:T1T3-upper}) and (\ref{eq:T1T3-lower}),
\[
|\TT_1|+|\TT_3|\leq\frac{m^3}{12}+\left(\frac{2\eps_2^3}{3}+18\eps_2\right)m^3-\left(\frac 12-\eps_2\right)\frac{m^3}{6}<19\eps_2m^3
\]
and with (\ref{eq:T0T2}) this yields $|\Tbad|=|\TT_0|+|\TT_1|+|\TT_2|+|\TT_3|<21\eps_2m^3$.
\end{proof}

Observe that $|E(G[W])\setminus E(\Godd(W_1,W_2))|\leq|W|\times|\Tbad|\leq 21\eps_2m^4$.
Choose a partition $V=V_1\cup V_2$ such that $V_i\supseteq W_i$. We have $|V_i|\geq|W_i|\geq\left(\frac 12-\eps_2\right)m$ and
\begin{align*}
|E(G)\setminus E(\Godd(V_1,V_2))|&\leq \binom n4-\binom m4+|E(G[W])\setminus E(\Godd(W_1,W_2))
\\&\leq\frac{n^4-m^4}{24}+21\eps_2m^4.
\end{align*}
Recall that $m=|V'|-1\geq(1-\eps_1)n-1$. Thus, if $\eps_1$ and $\eps_2$ are sufficiently small in terms of $\eps_3$ and $n$ is sufficiently large in terms of $\eps_1,\eps_2,\eps_3$, it holds that $|V_i|\geq\left(\frac 12-\eps_3\right)n$ and $|E(G)\setminus E(\Godd(V_1,V_2))|\leq \eps_3n^4$.
\end{proof}

\section{The Extremal $C_L^{(4)}$-hom-free Hypergraphs}\label{sec:cleaning}

In this section, we prove the following result, which implies \cref{thm:unif4-density}.

\begin{theorem}\label{thm:one-cycle}
There exists some $L_0\equiv 1\pmod 4$ such that $\ex(n,C_{L_0}^{(4)}\texthom)=\eopt(n)$ for all sufficiently large $n$. Moreover, every $C_{L_0}^{(4)}$-hom-free 4-graph with $\eopt(n)$ edges is complete oddly bipartite.
\end{theorem}

\cref{thm:one-cycle} immediately implies the following stronger version of \cref{thm:unif4-density}.

\begin{corollary}\label{cor:unif4-cycles}
Let $L_0$ be the integer defined in \cref{thm:one-cycle}. There is an absolute constant $n_0$ such that for any $n\geq n_0$ and any $L\geq 3L_0$ not divisible by 4, we have $\ex(n,C_L^{(4)}\texthom)=\eopt(n)$. Moreover, every $C_L^{(4)}$-hom-free 4-graph with $\eopt(n)$ edges is complete oddly bipartite.
\end{corollary}

\begin{proof}[Proof of \cref{cor:unif4-cycles} from \cref{thm:one-cycle}]
Fix $L\geq 3L_0$ not divisible by 4. By \cref{prop:hom-free-chain}, it follows that any $C_L^{(4)}$-hom-free 4-graph is also $C_{L_0}^{(4)}$-hom-free.
Thus, if $n$ is sufficiently large and $G$ is a $C_L^{(4)}$-hom-free 4-graph with $n$ vertices and at least $\eopt(n)$ edges, \cref{thm:one-cycle} implies that $G$ must be complete oddly bipartite. In particular, $G$ cannot have more than $\eopt(n)$ edges.
\end{proof}

\subsection{Preliminaries}

Let $V$ be a set of vertices and $\TT\subseteq\binom V3$ a set of triples of vertices. Its \emph{shadow} $\partial\TT$ is the set
\[\partial\TT=\left\{xy\in\binom{V(G)}2:xyz\in\TT\text{ for some }z\in V(G)\right\}.\]
Throughout this section, we utilize the following structure, which can be found whenever a hypergraph $G$ differs from a complete oddly bipartite 4-graph on a small fraction of 4-tuples.

\begin{definition}\label{def:eps-close}
Let $G$ be a 4-graph and fix two disjoint vertex subsets $A,B\subseteq V(G)$. Given $\eps>0$, we say $G$ is \emph{$\eps$-close to $\Godd(A,B)$ via} a set of triples $\TT\subseteq\binom{A\cup B}3$ if the following three conditions hold.
\begin{enumerate}
	\item For each $xyz\in\TT$, if $U\in\{A,B\}$ is the neighborhood of $xyz$ in $\Godd(A,B)$, then $\deg_G(xyz;U)\geq(1-\eps)|U|$.
	\item For each $xy\in\partial\TT$, we have $\deg_{\TT}(xy;A)\geq(1-\eps)|A|$ and $\deg_{\TT}(xy;B)\geq(1-\eps)|B|$.
	\item For each $x\in A\cup B$, we have $\deg_{\partial\TT}(x;A)\geq(1-\eps)|A|$ and $\deg_{\partial\TT}(x;B)\geq(1-\eps)|B|$.
\end{enumerate}
\end{definition}

\begin{remark}\label{rmk:eps-close}
To prove $\eps$-closeness when $\min(|A|,|B|)\geq|V(G)|/3$, it is generally easier to prove the following two statements, which immediately imply (2) and (3).
\begin{enumerate}
	\item[(2')] For each $xy\in\partial\TT$, we have $\deg_{\TT}(xy)\geq(1-\eps/3)\times |V(G)|$.
	\item[(3')] For each $x\in A\cup B$, we have $\deg_{\partial\TT}(x)\geq(1-\eps/3)\times |V(G)|$.
\end{enumerate}	
\end{remark}

The primary motivation for \cref{def:eps-close} is the following lemma, which allows us to find walks in $G$ of the appropriate stretch modulo 4 between two triples of $\TT$ whenever they are connected by a tight walk in $\Godd(A,B)$.

\begin{lemma}\label{lem:make-walk}
Let $G$ be a 4-graph and suppose $A,B\subseteq V(G)$ are disjoint subsets of vertices.
Suppose $G$ is $\eps$-close to $\Godd(A,B)$ via some set $\TT\subseteq\binom{A\cup B}3$, and suppose $\eps\leq 0.1$.
If $x_1\cdots x_{\ell+3}$ is a tight walk in $\Godd(A,B)$ with $\ell\geq 6$ such that $x_1x_2x_3,x_{\ell+1}x_{\ell+2}x_{\ell+3}\in\TT$, then there are vertices $y_4,\dots,y_{\ell}\in V(G)$ such that $x_1x_2x_3y_4\cdots y_\ell x_{\ell+1}x_{\ell+2}x_{\ell+3}$ is a tight walk of the same stretch in $G$.
\end{lemma}

\begin{proof}
For $1\leq i\leq\ell+3$, let $U_i\in\{A,B\}$ be the set with $x_i\in U_i$. Note that, for any $i\geq 3$ and any $(x'_{i-2},x'_{i-1},x'_i)\in U_{i-2}\times U_{i-1}\times U_i$, the neighborhood of $x'_{i-2}x'_{i-1}x'_i$ in $\Godd(A,B)$ is $U_{i+1}\setminus\{x'_{i-2},x'_{i-1},x'_i\}$.

For convenience, set $y_i=x_i$ for $1\leq i\leq 3$.
We inductively build a walk $y_1\cdots y_{\ell-3}$ such that $y_i\in U_i$ for each $i\geq 1$ and $y_{i-2}y_{i-1}y_i\in\TT$ for each $i\geq 3$. At the $i$th step (for $i\geq 4$), suppose $y_1\cdots y_{i-1}$ is a walk with $y_{i-3}y_{i-2}y_{i-1}\in\TT$. As observed above, the neighborhood of $y_{i-3}y_{i-2}y_{i-1}$ in $\Godd(A,B)$ is exactly $U_i$; thus, \cref{def:eps-close}(1) implies that there are at least $(1-\eps)|U_i|$ vertices $y_{i}\in U_i$ with $y_{i-3}y_{i-2}y_{i-1}y_{i}\in E(G)$. By \cref{def:eps-close}(2), $y_{i-2}y_{i-1}y_i\in\TT$ for all but at most $\eps |U_i|$ choices of $y_i$. It follows that there are at least $(1-2\eps)|U_i|$ choices of $y_i$ satisfying the desired conditions; because $\eps<\frac 12$, this is a positive number.

Now, suppose we have constructed a walk $x_1x_2x_3y_4\cdots y_{\ell-3}$ in $G$ with $y_{\ell-5}y_{\ell-4}y_{\ell-3}\in\TT$. Repeating the above argument, there are at least $(1-2\eps)^3|U_{\ell-2}||U_{\ell-1}||U_\ell|$ choices of $(y_{\ell-2},y_{\ell-1},y_\ell)\in U_{\ell-2} \times U_{\ell-1}\times U_\ell$ such that $x_1x_2x_3y_4\cdots y_\ell$ is a walk in $G$. Similarly, there are at least $(1-2\eps)^3|U_{\ell-2}||U_{\ell-1}||U_\ell|$ choices of the same vertices with $x_{\ell+3}x_{\ell+2}x_{\ell+1}y_\ell y_{\ell-1}y_{\ell-2}$ a walk in $G$. Because $\eps\leq 0.1$ and hence $(1-2\eps)^3>\frac 12$, the pigeonhole principle implies there is a choice of $(y_{\ell-2},y_{\ell-1},y_\ell)\in U_{\ell-2} \times U_{\ell-1}\times U_\ell$ such that $x_1x_2x_3y_4\cdots y_\ell$ and $y_{\ell-2}y_{\ell-1}y_\ell x_{\ell+1}x_{\ell+2}x_{\ell+3}$ are simultaneously walks in $G$.
\end{proof}

When applying \cref{lem:make-walk} it is easiest to only specify the sequence of sets $U_4,\ldots,U_{\ell}$, each being either $A$ or $B$, for which we seek $y_i\in U_i$. The resulting walk is notated as a \emph{tight walk of the form}
$$x_1x_2x_3U_4U_5\cdots U_\ell x_{\ell+1}x_{\ell+2}x_{\ell+3}.$$
Such a walk exists by \cref{lem:make-walk} if there are at least three intermediate sets $U_i$, the triples $x_1x_2x_3$ and $x_{\ell+1}x_{\ell+2}x_{\ell+3}$ are in $\TT$, and, for any choice of intermediate vertices $(x_4,\ldots,x_s)\in U_4\times\cdots\times U_\ell$, the resulting sequence $x_1\cdots x_{\ell+3}$ is a walk in $\Godd(A,B)$. In all applications, these hypotheses follow immediately from the setup.

Our last preliminary lemma is a useful tool for constructing the set $\TT$ of triples described in \cref{def:eps-close}.

\begin{lemma}\label{lem:make-T}
Fix constants $\al,\eps\in(0,1]$ and fix $\del\leq \min(\eps^{1/4},\al/8)$. Let $V$ be a set of $n$ vertices, with $n$ sufficiently large in terms of $\del$. Suppose $\TT\subset\binom V3$ has size $|\TT|\geq (\al-\eps)\frac{n^3}6$ and satisfies $\deg_\TT(xy)\leq \al n$ for each $xy\in\binom V2$.
Then there is $V'\subseteq V$ of size at least $(1-\del^2)n$ and $\TT'\subseteq\TT\cap \binom{V'}3$ such that $\deg_{\TT'}(xy)\geq(\al-7\del)n$ for each $xy\in\partial\TT'$ and $\deg_{\partial\TT'}(x)\geq(1-4\del)n$ for each $x\in V'$.
\end{lemma}

\begin{proof}
For each $x\in V$, it holds that
\[\deg_\TT(x)=\frac 12\sum_{y\in V-\{x\}}\deg_\TT(xy)\leq\frac{n}2\times \al n=\frac{\al n^2}2.\]

Let $V'=\{x\in V:\deg_\TT(x)\geq(\al-\del^2)n^2/2\}$ and set $m=|V'|$. We have
\[
|V\setminus V'|\times\frac{\del^2n^2}2\leq\sum_{x\in V}\left(\frac{\al n^2}2-\deg_{\TT}(x)\right)=\frac{\al n^3}2-3|\TT|\leq\frac{\eps n^3}2.
\]
so $|V\setminus V'|\leq \eps n/\del^2\leq \del^2n$ and $m\geq\left(1-\del^2\right)n$. 

Let $\TT''=\TT\cap\binom{V'}3$ and let $\EE=\left\{xy\in\binom{V'}2:\deg_{\TT''}(xy)<(\al-\del)n\right\}$.
For each $x\in V'$, observe that $\deg_{\TT}(x)-\deg_{\TT''}(x)\leq \binom{n-1}2-\binom{m-1}2\leq\frac{n^2-m^2}2$. Hence,
\[
\deg_{\TT''}(x)\geq\deg_{\TT}(x)-\frac{n^2-m^2}2
\geq (\al-\del^2)\frac{n^2}2-\frac{2\del^2 n^2}2=(\al-3\del^2)\frac{n^2}2.
\]
It follows that
\begin{align*}
\del n\times \deg_{\EE}(x)
&\leq\sum_{y\in V'-\{x\}}\left(\al n-\deg_{\TT''}(xy)\right)
\leq\al mn-2\deg_{\TT''}(x)
\\&\leq\al n^2-(\al-3\del^2)n^2=3\del^2n^2.
\end{align*}
Thus $\deg_{\EE}(x)\leq 3\del^2n^2/\del n=3\del n$ for each $x\in V'$.

Let $\TT'=\{xyz\in\TT'':xy,yz,xz\notin\EE\}$. For $xy\in\binom{V'}2-\EE$, we have
\[
\deg_{\TT'}(xy)\geq\deg_{\TT''}(xy)-(\deg_\EE(x)+\deg_\EE(y))
\geq(\al-\del)n-6\del n=(\al-7\del)n.
\]
In particular, $\deg_{\TT'}(xy)>0$, and we conclude that $\partial\TT'=\binom{V'}2-\EE$. It follows that for any $x\in V'$,
\[
\deg_{\partial\TT'}(x)=m-1-\deg_{\EE}(x)\geq (1-\del^2)n-1-3\del n\geq (1-4\del) n
\]
if $n$ is sufficiently large in terms of $\del$.
\end{proof}

\subsection{Proof of \cref{thm:one-cycle}}\label{subsec:prove-one-cycle}

Fix absolute constants $\eps_4\leq \frac 1{200}$ and $\eps_3,\eps_2,\eps_1,\eps_0$ with each $\eps_i$ chosen sufficiently small in terms of $\eps_{i+1}$. Specifically, we take 
\[\eps_3\leq\eps_4^2/10,\quad
\eps_2\leq(\eps_3/100)^4,\quad
\eps_1\leq(\eps_2/100)^5,\quad\text{and}\quad
\eps_0\leq\epsref{thm:stability}(\eps_1)/25.\]
Choose $L_0\equiv 1\pmod 4$ so that $\eps_0\geq 8L_0^{-1/4}$. We prove \cref{thm:one-cycle} for this value of $L_0$.

Suppose $n$ is sufficiently large in terms of the absolute constants $\eps_4,\ldots,\eps_0,L_0$.
Let $G$ be a $n$-vertex 4-graph that is $C_{L_0}^{(4)}$-hom-free with $e(G)=\ex(n,C_{L_0}^{(4)}\texthom)$. As noted in \cref{prop:eopt}, $\ex(n,C_{L_0}^{(4)}\texthom)\geq \eopt(n)\geq(n^4-6n^3)/48$. By \cref{prelim:turan-regular}, for any vertex $v\in V(G)$, it holds that
\[
\deg_G(v)+\binom n2\geq\max_{w\in V(G)}\deg_G(w)\geq\frac 1n\sum_{w\in V(G)}\deg_G(w)=\frac{4e(G)}n\geq\frac{n^3-6n^2}{12},
\]
so $\deg_G(v)\geq \frac{n^3}{12}-n^2$ for each $v\in V(G)$. 

By \cref{delete-long-cycles}, we may delete at most $\eps_0n^4$ edges from $G$ such that the resulting 4-graph $G'$ is $\C14$-hom-free. Because $e(G')\geq e(G)-\eps_0n^4\geq\left(\frac 12-25\eps_0\right)\frac{n^4}{24}$, \cref{thm:stability} implies there is a partition $V(G)=A_1\cup B_1$ with $|A_1|,|B_1|\geq\left(\frac 12-\eps_1\right)n$, such that $|E(G')\setminus E(\Godd(A_1,B_1))|\leq\eps_1n^4$.

This proof is organized into two types of statements, Claims and Obstructions, to aid the reader. An {Obstruction} is a claim regarding structures forbidden in the $C_{L_0}^{(4)}$-hom-free 4-graph $G$, which usually encapsulates the key idea in the proof of the succeeding claim. The proofs of the obstructions each locate a homomorphic copy of $C_\ell^{(4)}$ with $\ell\in\{5,7,9,13,17,21\}$, usually by applying \cref{lem:make-walk}. By \cref{prop:hom-free-chain}, these are forbidden in the $C_{L_0}^{(4)}$-hom-free 4-graph $G$, provided $L_0\geq 21$.

The proof comprises five claims. In \cref{claim:make-T}, we identify sets $A_2$ and $B_2$ whose union is almost all of $V(G)$, such that $G$ is $\eps_2$-close to $\Godd(A_2,B_2)$ via some set $\TT$ of triples. In \cref{claim:make-T'}, we add the remaining vertices of $V(G)$ to $A_2$ or $B_2$, resulting in a partition $V(G)=A_3\cup B_3$ of the entire vertex set such that $G$ is $\eps_3$-close to $\Godd(A_3,B_3)$ via some $\TT'\supseteq\TT$. Then, in \cref{claim:A4,claim:B4,claim:A2B2}, we study the symmetric difference of $E(G)$ and $E(\Godd(A_3,B_3))$.
Most claims are preceded by one to two obstructions, which have relatively simple proofs, but highlight the key structural observations.

\begin{claim}\label{claim:make-T}
There are sets $A_2\subseteq A_1$ and $B_2\subseteq B_1$ and $\TT\subseteq\binom{V(G)}3$ such that $G$ is $\eps_2$-close to $\Godd(A_2,B_2)$ via $\TT$. Moreover, $|A_2|,|B_2|\geq\left(\frac 12-\eps_2\right)n$.
\end{claim}

\begin{proof}
Let $H=\Godd(A_1,B_1)-G$	. We observe that
\begin{align*}
e(H)&=|E(\Godd(A_1,B_1)\setminus E(G)|=|E(G)\setminus \Godd(A_1,B_1)|+e(\Godd(A_1,B_1))-e(G)
\\&\leq |E(G)\setminus\Godd(A_1,B_1)|
\leq |E(G)\setminus E(G')|+|E(G')\setminus E(\Godd(A_1,B_1))|
\\&\leq (\eps_0+\eps_1)n^4\leq 2\eps_1n^4,
\end{align*}
as $e(\Godd(A_1,B_1))\leq\eopt(n)\leq e(G)$.

Let $\TT_0=\left\{xyz\in\binom V3:\deg_H(xyz)\leq 8\eps_1^{1/5}n\right\}$. Setting $\overline{\TT_0}=\binom V3-\TT_0$, we have that
\[
8\eps_1^{1/5}n|\overline{\TT_0}|\leq \sum_{xyz\in\binom V3}\deg_H(xyz)= 4e(H)\leq8\eps_1n^4,
\]
so $|\TT_0|\geq\binom n3-\eps_1^{4/5}n^3\geq\frac{n^3}6-1.5\eps_1^{4/5}n^3$.

Applying \cref{lem:make-T} to $\TT_0$ with $(\al,\eps)_{\ref{lem:make-T}}=\big(1,9\eps_1^{4/5}\big)$ yields a set $V'\subseteq V(G)$ of size $|V'|\geq\left(1-3\eps_1^{2/5}\right)n$ and a family $\TT\subseteq\binom{V'}3$ satisfying $\deg_{\TT'}(xy)\geq(1-\eps_2/3)n$ for each $xy\in\partial\TT'$ and $\deg_{\partial\TT}(x)\geq(1-\eps_2/3)n$ for each $x\in V'$. Thus, $\TT$ satisfies conditions (2')--(3') of \cref{rmk:eps-close}, implying \cref{def:eps-close}(2)--(3).

Let $A_2=A_1\cap V'$ and let $B_2=B_1\cap V'$. We have $|A_2|\geq |A_1|-3\eps_1^{2/5}n\geq\left(\frac 12-\eps_1-3\eps_1^{2/5}\right)n\geq\left(\frac 12-\eps_2\right) n$; similarly $|B_2|\geq\left(\frac 12-\eps_2\right)n$. \cref{def:eps-close}(1) holds because, for any $xyz\in\binom{V'}3$ whose neighborhood in $\Godd(A_2,B_2)$ is $U\in\{A_2,B_2\}$, we have that
\[
\deg_G(xyz;U)\geq|U|-\deg_H(xyz)\geq|U|-8\eps_1^{1/5}n\geq|U|-\eps_2n.
\]
Thus, $\TT$ satisfies \cref{def:eps-close}(1)--(3) for the pair $(A_2,B_2)$.
\end{proof}

\cref{claim:make-T'} bootstraps the structure $(A_2,B_2,\TT)$ constructed in \cref{claim:make-T} to build larger sets $(A_3,B_3,\TT')$ such that $A_3\cup B_3=V(G)$ and $G$ is $\eps_3$-close to $\Godd(A_3,B_3)$ via $\TT'$. This claim relies on the following obstructions. 

\begin{obstruction}\label{obstruction:build-TT'}
Let $u\in V(G)$ be any vertex and let $\LL(u)$ denote the link of $u$ in $G$. Suppose $a_1,a_2,a_3\in A_2$ and $b_1,b_2,b_3\in B_2$.
\begin{enumerate}
\item $a_1a_2b_1$ and $a_1b_1b_2$ cannot both be in $E(\LL(u))\cap\TT$.
\item $a_1a_2a_3$ and $a_1a_2b_1$ cannot both be in $E(\LL(u))\cap\TT$.
\item $b_1b_2a_1$ and $b_1b_2b_3$ cannot both be in $E(\LL(u))\cap\TT$.
\end{enumerate}
In particular, for any $xy\in\binom{A_2\cup B_2}2$, its neighborhood in $E(\LL(u))\cap\TT$ is contained either in $A_2$ or in $B_2$.
\end{obstruction}

\begin{proof}
(1) If $a_1a_2b_1$ and $a_1b_1b_2$ were both in $E(\LL(u))\cap\TT$, then applying \cref{lem:make-walk} twice would yield a tight walk of the form
$$a_2a_1b_1u b_2a_1b_1B_2 B_2A_2B_2b_1 b_2a_1ub_1 a_2a_1A_2B_2 A_2 a_2a_1b_1u,$$
which is a homomorphic copy of $C_{21}^{(4)}$.

(2) If $a_1a_2a_3$ and $a_1a_2b_1$ were both in $E(\LL(u))\cap\TT$, then applying \cref{lem:make-walk} would yield a tight walk of the form
$a_3a_1a_2u b_1a_1a_2A_2 B_2A_2A_2A_2 B_2 a_3a_1a_2u,$
 which is a homomorphic copy of $C_{13}^{(4)}$.

(3) The proof is analogous to (2).	
\end{proof}

\begin{claim}\label{claim:make-T'}
There are sets $A_3\supseteq A_2$ and $B_3\supseteq B_2$ and $\TT'\supseteq\TT$ such that $V(G)=A_3\cup B_3$ and $G$ is $\eps_3$-close to $\Godd(A_3,B_3)$ via $\TT'$.	
\end{claim}

\begin{proof}
Let $V_2=A_2\cup B_2$ and set $V_3=V(G)\setminus V_2$. We construct $\TT'$ by, for each $u\in V_3$, adding some triples $\{uxy:xy\in\EE_u\}$ to $\TT$, where $\EE_u\subseteq\binom{V_2}2$ is carefully chosen to have ``large minimum degree'' in several senses. More specifically, for each $u\in V_3$ and $xy\in\EE_u$, we need $\deg_G(uxy)$ to be close to $n/2$ and $\deg_{\EE_u}(x)$ and $\deg_{\TT}(xy)$ to be close to $n$.
For each $u\in V_3$, set $\TT_u=\TT\cap E(\LL(u))$, where $\LL(u)$ is the link of $u$ in $G$. It holds that $|\TT|\geq\left(1-2\eps_2\right)\binom{|V_2|}3$, because choosing $xy\in\binom{V_2}2$ and $z\in V_2\setminus\{x,y\}$ uniformly at random yields
\[
\Pr[xyz\in\TT]=\Pr[xy\in\partial\TT]\times\Pr[xyz\in\TT\mid xy\in\partial\TT]
\geq(1-\eps_2)^2\geq 1-2\eps_2.
\]
Thus,\vspace{-0.1in}
\begin{align*}
|\TT_u|&\geq e(\LL(u))-\left(\binom n3-|\TT|\right)\geq\deg_G(u)-\left(\binom n3-(1-\eps_2)\binom{|V_2|}3\right)
\\&\geq\frac{n^3}{12}-n^2-\left(1-(1-\eps_2)(1-2\eps_2)^3\right)\frac{n^3}6\geq\frac{n^3}{12}-8\eps_2\frac{n^3}6.
\end{align*}
Moreover, \cref{obstruction:build-TT'} implies that for any $xy\in\partial\TT_u$, its neighborhood $N_{\TT_u}(xy)$ is contained entirely within either $A_2$ or $B_2$. In particular, $\deg_{\TT_u}(xy)\leq\max(|A_2|,|B_2|)\leq\left(\frac 12+\eps_2\right)n$ for any $xy\in\partial\TT_u$. 

Applying \cref{lem:make-T} to each set $\TT_u$ with $(\al,\eps)_{\ref{lem:make-T}}=(\frac 12+\eps_2,9\eps_2)$ yields a set of vertices $V'_u$ of size at least $\left(1-3\eps_2^{1/2}\right)n$ and a set $\TT'_u\subseteq\TT_u\cap\binom{V'_u}3$ satisfying $\deg_{\TT'_u}(xy)\geq\left(\frac 12-\eps_3/3\right)n$ for each $xy\in\partial\TT'_u$ and $\deg_{\partial\TT'_u}(x)\geq (1-\eps_3/3)n$ for each $x\in V'_u$. Let $\TT'=\TT\cup\bigcup_{u\in V_3}\{uxy:xy\in\partial\TT'_u\}$, so $\partial\TT'_u$ is the set $\EE_u$ alluded to earlier.

To split vertices of $V_3$ between $A_3$ and $B_3$, we rely on the following subclaim.

\begin{subclaim}\label{subclaim:every-link-good}
For each $u\in V_3$, the set $\TT'_u$ is a subset of either $\binom{A_2}3\cup\left(A_2\times\binom{B_2}2\right)$ or $\left(\binom{A_2}2\times B_2\right)\cup\binom{B_2}3$.
\end{subclaim}

\begin{proof}\proofofclaim
We first make a useful observation about $\TT'_u$. Fix a triple $xyz\in\TT'_u$ and let $U,U'\in\{A_2,B_2\}$ be the sets with $x\in U$ and $y\in U'$. Recall that $N_{\TT'_u}(yz)\subseteq N_{\TT_u}(yz)$ is completely contained in $U$. Hence, for $x'\in U$ chosen uniformly at random, we have
\[
\Pr[x'yz\in\TT'_u]=\frac{\deg_{\TT'_u}(xy)}{|U|}\geq\frac{(1/2-\eps_3/3)n}{2n/3}\geq \frac 34-\frac{\eps_3}2.
\]
Iterating this argument twice shows that for $x'\in U$ and $y'\in U'$ chosen independently and uniformly at random we have
\[
\Pr[x'y'z\in\TT'_u]\geq\Pr[x'yz\in\TT'_u]\times\Pr[x'y'z\in\TT'_u\mid x'yz\in\TT'_u]\geq\left(\frac 34-\frac{\eps_3}2\right)^2>\frac 12.
\]
If $\TT'_u$ contained triples $a_1a_2a_3\in\binom {A_2}3$ and $a_4a_5b_1\in\binom {A_2}2\times B_2$, then for random $a'_1,a'_2\in A_2$, the triples $a'_1a'_2a_3$ and $a'_1a'_2b_1$ would both be contained in $\TT'_u$ with positive probability. This contradicts \cref{obstruction:build-TT'}(2). An analogous argument (using \cref{obstruction:build-TT'}(3)) shows that $\TT'_u$ cannot contain edges $b_1b_2b_3\in\binom{B_2}3$ and $a_1b_4b_5\in A_2\times\binom{B_2}2$.
Lastly, if $\TT'_u$ contained $a_1a_2b_1\in\binom{A_2}2\times B_2$ and $a_3b_3b_4\in A_2\times\binom{B_2}2$, then $\TT'_u$ would contain both $a'_1a_2b'_1$ and $a'_1b'_1b_4$ with positive probability --- where $a'_1\in A_2$ and $b'_1\in B_2$ are chosen randomly --- contradicting \cref{obstruction:build-TT'}(1).

We conclude that $\TT'_u$ is a subset of $\binom {A_2}3\cup\left({A_2}\times\binom {B_2}2\right)$ or $\binom {B_2}3\cup \left(\binom {A_2}2\times {B_2}\right)$ or $\binom {A_2}3\cup\binom {B_2}3$. Moreover, the third option is impossible because $\partial\TT'_u$ has minimum degree at least $(1-\eps_3/3)n>\frac n2\geq\min(|A_2|,|B_2|)$, and if $\TT'_u\subseteq\binom{A_2}3\cup\binom{B_2}3$, then $\partial\TT'_u\subseteq\binom{A_2}2\cup\binom{B_2}2$.
\end{proof}

Say $u\in V_3$ is of \emph{type $A$} if $\TT'_u\subseteq\left(\binom{A_2}2\times B_2\right)\cup\binom{B_2}3$ and of \emph{type $B$} if $\TT'_u\subseteq\binom{A_2}3\cup\left(A_2\times\binom{B_2}2\right)$. Set
\[A_3=A_2\cup\{u\in V_3: u\text{ is of type }A\}\text{ and }
B_3=B_2\cup\{u\in V_3: u\text{ is of type }B\}.
\tag{\thetheorem.1}\label{eq:types-AB}\]
We claim that $G$ is $\eps_3$-close to $\Godd(A_3,B_3)$ via $\TT'$. Because $G$ is already $\eps_2$-close to $\Godd(A_2,B_2)$ via $\TT$, we need only check these \cref{def:eps-close}(1)--(3) for $xyz\in\TT'\setminus\TT$, for $xy\in\partial\TT'\setminus\partial\TT$, and for $x\in V(G)\setminus V_2$, respectively. Conditions (2)--(3) are proven in the slightly stronger forms detailed in \cref{rmk:eps-close}(2')--(3')

\begin{enumerate}
	\item Every element of $\TT'\setminus\TT$ may be written as $uxy$ with $u\in V_3$ and $xy\in\partial\TT'_u$. Note that
	\[
	\deg_{\TT'_u}(xy)\geq\left(\frac 12-\eps_3/3\right)n\geq\left(1-\eps_3\right)\left(\frac 12+\eps_2\right)n\geq(1-\eps_3)\max(|A_3|,|B_3|).
	\]
	Additionally, letting $U\in\{A_3,B_3\}$ be the neighborhood of $uxy$ in $\Godd(A_3,B_3)$, we see by (\ref{eq:types-AB}) that $N_{\TT'_u}(xy)$ is completely contained in $U$. Because $\TT'_u\subseteq E(\LL(u))$, it also holds that $uxyz\in E(G)$ for each $xyz\in\TT'_u$. Thus, $\deg_G(uxy;U)\geq\deg_{\TT'_u}(xy)\geq(1-\eps_3)|U|$.
	\item[(2')] Every element of $\partial\TT'\setminus\partial\TT$ may be written as $ux$ with $u\in V_3$ and $x\in V'_u$. For each $xy\in\partial\TT'_u$, we have $uxy\in\TT'$, so $\deg_{\TT'}(ux)\geq\deg_{\partial\TT'_u}(x)\geq(1-\eps_3/3)n$.
	\item[(3')] For each $u\in V_3$, we have $\deg_{\partial\TT'}(u)=|V'_u|\geq(1-\eps_3/3)n$.\qedhere
\end{enumerate}
\end{proof}

Write $A$ and $B$ for the sets $A_3$ and $B_3$ constructed by \cref{claim:make-T'}, so that $G$ is $\eps_3$-close to $\Godd(A,B)$ via $\TT'$. The remainder of this proof is devoted to showing that $e(G)\leq e(\Godd(A,B))$.

Let $E(G[A^i\times B^{4-i}])$ denote the set of edges of $G$ with $i$ vertices in $A$ and $4-i$ vertices in $B$. Write $e(G[A^i\times B^{4-i}])$ for the size of this set. Similarly, write $\ebar(G[A^i\times B^{4-i}])=e(\Gbar[A^i\times B^{4-i}])$, where $\Gbar$ is the complement of $G$. When $i$ is 4 or 0, we use the shortened notation $G[A^4]$ or $G[B^4]$.

In \cref{claim:A4}, we bound $e(G[A^4])$ in terms of $\ebar(G[A^3\times B])$. The same arguments will bound $e(G[B^4])$ in terms of $\ebar(G[A\times B^3])$. The proof relies on the following forbidden configurations.

\begin{obstruction}\label{obstruction:A4}
Suppose the four vertices $a_1,a_2,a_3,a_4\in A$ form an edge $a_1a_2a_3a_4\in E(G)$.
\begin{enumerate}
\item For each $b\in B$, the four 4-tuples of the form $a_ia_ja_kb$ cannot all be edges of $G$. In particular, $\deg_G(a_ia_ja_k;B)\leq\frac 34|B|$ for some size-3 subset $\{i,j,k\}$ of $\{1,2,3,4\}$.
\item $a_1a_2a_4$ and $a_1a_3a_4$ cannot both be in $\TT'$.
\item If $a_1a_4$ and $a_2a_4$ are both in $\partial\TT'$, then one of $a_1a_3a_4$ and $a_2a_3a_4$ has at most $\eps_3|B|$ neighbors in $B$.
\end{enumerate}
\end{obstruction}

\begin{proof}
(1) If all four 4-tuples $a_ia_ja_kb$ were present in $E(G)$, then the five vertices $\{a_1,a_2,a_3,a_4,b\}$ would induce a copy of $C_5^{(4)}=K_5^{(4)}$ in $G$. It follows that
\[
\sum_{\{i,j,k\}\subset[4]}\deg_G(a_ia_ja_k;B)
=\sum_{b\in B}\#\{\{i,j,k\}\subset[4]:a_ia_ja_kb\in E(G)\}\leq 3|B|
\]
so one of the four triples $a_ia_ja_k$ has at most $\frac 34|B|$ neighbors in $B$.

(2) If $a_1a_2a_4,a_1a_3a_4\in\TT'$ then applying \cref{lem:make-walk} would yield a tight walk of the form $$a_2a_1a_4a_3BAAABa_2a_1a_4a_3,$$ which is a homomorphic copy of $C_9^{(4)}$.

(3) Because $a_1a_4\in\partial\TT'$, there are at most $\eps_3|B|$ vertices $b_1\in B$ with $a_1a_4b_1\notin\TT'$. Thus, if $a_1a_3a_4$ had more than $\eps_3|B|$ neighbors in $B$, then there would be $b_1\in B$ with  $a_1a_4b_1\in\TT'$ and $a_1a_3a_4b_1\in E(G)$. Similarly, if $a_2a_3a_4$ had more than $\eps_3|B|$ neighbors in $B$, then there would be $b_2\in B$ with $a_2a_4b_2\in\TT'$ and $a_2a_3a_4b_2\in E(G)$. Applying \cref{lem:make-walk} would yield a tight walk of the form $b_1a_1a_3a_4a_2b_2 AAA b_1a_1a_3a_4$, which is a homomorphic copy of $C_9^{(4)}$.
\end{proof}

\begin{claim}\label{claim:A4}
$e(G[A^4])\leq \frac{12\eps_3|A|}{|B|}\times\ebar(G[A^3\times B])$.
\end{claim}

\begin{proof}
Let $T=\{a_1a_2a_3\in\binom{A}3: \deg_G(a_1a_2a_3;B)\leq \frac 34|B|\}$. By \cref{obstruction:A4}(1), any edge $a_1a_2a_3a_4\in E(G[A])$ has a sub-triple satisfying $a_ia_ja_k\in T$. Observe also that $\ebar(G_1[A^3\times B])\geq\frac 14|B|\times|T|$.

Partition $T$ into $T_1$ and $T_2$ as follows. Set $a_1a_2a_3\in T_1$ if at most one of the three pairs $a_1a_2,a_1a_3,a_2a_3$ is in $\partial\TT'$ and set $a_1a_2a_3\in T_2$ if at least two of these pairs are in $\partial\TT'$. We count edges $a_1a_2a_3a_4\in E(G[A^4])$ in two steps. First, we count those edges $a_1a_2a_3a_4$ for which some sub-triple $a_ia_ja_k$ is in $T_2$. Call these the edges of type 1. Then, we count the remaining edges not counted in the first step --- these are called the edges of type 2. Note that every edge of type 2 has some sub-triple $a_ia_ja_k$ contained in $T_1$.

First, consider a triple $a_1a_2a_3\in T_2$; without loss of generality we assume $a_1a_2,a_1a_3\in\partial\TT'$. For all but at most $2\eps_3|A|$ vertices $a_4\in A$, the triples $a_1a_2a_4$ and $a_1a_3a_4$ are both in $\TT'$, and it follows by \cref{obstruction:A4}(2) that $a_1a_2a_3a_4$ cannot be an edge of $G$. Thus, $\deg_G(a_1a_2a_3;A)<2\eps_3|A|$ for each $a_1a_2a_3\in T_2$ and we conclude that $G[A^4]$ contains at most $2\eps_3|A|\times |T_2|$ edges of type 1.

To count edges of type 2, we first make the following observation. If $a_1a_2a_3a_4\in E(G[A^4])$ is such that the three pairs $a_1a_4,a_2a_4,a_3a_4$ containing $a_4$ are all in $\partial\TT'$, we claim that $a_1a_2a_3a_4$ is of type 1. Indeed, it follows from \cref{obstruction:A4}(3) that one of the triples $a_1a_3a_4$ and $a_2a_3a_4$ has at most $\eps_3|B|<\frac 34|B|$ neighbors in $B$, and is thus an element of $T_2$. 

If $a_1a_2a_3\in T_1$, there are at most $3\eps_3|A|$ vertices $a_4\in A$ for which one of the three pairs $a_1a_4$, $a_2a_4$, or $a_3a_4$ is not in $\partial\TT'$. Thus, there are at most $3\eps_3|A|$ edges of type 2 containing each triple $a_1a_2a_3\in T_1$, and we conclude that $G[A^4]$ has at most $3\eps_3|A|\times |T_1|$ edges of type 2.

Hence,
\[
e(G[A^4])\leq 2\eps_3|A|\times |T_2|+3\eps_3|A|\times |T_1|\leq 3\eps_3|A|\times |T|.
\]
Recalling that $\ebar(G[A^3\times B])\geq\frac 14|B|\times|T|$, we have that $e(G[A^4])\leq \frac{12\eps_3|A|}{|B|}\times\ebar(G[A^3\times B])$.
\end{proof}

Swapping $A$ and $B$ in \cref{claim:A4} yields the following claim, whose proof is analogous.

\begin{claim}\label{claim:B4}
$e(G[B^4])\leq\frac{12\eps_3|B|}{|A|}\times e(G[A\times B^3])$.	
\end{claim}

Our last claim bounds $e(G[A^2\times B^2])$ in terms of $\ebar(G[A^3\times B])+\ebar(G[A\times B^3])$. We leverage several forbidden configurations of a simple flavor, given in \cref{obstruction:A2B2-easy}, as well as a more in-depth structure discussed in \cref{obstruction:A2B2-hard}.

\begin{obstruction}\label{obstruction:A2B2-easy}
Suppose the four vertices $a_1,a_2\in A$ and $b_1,b_2\in B$ form an edge $a_1a_2b_1b_1\in E(G)$.
\begin{enumerate}
\item $a_1b_1a_2$ and $a_1b_1b_2$ cannot both be in $\TT'$.
\item If $b_1b_2\in\partial\TT'$, then either $a_1b_1b_2$ or $a_2b_1b_2$ has no more than $\eps_3|B|$ neighbors in $|B|$.
\end{enumerate}
\end{obstruction}

\begin{proof}
(1) If $a_1b_1a_2,a_1b_1b_2\in\TT'$ then \cref{lem:make-walk} would yield a tight walk of the form
$$a_1a_2b_1b_2 a_1BBB ABb_2b_1 a_1a_2AB A a_1a_2b_1b_2,$$
which is a homomorphic copy of $C_{17}^{(4)}$.

(2) Because $b_1b_2\in\partial\TT'$, there are at most $\eps_3|B|$ vertices $b'\in B$ with $b_1b_2b'\notin\TT'$. If both $a_1b_1b_2$ and $a_2b_1b_2$ had more than $\eps_3|B|$ neighbors in $B$, then there would be vertices $b'_1,b'_2\in B$ satisfying $b_1b_2b'_i\in\TT'$ and $a_ib_1b_2b'_i\in E(G)$ for $i=1,2$. Applying \cref{lem:make-walk} would yield a tight walk of the form
$a_1a_2b_1b_2 b'_2ABB BAb'_1b_1 b_2 a_1a_2b_1b_2,$
which is a homomorphic copy of $C_{13}^{(4)}$.
\end{proof}

For our final obstruction and claim, we adopt the notation $\deg(xy;A^2)$ and $\deg(x;A^3)$ for $\deg\big(xy;\binom A2\big)$ and $\deg\big(x;\binom A3\big)$, respectively, to avoid multiline binomial coefficients. It is also convenient to introduce the notation $\degbar_G$, which counts non-edges in $G$. More formally, define
\begin{align*}
\degbar_G(xyz;A)&=|A|-\deg_G(xyz;A),
&\degbar_G(xyz;B)&=|B|-\deg_G(xyz;B),
\\\degbar_G(xy;A^2)&=|A|^2/2-\deg_G(xy;A^2),
&\hspace*{-0.15in}\text{and}\quad\degbar_G(x;A^3)&=|A|^3/6-\deg_G(x;A^3).
\end{align*}
Note that $\degbar_G$ also counts tuples with a repeated vertex as non-edges; this choice simplifies the following computations.

\begin{obstruction}\label{obstruction:A2B2-hard}
Suppose that the four vertices $a_1,a_2\in A$ and $b_1,b_2\in B$ form an edge $a_1a_2b_1b_2\in E(G)$.
Suppose further that $\degbar_G(a_1a_2b_i;A)\leq\eps_4|A|$ for $i=1,2$.
\begin{enumerate}
\item Let $a'\in A$ be any vertex. We cannot have $\deg_{G}(a'a_ib_i;A)>4\eps_4|A|$ and $\degbar_G(a'b_i;A^2)<\eps_4^2|A|^2$ for both indices $i=1,2$.
\item In particular, we cannot have $\deg_G(a_ib_i;A^2)>\left(\frac 29+4\eps_4\right)|A|^2$ for both indices $i=1,2$.
\end{enumerate}
\end{obstruction}

\begin{proof}
(1)	Assume for the sake of contradiction that both equations hold for both indices $i=1,2$. For $i=1,2$, choose $a'_i\in N(a'a_ib_i;A)$ uniformly at random. We show that the sequence $C=b_2a_2a_1b_1a'_1a'a'_2b_2a_2a_1b_1$ is a tight walk with positive probability, in which case it forms a homomorphic copy of $C_7^{(4)}$. First, note that $b_2a_2a_1b_1$, $a_1b_1a'_1a'$, and $a'a'_2b_2a_2$ are always edges of $G$. For the other four possible edges, we have
\begin{align*}
\Pr[a_2a_1b_1a'_1\notin E(G)]&\leq\frac{\degbar_G(a_2a_1b_1;A)}{\deg_G(a'a_1b_1;A)}<\frac{\eps_4|A|}{4\eps_4|A|}=\frac 14,
\\\Pr[b_1a'_1a'a'_2\notin E(G)]
&\leq\frac{2\degbar_G(a'b_1;A^2)}{\deg_G(a'a_1b_1;A)\times \deg_G(a'a_2b_2;A)}
	<\frac{2\eps_4^2|A|^2}{(4\eps_4|A|)^2}=\frac 18,
\end{align*}
and analogously $\Pr[a_2'b_2a_2a_1\notin E(G)]<1/4$ and $\Pr[a_1'a'a_2'b_2\notin E(G)]< 1/8$. It follows that the random sequence $C$ forms a homomorphic copy of $C_7^{(4)}$ with probability at least $1-\left(\frac 14+\frac 18+\frac 14+\frac 18\right)>0$, contradicting the assumption that $G$ is $C_7^{(4)}$-hom-free.

(2) Partition $A=A'_1\sqcup A'_2\sqcup A''_1\sqcup A''_2$ such that $\degbar_G(a'b_i;A^2)\geq\eps_4^2|A|^2$ if $a'\in A'_i$ and $\deg_G(a'a_ib_i;A)\leq4\eps_4|A|$ if $a'\in A''_i$. We have $|A|=|A'_1\cup A''_1|+|A'_2\cup A''_2|$; without loss of generality, assume that $|A'_1\cup A''_1|\leq |A|/2$. Observe that
\[
|A'_2|\times \eps_4^2|A|^2\leq\sum_{a'\in A'_2}\degbar_G(a'b_2;A^2)\leq 3\degbar_G(b_2;A^3).
\]
To bound $\degbar_G(b_2;A^3)$, note that $a'_1,a'_2,a'_3\in A$ chosen uniformly at random satisfy
\begin{align*}
\Pr[b_2a'_1a'_2a'_3\in E(G)]\geq&\Pr[b_2a'_1\in\partial\TT']\times\Pr[b_2a'_1a'_2\in\TT'\mid b_2a'_1\in\partial\TT']
\\&\times\Pr[b_2a'_1a'_2a'_3\in E(G)\mid b_2a'_1a'_2\in\TT']
\\\geq&(1-\eps_3)^3> 1-3\eps_3.
\end{align*}
Thus, $\degbar_G(b_2;A^3)\leq 3\eps_3|A|^3/6$ and $|A'_2|\leq3\degbar_G(b_2;A^3)/\eps_4^2|A|^2\leq 3\eps_3|A|/2\eps_4^2\leq |A|/6$.
It follows that $|A''_2|\geq|A|/2-|A'_2|\geq|A|/3$. We conclude that
\begin{align*}
\deg_G(a_2b_2;A^2)&\leq\binom{|A\setminus A''_2|}2+\sum_{a'\in A''_2}\deg_G(a'a_2b_2;A)
\\&\leq\frac{(|A|-|A''_2|)^2}2+4\eps_4|A||A''_2|
\leq\frac 29|A|^2+4\eps_4|A|^2.\qedhere
\end{align*}
\end{proof}

\begin{claim}\label{claim:A2B2}
$e(G[A^2\times B^2])\leq \frac{19\eps_3}{\eps_4}\times \max\left(\frac{|A|}{|B|},\frac{|B|}{|A|}\right)\times\left(\ebar(G[A\times B^3])+\ebar(G[A^3\times B])\right)$.
\end{claim}

\begin{proof}

Set $T_1=\{a_1b_1b_2\in A\times\binom B2:\deg_G(a_1b_1b_2;B)\leq\eps_3|B|\}$ and let $T_1'=\{a_1b_1b_2\in T_1:b_1b_2\in\partial\TT'\text{ and }a_1b_i\in\partial\TT'\text{ for some }i\in[2]\}$.

We count edges in $E(G[A^2\times B^2])$ in three steps. First, we count edges $a_1a_2b_1b_2$ with $a_1b_1b_2$ or $a_2b_1b_2$ in $T_1'$ --- call these the edges of type 1. Suppose $a_1b_1b_2\in T'_1$ with $a_1b_1\in\partial\TT'$. For all but at most $2\eps_3|A|$ vertices $a_2\in A$, the triples $a_1b_1a_2$ and $a_2b_1b_2$ are both in $\TT'$, in which case \cref{obstruction:A2B2-easy}(1) implies that $a_1a_2b_1b_2\notin E(G)$. It follows that there are at most $2\eps_3|A|\times|T'_1|$ edges of type 1.

To count edges not of type 1, we first make the following observation.
If $a_1a_2b_1b_2$ is an edge of $G$ and the three pairs $a_1b_2,a_2b_2,b_1b_2$ are all in $\partial\TT'$, 
then we claim that $a_1a_2b_1b_2$ is of type 1. Indeed, \cref{obstruction:A2B2-easy}(2) implies that $\deg_G(a_ib_1b_2;B)\leq\eps_3|B|$ for some index $i\in[2]$, implying that $a_ib_1b_2\in T_1'$. Note that for any triple $a_1a_2b_1$, there are at most $3\eps_3|B|$ vertices $b_2\in B$ for which the pairs $a_1b_2,a_2b_2,b_1b_2$ are not all in $\partial\TT'$, and it follows that each triple $a_1a_2b_1$ is contained in at most $3\eps_3|B|$ edges $a_1a_2b_1b_2$ that are not of type 1.

Suppose $a_1a_2b_1b_2$ is an edge of $G$ not of type 1. By \cref{obstruction:A2B2-hard}(2), either $\degbar_G(a_1a_2b_i;A)>\eps_4|A|$ or $\deg_G(a_ib_i;A^2)\leq\left(\frac 29+4\eps_4\right)|A|^2$ for some index $i\in[2]$. We say $a_1a_2b_1b_2$ is of type 2 if the former holds and type 3 if the latter holds. Let $T_2=\{a_1a_2b_1\in\binom A2\times B:\degbar_G(a_1a_2b_1;A)>\eps_4|A|\}$ and let $S_3=\{a_1b_1\in A\times B:\deg_G(a_1b_1;A^2)\leq\left(\frac 29+4\eps_4\right)|A|^2\}$. The observation in the prior paragraph implies that there are at most $3\eps_3|B|\times |T_2|$ edges of type 2 and at most $3\eps_3|B|\times|S_3|\times|A|$ edges of type 3.

Lastly, we relate $|T'_1|$, $|T_2|$, and $|S_3|$ to $\ebar(G[A\times B^3])$ and $\ebar(G[A^3\times B])$. Several lower-order terms arise from quadruples with repeated vertices counted by $\degbar$ but not $\ebar$ --- for example, $\degbar_G(a_1b_1b_2;B)$ counts quadruples $a_1b_1b_2b$ not in $E(G)$, which includes the two quadruples $a_1b_1b_2b_1$ and $a_1b_1b_2b_2$ not counted in $\ebar(G[A\times B^3])$. 
Hence,
\begin{align*}
3\ebar(G[A\times B^3])
	&\geq\sum_{a_1b_1b_2\in T'_1}\left(\degbar_G(a_1b_1b_2;B)-2\right)
	\geq((1-\eps_3)|B|-2)|T'_1|\geq \frac{|B|}2\times |T'_1|,
\\3\ebar(G[A^3\times B])
	&\geq\sum_{a_1a_2b_1\in T_2}\left(\degbar_G(a_1a_2b_1;A)-2\right)
	\geq(\eps_4|A|-2)|T_2|\geq\frac{\eps_4|A|}2\times |T_2|,\ \text{and}
\\3\ebar(G[A^3\times B])
	&\geq\sum_{a_1b_1\in S_3}\left(\degbar_G(a_1b_1;A^2)-\frac 32|A|\right)
\geq\left(\frac{|A|^2}{2}-\left(\frac 29+4\eps_4\right)|A|^2-\frac 32|A|\right)|S_3|
\\&=\left(\frac 14+\frac 1{36}-4\eps_4-\frac 3{2|A|}\right)|A|^2|S_3|
	\geq\frac{|A|^2}4\times|S_3|.
\end{align*}
We conclude that
\begin{align*}
e(G[A^2\times B^2])&\leq 2\eps_3|A|\times |T'_1|+3\eps_3|B|\times|T_2|+3\eps_3|B|\times |A|\times |S_3|
\\&\leq\frac{12\eps_3|A|}{|B|}\times \ebar(G[A\times B^3])+\left(\frac{18\eps_3|B|}{\eps_4|A|}+\frac{36\eps_3|B|}{|A|}\right)\times\ebar(G[A^3\times B])
\\&\leq\frac{19\eps_3}{\eps_4}\times \max\left(\frac{|A|}{|B|},\frac{|B|}{|A|}\right)\times\left(\ebar(G[A\times B^3])+\ebar(G[A^3\times B])\right).\qedhere
\end{align*}
\end{proof}

To complete the proof of \cref{thm:one-cycle}, observe that
\begin{align*}
|E(G)\setminus E(\Godd(A,B))|&=e(G[A^4])+e(G[B^4])+e(G[A^2\times B^2]),\text{ and}
\\|E(\Godd(A,B))\setminus E(G)|&=\ebar(G[A^3\times B])+\ebar(G[A\times B^3]).
\end{align*}
By \cref{claim:A4,claim:B4,claim:A2B2}, it holds that
\begin{align*}
e(G[A^4])&+e(G[B^4])+e(G[A^2\times B^2])
\\&\leq\left(12\eps_3+\frac{19\eps_3}{\eps_4}\right)\times\max\left(\frac{|A|}{|B|},\frac{|B|}{|A|}\right)
\times\left(\ebar(G[A^3\times B])+\ebar(G[A\times B^3])\right).
\end{align*}
We have $12\eps_3+\frac{19\eps_3}{\eps_4}\leq 2\eps_4$. Additionally, \cref{claim:make-T} yields $\max\left(\frac{|A|}{|B|},\frac{|B|}{|A|}\right)\leq\left(\frac 12+\eps_2\right)/\left(\frac 12-\eps_2\right)\leq 2$. Thus,
\[|E(G)\setminus E(\Godd(A,B))|\leq 4\eps_4|E(\Godd(A,B))\setminus E(G)|.\]
Because $4\eps_4<1$, we conclude that $e(G)\leq e(\Godd(A,B))$, with equality if and only if $G=\Godd(A,B)$.

\section{Concluding Remarks}\label{sec:concluding}

\subsection{Twisted tight cycles}\label{subsec:twisted}

The classification derived in \cref{sec:bipartite} actually allows us to classify broader families of cycle-like hypergraphs. In uniformity $r$, given an integer $\ell\geq 2r$ and a permutation $\pi\in S_r$, the \emph{tight cycle} of \emph{length} $\ell$ \emph{twisted by} $\pi$ has $\ell$ vertices $v_1\cdots v_\ell$ and edges of the form $v_{i+1}\cdots v_{i+r}$ for each $0\leq i\leq\ell$, where we define $v_{\ell+1}\cdots v_{\ell+r}:=\pi(v_1\cdots v_r)$. For $\pi\in S_r$, define $\CC_\pi$ to be the family of tight cycles twisted by $\pi$ of length 0 modulo $r$. When $\pi=\cyc^k$, it is straightforward to verify that every $C\in \CC_{\pi}$ is a homomorphic copy of some $C'\in \C kr$ and vice versa, implying that $\CC_{\cyc^k}$-hom-freeness and $\C kr$-hom-freeness are identical conditions. However, let us note that the families $\CC_{\cyc^k}$ and $\C kr$ are not precisely the same.

For twisted tight cycles, the techniques in \cref{sec:bipartite} immediately generalize to show the following strengthening of \cref{thm:visual-applicable}.

\begin{theorem}\label{thm:twisted-coloring}
Fix a uniformity $r\geq 2$ and a permutation $\pi\in S_r$. An $r$-graph is $\CC_\pi$-hom-free if and only if $E(G)$ admits an accordant oriented coloring by the set $\Delbf_\pi$ of pictograms defined in \cref{sec:pictorial-oriented-colorings}.	
\end{theorem}

Similarly, one may easily derive a variant of \cref{delete-long-cycles} from \cref{prop:delete-long-cycles} that applies more generally to the families $\CC_\pi$. Thus, the framework introduced in \cref{sec:framework} is potentially applicable to understanding the Tur\'an densities of long twisted tight cycles in any family $\CC_\pi$.

The most natural question in this direction is in uniformity 3, when $\pi\in S_3$ is a transposition. In this case, \cref{thm:twisted-coloring} states that the $\CC_\pi$-hom-free 3-graphs $G$ are exactly those that admit an accordant oriented coloring by $\Delbf_\pi=\left\{\pinwheel\right\}$. Equivalently, following the approach in \cref{sec:coloring-corollaries}, there is an oriented coloring of $V(G)^{(2)}$ by $\Delbf'=\{\purpleto\}$ such that every edge of $G$ is colored as a directed triangle \purpletri. This latter result was shown previously in \cite{BaLu22}.

Call a 3-graph $G$ \emph{tournament-like} if $E(G)$ admits an accordant oriented coloring by $\left\{\pinwheel\right\}$; equivalently, $E(G)$ is the set of directed triangles in some tournament. It is known (see the proof of \cref{lem:color-ineqs}(5)) that any tournament-like 3-graph on $n$ vertices has at most $\frac 14\binom{n+1}3$ 3-edges, with equality if and only if it arises from a regular tournament. We conjecture that this is the extremal construction for all long 3-uniform tight cycles twisted by a transposition.

\begin{conjecture}\label{conj:tournaments}
Let $\pi\in S_3$ be a transposition. For all but finitely many $C\in\CC_\pi$, we have $\pi(C)=1/4$. Moreover, for $n$ sufficiently large, every extremal $C$-hom-free 3-graph is tournament-like.
\end{conjecture}

This seems the most natural next family to apply the framework in \cref{sec:framework} to. Of the remaining steps in the framework --- recalling that Steps 1 and 3 are handled in general --- it seems that Step 2, including stability, is achievable by analyzing the proof of \cref{lem:color-ineqs}(5). However, we were unable to formulate a cleaning argument achieving Step 4. Part of the challenge is the number of extremal examples. In Step 2, there will be many extremal tournament-like 3-graphs differing from $G$ in at most $\eps n^3$ edges; as a result, it is no longer be clear which edges should be added and/or removed in the cleaning argument.

\subsection{Other families of cycle-like hypergraphs}

Let us mention two other families of cycle-like hypergraphs whose Tur\'an densities might possibly be approached with our framework.

First, we suspect that tight cycles in higher uniformities lie within reach. Here, the challenge is understanding the extremal $\C kr$-hom-free $r$-graphs.
Given the ad hoc nature of \cref{lem:color-ineqs}, there is no hope of generalizing it to higher uniformities without some new idea. However, we suspect that the ideas in \cref{sec:cleaning} would generalize to higher uniformities, assuming the extremal constructions still look similar to blowups.

\begin{problem}
Derive more natural proofs of \cref{thm:C14-turan-density,thm:stability}.
\end{problem}

For residues $k$ relatively prime to $r$, we conjecture that the extremal $\C kr$-hom-free constructions are the natural analogues of the constructions in uniformities 2, 3, and 4. For even uniformities $r$, we suspect that the extremal construction is a complete oddly bipartite hypergraph, with vertex set $A\cup B$ and edge set comprising those size-$r$ subsets of $A\cup B$ with an odd number of vertices in $A$.
For odd uniformities $r$, we suspect that the extremal construction involves all size-$r$ subsets with a nonzero even number of vertices in $A$, together with an iterated blowup in $B$. These constructions are pictured in \Cref{fig:extremal-higher-uniformities}

\begin{figure}\centering
\begin{tikzpicture}[scale=0.7]
	\node[left] at (-1,0) {$A$};
	\coordinate (o1) at (0,0);
	\coordinate (o2) at (3.2,0);
	\coordinate (o3) at (5.65,0);
	\node at (7,0) {$\dots$};
	\draw[yscale=1.3] (o1) circle (1) (o2) circle (0.8) (o3) circle (0.64);
	\draw[rounded corners] (1.9,1.3) --++ (6.1,0) --++ (0,-2.6) -- (2,-1.3) -- cycle;
	\draw[rounded corners] (4.55,1.3*0.8) --++ (0,-2*1.3*0.8) --++ (3.1,0) --++ (0,2*1.3*0.8) -- cycle;
	\node[right] at (1.9+6.1,0) {$B$};
	
	\coordinate (p1) at ($(o1) + (0,0.5)$);
	\coordinate (p2) at ($(o2) + (0,0.4)$);
	\coordinate (p3) at ($(o3) + (0,-0.5)$);
	\coordinate (q1) at ($(o1) + (0,-0.5)$);
	\coordinate (q2) at ($(o2) + (0,-0.4)$);
	\coordinate (q3) at ($(o3) + (0,-0.5)$);
	
	\draw[fill,fill opacity=0.2] 
		($(p1) + (-140:0.6)$) -- ($(p1) + (-170:0.7)$) -- ($(p1) + (170:0.7)$) -- ($(p1) + (140:0.6)$) -- ($(p1)!0.7!(p2|-p1)$) -- cycle;
	\draw[fill,fill opacity=0.2] 
		($(q1) + (-150:0.6)$) -- ($(q1) + (150:0.6)$)
		-- ($(q1)!0.72!(q2|-q1) + (120:0.5)$) 
		-- ($(q1)!0.72!(q2|-q1)$) 
		-- ($(q1)!0.72!(q2|-q1) + (240:0.5)$) -- cycle;
	\draw[fill,fill opacity=0.2] 
		($(p2) + (-140:0.45)$) -- ($(p2) + (-170:0.5)$) -- ($(p2) + (170:0.5)$) -- ($(p2) + (140:0.45)$) -- ($(p2)!0.7!(p3|-p2)$) -- cycle;
	\draw[fill,fill opacity=0.2] 
		($(q2) + (-150:0.45)$) -- ($(q2) + (150:0.45)$)
		-- ($(q2)!0.72!(q3|-q2) + (120:0.4)$) 
		-- ($(q2)!0.72!(q3|-q2)$) 
		-- ($(q2)!0.72!(q3|-q2) + (240:0.4)$) -- cycle;
\end{tikzpicture}
\qquad
\begin{tikzpicture}[scale=0.7]
	\coordinate (o1) at (0,0);
	\coordinate (o2) at (3,0);
	\draw[yscale=1.3] (o1) circle (1) (o2) circle (1);
	\path (o1) --++(-1,0) node[left]{$A$};
	\path (o2) --++(1,0) node[right]{$B$};
	
	\coordinate (h1) at (0,0.69);
	\coordinate (h2) at ($(o2)-(h1)$);
	\draw[fill, fill opacity=0.2] 
		($(h1)+(-100:0.4)$) -- ($(h1)+(-140:0.4)$) -- ($(h1)+(180:0.4)$)
		-- ($(h1)+(140:0.4)$) -- ($(h1)+(100:0.4)$) 
		-- (o2|-h1) -- cycle;
	\draw[fill, fill opacity=0.2] 
		($(h2)+(80:0.4)$) -- ($(h2)+(40:0.4)$) -- ($(h2)+(0:0.4)$)
		-- ($(h2)+(-40:0.4)$) -- ($(h2)+(-80:0.4)$) 
		-- ($(o1|-h2)$) -- cycle;
	\draw[fill, fill opacity=0.2] ($(o1)+(0,0.1)$) --++ (220:0.35) --++ (-60:0.37)
		-- ($(o2)+(0,-0.1)$) --++ (220:-0.35) --++ (-60:-0.37) -- cycle;
\end{tikzpicture}
\caption{The conjectural extremal constructions for tight cycles of length relatively prime to the uniformity in uniformities 5 and 6.}
\label{fig:extremal-higher-uniformities}
\end{figure}
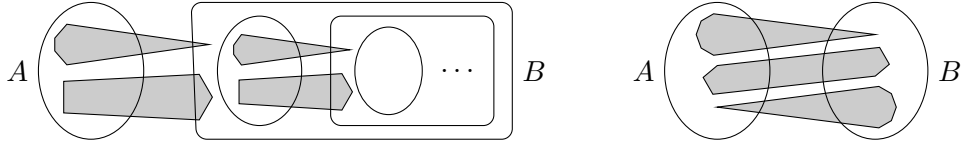

\begin{conjecture}
Fix a uniformity $r$. For all sufficiently long $L$ relatively prime to $r$, the extremal $C_L^{(r)}$-hom-free constructions on any sufficiently large numbers of vertices are given by the complete oddly bipartite $r$-graph (for even $r$) or the iterated blowup described above (for odd $r$), as pictured in \Cref{fig:extremal-higher-uniformities}.
\end{conjecture}

Next, we ask about $C^{(r)-}_L$, the tight $r$-uniform cycle of length $L$ with $r-2$ consecutive edges deleted. In uniformity 3, Balogh and Luo \cite{BaLu22} showed that the extremal $C^{(3)-}_L$-hom-free construction differs from an iterated blowup of a 3-edge on $O(n)$ edges for sufficiently long $L$, a result that was recently generalized to all $L\geq 5$ not divisible by 3 by Lidick\'y, Mattes, and Pfender \cite{LiMaPf24} and Bodn\'ar, Le\'on, Liu, and Pikhurko \cite{BLLP24} using flag algebras. In particular, $\pi(C_L^{(3)-})=1/4$ for $L\geq 5$. This is the best possible bound on $L$ --- for $L=4$, we have $C_4^{(3)-}=K_4^{(3)-}$, which is known to have Tur\'an density at least $2/7$ \cite{FrFu84}.

Let us remark that \cref{conj:tournaments} is a strict strengthening of this result. For example, with $\pi=(1\,2)$, each twisted tight cycle $C=v_1\cdots v_Lv_2v_1v_3$ contains a copy of $C_{L-1}^{(3)-}$ given by the tight walk $v_2\cdots v_Lv_2$. However, \cref{conj:tournaments} predicts a much larger family of conjecturally extremal $C$-free hypergraphs.
More generally, tight cycles minus edges are inherently connected to the tight connectivity parameter, as follows. For an $r$-graph $G$, define $\tc(G)=\{\tc(\x):\x\in\Edir(G)\}$, which is a family of subgroups of $S_r$. The blowup of an edge is the densest 3-graph $G$ with $\tc(G)=\{\{\id\}\}$; that is, $\tc(\x)=\{\id\}$ for each $\x\in\Edir(G)$. This connection is not a coincidence.

\begin{proposition}\label{prop:cycle-minus-edge}
Fix a uniformity $r$ and a nonzero residue $k$ modulo $r$, and let $G$ be an $r$-graph.
If $\tc(G)$ contains any subgroup of $S_r$ other than $\{\id\}$, then $G$ contains a homomorphic copy of $C_L^{(r)-}$ for some $L\equiv k\pmod r$.
\end{proposition}

\begin{proof}
Suppose there were some $\x\in\Edir(G)$ such that $\tc(\x)$ contained some non-identity element $\pi\in S_r$. It has a conjugate $\pi'=\sg\pi\sg^{-1}$ of $\pi$ for which $\pi'^{-1}(1)=r-k+1$; by \cref{prop:tc-properties}(3), $\pi'\in\tc(\sg(x))$. That is, there is a sequence
\[
\sg(\x)=\z^{(0)},\z^{(1)},\ldots,\z^{(s)}=\pi'(\sg(\x))
\]
of oriented edges of $G$ such that $\z^{(i-1)}$ and $\z^{(i)}$ differ on at most one coordinate for each $1\leq i\leq s$. Writing $\z^{(i)}=z_1^{(i)}\cdots z_r^{(i)}$ and recalling from the proof of \cref{prop:findoddcycle} that $\z^{(0)}\z^{(1)}\cdots \z^{(s)}$ is a tight walk, we obtain a tight subwalk
\[
z^{(0)}_{r-k+1}\cdots z^{(0)}_rz_1^{(1)}\cdots z_r^{(1)}\cdots z_1^{(s-1)}\cdots z_r^{(s-1)}z_1^{(s)}.
\]
Because $z_1^{(s)}=z_{\pi'^{-1}(1)}^{(0)}=z^{(0)}_{r-k+1}$, this forms a homomorphic copy of $C_{(s-1)r+k}^{(r)-}$.
\end{proof}

We remark that \cref{prop:cycle-minus-edge} is not an equivalent characterization: there are $r$-graphs $G$ of any uniformity with $\tc(G)=\{\{\id\}\}$ that contain tight cycles minus edges $C_L^{(r)-}$. Nevertheless, it seems likely that this is the dominating condition in many cases.

\begin{question}
For which uniformities $r$ and integers $L$ not divisible by $r$ is $\ex(n,C_L^{(r)-}\texthom)$ exactly the densest $r$-graph $G$ on $n$ vertices with $\tc(G)=\{\{\id\}\}$?
\end{question}

\subsection{Deriving exact results}

The results in this paper, as well as those of \cite{KLP24,BaLu22} for tight cycles and tight cycles minus an edge in uniformity 3 are all focused on understanding the extremal $C$-hom-free hypergraphs for sufficiently long $C$. Because the extremal examples are so well-structured, it seems very likely that $\ex(n,C)$ could be determined exactly for sufficiently large $n$. It would be very interesting to understand how to turn results about $C$-hom-free hypergraphs into results about $C$-free hypergraphs in general, or indeed for any of the families mentioned above.

Let us remark that the ideas in \cref{subsec:prove-one-cycle} show that, for all sufficiently large $L$ not divisible by 4 and all $n$ sufficiently large in terms of $L$, the extremal $C_L^{(4)}$-free 4-graph(s) on $n$ vertices differ from a complete oddly bipartite 4-graph in $o(n^4)$ edges, where the hidden constant depends on $n$. Although this is not an immediate consequence of \cref{thm:one-cycle}, it can be deduced as follows.

Let $G_0$ be an extremal $C_L^{(4)}$-free 4-graph on $n$ vertices. Here, we consider the regime with $L$ a large constant and $n\to\infty$; the dependency on $L$ will be suppressed in our asymptotic notation.
One can delete $o(n^4)$ edges from $G_0$ so that the resulting 4-graph $G_0'$ is $C_L^{(4)}$-hom-free --- for example, by a standard regularity argument in the style of the graph removal lemma (see e.g.\ \cite[Chapter 2]{gtac}). After deleting $o(n)$ vertices from $G'_0$, the resulting graph $G$ will have minimum degree at least $\left(\frac 12-o(1)\right)n^3$. Applying the argument in \cref{subsec:prove-one-cycle} to $G$ --- which requires only that $G$ is $C_L^{(4)}$-hom-free with large minimum degree --- we conclude that there is a partition $V(G)=A\sqcup B$ such that
\[
|E(G)\setminus E(\Godd(A,B))|\leq 4\eps_4|E(\Godd(A,B))\setminus E(G)|,
\]
where $\eps_4$ is the parameter chosen in \cref{subsec:prove-one-cycle}. Thus,
\begin{align*}
E(G)&\leq |E(\Godd(A,B))\cap E(G)|+4\eps_4|E(\Godd(A,B))\setminus E(G)|
\\&\leq \eopt(|V(G)|)-(1-4\eps_4)|E(\Godd(A,B))\setminus E(G)|.
\end{align*}
Because $\eopt(|V(G)|)=(1-o(1))\eopt(n)$ and $E(G)\geq (1-o(1))E(G_0)\geq (1-o(1))\eopt(n)$, we conclude that $G$ and $\Godd(A,B)$ differ on $o(n^4)$ edges. This implies that $G_0$ differs from a complete oddly bipartite graph in $o(n^4)$ edges.

However, it is not clear what the true order of growth of $\ex(n,C_L^{(4)})-\eopt(n)$ might be.

\appendix
\section*{Appendix}
\section{Subgroups of $S_4$}
\label{appendix:subgroups}

In this appendix, we list the subgroups of $S_4$ up to conjugacy. For each of the eleven conjugacy classes, we provide its name (if it is known by a canonical name), an explicit presentation of one of the groups in the conjugacy class, whether or not this conjugacy class avoids conjugates of $\cyc=(1\,2\,3\,4)$ or $\cyc^2=(1\,3)(2\,4)$, and any relevant containment in a larger subgroup of $S_4$.

\begin{table}[h!]
\begin{tabular}{|c|c|c|c|c|}
\hline
Name&Presentation&Avoids $\cyc$&Avoids $\cyc^2$&Contained in
\\\hline\hline
$S_1$&$\{\id\}$
	&yes&yes&$S_3$
\\\hline
$S_2$&$\langle(1\,2)\rangle$
	&yes&yes&$S_3$
\\\hline
&$\langle (1\,2)(3\,4)\rangle$
	&yes&no&$\langle(1\,2),\,(3\,4)\rangle$
\\\hline
$A_3$&$\langle(1\,2\,3)\rangle$
	&yes&yes&$S_3$
\\\hline
&$\langle(1\,2\,3\,4)\rangle$
	&no&no&
\\\hline
&$\langle (1\,2),\,(3\,4)\rangle$
	&yes&no&
\\\hline
&$\begin{array}{l}\{\id,\,(1\,2)(3\,4),
	\\\quad (1\,3)(2\,4),\,(1\,4)(2\,3)\}
\end{array}$
	&yes&no&$A_4$
\\\hline
$S_3$&$\langle(1\,2\,3),\,(1\,2)\rangle$
	&yes&yes&
\\\hline
$D_4$&$\langle(1\,2\,3\,4),\,(1\,3)\rangle$
	&no&no&
\\\hline
$A_4$&$\{\pi\in S_4:\mathrm{sgn}(\pi)=1\}$
	&yes&no&
\\\hline
$S_4$&$S_4$
	&no&no&
\\\hline
\end{tabular}
\end{table}

From this table, we conclude that the maximal $\cyc^2$-conjugate avoiding subgroups of $S_4$ are conjugates of $S_3$ and the maximal $\cyc$-conjugate avoiding groups are conjugates of $S_3$, $A_4$, and $\langle (1\,2)\,(3\,4)\rangle$. We refer to this last group as the \emph{non-normal Klein four-subgroup} of $S_4$, distinguishing it from the \emph{normal Klein four-subgroup} $\{\id,\ (1\,2)(3\,4),\,(1\,3)(2\,4),\,(1\,4)(2\,3)\}$.

\section{Proof of \cref{calculus}}
\label{appendix:calculus}

In this appendix we prove \cref{calculus}. Given nonnegative real numbers $\abcd$, set
\[
f(\abcd)=\min\left(2\del^{3/2}+g(\abc),\frac 14+\frac 34g(\abc)\right)
\]
where 
\[
g(\abc)=\ga^{3/2}+\min\left(0.465,\frac{3\al\bet}{\al+\bet},3\al\sqrt\bet\right).
\]
Throughout this section, we define $\frac{\al\bet}{\al+\bet}$ to be 0 at the point $(\al,\bet)=(0,0)$, so that $\frac{\al\bet}{\al+\bet}$ is continuous on $\mathbb R_{\geq 0}^2$. Because $f$ is a minimum of continuous functions on $\mathbb R_{\geq 0}^4$, it follows that $f$ is continuous on $\mathbb R_{\geq 0}^4$.

\cref{calculus} states that $f$, when restricted to the region
\[
R=\{(\abcd)\in\mathbb R_{\geq 0}^4:2\al+\bet+\ga+2\del=1,\,\ga\leq\al\},
\]
attains a maximum uniquely at the point $f\abcdopt=\frac 12$.
(Because $f$ is continuous and $R$ is closed and bounded, we know that $f$ attains a maximum on $R$.) The proof is given in a sequence of propositions, each deriving further restrictions on the global maxima of $f$. Throughout, we implicitly use that the value of $f$ at its maximum $(\abcd)$ on $R$ must satisfy $f(\abcd)\geq f\abcdopt=0.5$.

We first prove a simple bound on $\frac{\al\bet}{\al+\bet}$.

\begin{proposition}\label{calc:ab:a+b}
For any $\al,\bet\geq 0$, we have $\frac{\al\bet}{\al+\bet}\leq\frac{2\al+\bet}{3+2\sqrt 2}$, with equality if and only if $\bet=\sqrt 2\al$.
\end{proposition}

\begin{proof}
We have $(2\al+\bet)(\al+\bet)-(3+2\sqrt 2)\al\bet=2\al^2+\bet^2-2\sqrt 2\al\bet=(\sqrt 2\al-\bet)^2\geq 0$ with equality if and only if $\bet=\sqrt 2\al$.
\end{proof}

We now show that if $\ga+2\del$ is too small or $\del$ is too large, then $f(\abcd)<0.5$.

\begin{proposition}\label{calc:reduce-to-R0}
Suppose $\abcd$ are nonnegative real numbers with $2\al+\bet+\ga+2\del=1$ and $\ga\leq\al$.
\begin{enumerate}
	\item If $\ga+2\del\leq 0.1$, then $f(\abcd)<0.5$.
	\item If $\del\geq 0.18$, then $g(\abc)<\frac 13$ and $f(\abcd)<0.5$.
\end{enumerate}
\end{proposition}

\begin{proof}
(1) If $\ga+2\del\leq 0.1$, then $\ga^{3/2}+2\del^{3/2}\leq(\ga+2\del)^{3/2}\leq 0.1^{3/2}$. Thus,
\[
f(\abcd)\leq 2\del^{3/2}+\ga^{3/2}+0.465\leq 0.1^{3/2}+0.465\approx 0.497<0.5.
\]
(2)	By \cref{calc:ab:a+b},
\[
g(\al,\bet,\ga)\leq \ga^{3/2}+\frac{3\al\bet}{\al+\bet}\leq \ga\left(\sqrt\ga-\frac 3{3+2\sqrt 2}\right)+\frac{3(2\al+\bet+\ga)}{3+2\sqrt 2}.
\]
If $\del\geq 0.18$, then $2\al+\bet+\ga=1-2\del\leq 0.64$. Additionally, $3\ga\leq 2\al+\ga\leq 0.64$, so
\[
\sqrt\ga\leq \left(\frac{0.64}3\right)^{1/2}<\frac{3}{3+2\sqrt 2}.
\]
It follows that
\[
g(\abc)\leq\frac{3\times 0.64}{3+2\sqrt 2}\approx 0.329<\frac 13,
\]
so $f(\abcd)\leq \frac 14+\frac 34g(\abc)<0.5$.
\end{proof}

We now restrict our focus to studying the function
\begin{align*}
f_0(\abcd)&=\ga^{3/2}+2\del^{3/2}+\min\left(\frac{3\al\bet}{\al+\bet}, 3\al\sqrt\bet\right)
\\&=\ga^{3/2}+2\del^{3/2}+\begin{cases}
		3\al\sqrt\bet&\text{if }\sqrt\bet\leq\al+\bet\\
		\frac{3\al\bet}{\al+\bet}&\text{if }\sqrt\bet\geq\al+\bet
\end{cases}
\end{align*}
on the region
\begin{align*}
R_0&=\{(\abcd)\in\mathbb R_{\geq 0}^4:2\al+\bet+\ga+2\del=1,\,\al\leq\ga,\,\ga+2\del\geq 0.1,\,\del\leq 0.18\}
\\&=\{(\abcd)\in R:\ga+2\del\geq 0.1,\,\del\leq 0.18\}.
\end{align*}
To prove \cref{calculus}, it suffices to show that $f_0$ attains a unique maximum on $R_0$ at the point $\abcdopt$, where $f_0\abcdopt=f\abcdopt=\frac 12$. The next four propositions study the global maxima of $f_0$ on $R_0$.
Set $f_1(\abcd)=\ga^{3/2}+2\del^{3/2}+3\al\sqrt\bet$ and $f_2(\abcd)=\ga^{3/2}+2\del^{3/2}+\frac{3\al\bet}{\al+\bet}$, so $f_0=\min(f_1,f_2)$.

\begin{proposition}\label{calc:not-interior}
Let $(\abcd)$ be a global maximum of $f_0$ on the region $R_0$ defined above. Then either $\ga=0$ or $\ga=\al$ or $\del=0$ or $\del=0.18$.
\end{proposition}

\begin{proof}
Suppose that $0<\ga<\al$ and $0<\del<0.18$. Choose $\eps=\min(\frac{\ga}2,\frac{\al-\ga}2,\del,0.18-\del)>0$. Because the function $x^{3/2}$ is convex on $\mathbb R_{\geq 0}$, we have
\begin{align*}
\ga^{3/2}+2\del^{3/2}&<\frac{(\ga+2\eps)^{3/2}+(\ga-2\eps)^{3/2}}2+(\del+\eps)^{3/2}+(\del-\eps)^{3/2}
\\&\leq\max\left((\ga+2\eps)^{3/2}+2(\del-\eps)^{3/2},(\ga-2\eps)^{3/2}+(\del+\eps)^{3/2}\right).
\end{align*}
Thus, we may choose $(\ga',\del')\in\left\{(\ga+2\eps,\del-\eps),(\ga-2\eps,\del+\eps)\right\}$ such that $(\ga')^{2/3}+2(\del')^{2/3}>\ga^{2/3}+2\del^{2/3}$. Moreover, $(\al,\bet,\ga',\del')$ is also in $R_0$ because of our choice of $\eps$. It follows that $f_0(\al,\bet,\ga',\del')>f_0(\abcd)$, contradicting the assumption that $f_0$ attains a global maximum at $(\abcd)$. Hence, both inequalities $0<\ga<\al$ and $0<\del<0.18$ cannot simultaneously hold at a global maximum of $f_0$ on $R_0$.
\end{proof}

\begin{proposition}\label{calc:easy-boundaries}
Let $(\abcd)\in R_0$. If $\al=0$ or $\bet=0$ or $\ga+2\del=0.1$ or $\del=0.18$, then $f_0(\abcd)<0.5$. Additionally, if $\ga=\al$, then $f_0(\abcd)\leq 0.5$, with equality if and only if $(\abcd)=\abcdopt$.
\end{proposition}

\begin{proof}
First, suppose $\min(\al,\bet)=0$. We have $\ga\leq 1/3$, as $3\ga\leq 2\al+\ga\leq 1$, so
\[
f_0(\al,\bet,\ga,\del)=\ga^{3/2}+2\del^{3/2}+0\leq \left(\frac 13\right)^{3/2}+2\times 0.18^{3/2}\approx 0.345<0.5.
\]

Next, suppose $\ga+2\del=0.1$. Then $\ga^{3/2}+2\del^{3/2}\leq (\ga+2\del)^{3/2}\leq 0.1^3$. Additionally, $\frac{\al\bet}{\al+\bet}\leq\frac{2\al+\bet}{3+2\sqrt 2}$ by \cref{calc:ab:a+b}. Because $2\al+\bet=1-(\ga+2\del)=0.9$, we have
\[
f_0(\abcd)\leq f_2(\abcd)\leq(\ga+2\del)^{3/2}+\frac{3\times(2\al+\bet)}{3+2\sqrt 2}=0.1^{3/2}+\frac{3\times 0.9}{3+2\sqrt 2}\approx 0.495<0.5.
\]

Third, suppose $\del=0.18$. We have $f_0(\abcd)=2\del^{3/2}+g(\abc)$. By \cref{calc:reduce-to-R0}, we have $g(\abc)<\frac 13$, so $f_0(\abcd)\leq 2\times 0.18^{3/2}+\frac 13\approx 0.486<0.5$.

Lastly, suppose $\ga=\al$. It holds that
\[
\ga^{3/2}+3\al\sqrt\bet=\al^{3/2}+3\al\sqrt{1-2\del-3\al}.
\]
Treating $c=1-2\del$ as a fixed constant, the function $j(x)=x^{3/2}+3x\sqrt{c-3x}$ is differentiable with regards to $x$ on the interval $\left(0,\frac c3\right)$, with derivative equal to 
\[
j'(x)=\frac 32\sqrt x+3\sqrt{c-3x}-\frac{9x}{2\sqrt{c-3x}}
=\left(\sqrt{c-3x}-\sqrt x\right)\left(3+\frac{9\sqrt x}{2\sqrt {c-3x}}\right).
\]
Because $j'(x)$ is positive for $0<x<\frac c4$ and negative for $\frac c4<x<\frac c3$, the function $j(x)$ is uniquely maximized at $j(\frac c4)=(c/4)^{3/2}+3(c/4)\sqrt{c/4}=c^{3/2}/2=(1-2\del)^{3/2}/2$ on the interval $\left[0,\frac c3\right]$. Thus,
\begin{align*}
f_0(\abcd)&\leq f_1(\abcd)=2\del^{3/2}+\al^{3/2}+3\al\sqrt{1-3\al}
\\&=2\del^{3/2}+j(\al)
\leq 2\del^{3/2}+\frac 12(1-2\del)^{3/2}.
\end{align*}
Observing that $\del^{1/2}\leq 0.18^{1/2}<\frac 12$ and $(1-2\del)^{1/2}\leq 1$, we have that
\[
f_0(\abcd)\leq 2\del^{3/2}+\frac 12(1-2\del)^{3/2}
\leq\frac 12\times 2\del+\frac 12\times(1-2\del)=\frac 12,
\]
with strict inequality if $\del>0$ or $\al\neq\frac{1-2\del}{4}$. The unique point in $R_0$ with $\del=0$ and $\ga=\al=\frac{1-2\del}4$ is $(\abcd)=\abcdopt$, and $f_0$ evaluates to 0.5 at this point.
\end{proof}

\begin{proposition}\label{calc:hard-boundaries}
Let $(\abcd)$ be a global maximum of $f_0$ on the region $R_0$. Then $\sqrt\bet=\al+\bet$ and either $\ga=0$ or $\del=0$.
\end{proposition}

\begin{proof}
By \cref{calc:not-interior}, any global maximum $(\abcd)$ of $f_0$ on $R_0$ must satisfy one of the four equations $\ga=0$ or $\ga=\al$ or $\del=0$ or $\del=0.18$. \cref{calc:easy-boundaries} implies $\del\neq 0.18$ and that, if $\ga=\al$, then $(\abcd)=\abcdopt$ and in particular $\del=0$. Thus, we have either $\ga=0$ or $\del=0$.

If $\sqrt\bet=\al+\bet$, then the conclusion of the proposition is satisfied. We handle the cases $\al>\sqrt\bet-\bet$ and $\al<\sqrt\bet-\bet$ separately. Within each of these cases, we consider the $\ga=0$ and $\del=0$ cases separately. In all cases, we use the strict inequalities $\al>0$, $\bet>0$, $\ga+2\del>0.1$, and $\del<0.18$, which are implied by \cref{calc:easy-boundaries}.

\medskip\noindent
\textit{Case 1.} $\sqrt\bet<\al+\bet$. That is, $f_0(\abcd)=f_1(\abcd)$.

If $\ga=0$, then $(\abd)$ must be a global maximum of $h_1(\abd):=f_1(\al,\bet,0,\del)$ on the region
\[
R_1=\{(\abd)\in\mathbb R^3:2\al+\bet+2\del=1,\,
\al>0,\,\bet>0,\,2\del>0.1,\,\del< 0.18,\,\sqrt\bet<\al+\bet\},
\]
which is an open subset of the hyperplane $\{(\al,\bet,\del):2\al+\bet+2\del=1\}$. Because $h_1$ is differentiable on $R_1$, the gradient $\nabla h_1=\left\langle 3\sqrt\bet,\frac{3\al}{2\sqrt\bet},3\sqrt\del\right\rangle$ must be parallel to $\langle 2,1,2\rangle$. This implies $\sqrt\al=\sqrt\bet=\sqrt\del$, and we conclude that $\al=\bet=\del=\frac 15$. However, $f_0(\abcd)=h_1(\frac 15,\frac 15,\frac 15)\approx 0.447<0.5$, so this point is not a global maximum of $f_0$ on $R_0$.

If $\del=0$, then $(\abc)$ must be a global maximum of $j_1(\abc):=f_1(\abc,0)$ on the region
\[
R'_1=\{(\abc)\in\mathbb R^3:2\al+\bet+\ga=1,\,
\al>0,\,\bet>0,\,\ga>0.1,\,\ga<\al,\,\sqrt\bet<\al+\bet\},
\]
which is an open subset of the hyperplane $\{(\abc):2\al+\bet+\ga=1\}$. Because $j_1$ is differentiable on $R'_1$, the gradient $\nabla j_1=\left\langle 3\sqrt\bet,\frac{3\al}{2\sqrt\bet},\frac 32\sqrt\ga\right\rangle$ must be parallel to $\langle 2,1,1\rangle$. This implies $\al=\bet=\ga$, contradicting the inequality $\ga<\al$ in the definition of $R'_1$.

\medskip\noindent
\textit{Case 2.} $\sqrt\bet>\al+\bet$. That is, $f_0(\abcd)=f_2(\abcd)$.

If $\ga=0$, then $(\abd)$ must be a global maximum of $h_2(\abd):=f_2(\al,\bet,0,\del)$ on the region 
\[
R_2=\{(\abd)\in\mathbb R^3:2\al+\bet+2\del=1,\,
\al>0,\,\bet>0,\,2\del>0.1,\,\del< 0.18,\,\sqrt\bet>\al+\bet\},
\]
which is an open subset of the hyperplane $\{(\al,\bet,\del):2\al+\bet+2\del=1\}$. Because $h_2$ is differentiable on $R_2$, the gradient $\nabla h_2=\left\langle \frac{3\bet^2}{(\al+\bet)^2},\frac{3\al^2}{(\al+\bet)^2},3\sqrt\del\right\rangle$ must be parallel to $\langle 2,1,2\rangle$. This occurs if and only if $\bet=\sqrt 2\al$ and $\del=\frac{\bet^4}{(\al+\bet)^4}=\frac{4}{(1+\sqrt 2)^4}\approx 0.118$. In this case, \cref{calc:ab:a+b} implies
\[
\frac{3\al\bet}{\al+\bet}=\frac{3(2\al+\bet)}{3+2\sqrt 2}=\frac{3(1-2\del)}{3+2\sqrt 2}\approx 0.394,
\]
and we conclude $f_0(\abcd)=h_2(\abd)\approx 0.474<0.5$. Hence, such a point is not a global maximum of $f_0$ on $R_0$.

If $\del=0$ and $\ga\neq 0$, then $(\abc)$ must be a global maximum of $j_2(\abc):=f_2(\abc,0)$ on the region
\[
R'_2=\{(\abc)\in\mathbb R^3:2\al+\bet+\ga=1,\,
\al>0,\,\bet>0,\,\ga>0.1,\,\ga<\al,\,\sqrt\bet>\al+\bet\},
\]
which is an open subset of the hyperplane $\{(\abc):2\al+\bet+\ga=1\}$. Because $j_2$ is differentiable on $R'_2$, the gradient $\nabla j_2=\left\langle \frac{3\bet^2}{(\al+\bet)^2},\frac{3\al^2}{(\al+\bet)^2},\frac 32\sqrt\ga\right\rangle$ must be parallel to $\langle 2,1,1\rangle$. This occurs if and only if $\bet=\sqrt 2\al$ and $\ga=\frac{4\al^4}{(\al+\bet)^4}=\frac{4}{(1+\sqrt 2)^4}\approx 0.118$. This time, we have
\[
\frac{3\al\bet}{\al+\bet}=\frac{3(2\al+\bet)}{3+2\sqrt 2}=\frac{3(1-\ga)}{3+2\sqrt 2}\approx 0.454,
\]
and we conclude $f_0(\abcd)=j_2(\abc)\approx 0.495< 0.5$. Hence, such a point is not a global maximum of $f_0$ on $R_0$.
\end{proof}

\begin{proposition}\label{calc:max}
The unique global maximum of $f_0$ on $R_0$ is $\abcdopt$, where $f_0\abcdopt=\frac 12$.
\end{proposition}

\begin{proof}
Let $(\abcd)$ be a global maximum of $f_0$ on $R_0$. By \cref{calc:hard-boundaries}, $\al=\sqrt\bet-\bet$ and either $\ga=0$ or $\del=0$. We handle these two cases separately. If $\al=\sqrt\bet-\bet$, observe that $3\al\sqrt\bet=\frac{3\al\bet}{\al+\bet}=3\bet-3\bet^{3/2}$.

\smallskip
\textit{Case 1.} $\del=0$. In this case, $\ga=1-2\al-\bet=(1-\sqrt\bet)^2$, so
\[
f_0(\abcd)=\ga^{3/2}+3\bet-3\bet^{3/2}=\big(1-\sqrt\bet\big)^3+3\bet\big(1-\sqrt\bet\big).
\]
Rearranging terms, it holds that
\[
f_0(\abcd)=\frac 12+\frac 12\left(1-6\bet^{1/2}+12\bet-8\bet^{3/2}\right)=\frac 12+\frac 12\left(1-2\sqrt\bet\right)^3
\]
Moreover, the equation $\ga\leq\al$ implies $(1-\sqrt\bet)^2\leq\sqrt\bet(1-\sqrt\bet)$, implying that $\sqrt\bet\geq\frac 12$. It follows that $f_0(\abcd)\leq\frac 12$ with equality only if $\bet=\frac 14$, in which case $\al=\sqrt\bet-\bet=\frac 14$ and $\ga=1-2\al-\bet=\frac 14$.

\smallskip
\textit{Case 2.} $\ga=0$. In this case, $2\del=1-2\al-\bet=1-2\sqrt\bet+2\bet-\bet=(1-\sqrt\bet)^2$, so
\begin{align*}
f_0(\abcd)&=2\del^{3/2}+3\bet-3\bet^{3/2}=\frac 1{\sqrt 2}\left(1-\sqrt\bet\right)^3+3\bet\left(1-\sqrt\bet\right)
\\&\leq\frac 34\left(1-\sqrt\bet\right)^3+3\bet\left(1-\sqrt\bet\right)
=\frac 34\left(\frac{16}{25}+\frac{9}{25}-3\sqrt\bet+7\bet-5\bet^{3/2}\right)
\\&=0.48+\frac 34\left(1-5\sqrt\bet\right)\left(\frac 35-\sqrt\bet\right)^2.
\end{align*}
The equation $\del\leq 0.18$ implies $(1-\sqrt\bet)^2\leq 0.36=0.6^2$, so $\sqrt\bet\geq 0.4$ and $1-5\sqrt\bet<0$. It follows that $f_0(\abcd)\leq 0.48<0.5$ in this case, so $f_0$ has no global maximum with $\ga=0$ and $\sqrt\bet=\al+\bet$.

\smallskip
Combining the results of the two cases, we conclude that $f_0(\abcd)<\frac 12$ on $R_0$ except at the point $f_0\abcdopt=\frac 12$.
\end{proof}

Combining the previous propositions completes the proof of \cref{calculus}.

\begin{proof}[Proof of \cref{calculus}]
We show that the unique global maximum of $f$ on the region
\[R=\{(\abcd)\in\mathbb R^4_{\geq 0}:2\al+\bet+\ga+2\del=0,\ga\leq\al\}\]
occurs at the point $f\abcdopt=\frac 12$. 

By \cref{calc:reduce-to-R0}, we have that $f(\abcd)<\frac 12$ whenever $\ga+2\del\leq 0.1$ or $\del\geq 0.18$, implying that $f(\abcd)<\frac 12$ for any $(\abcd)\in R\setminus R_0$. By \cref{calc:max}, it holds that $f(\abcd)\leq f_0(\abcd)<\frac 12$ for any $(\abcd)\in R_0$ not equal to $\abcdopt$.
\end{proof}

\section*{Acknowledgments} 
The author is grateful to Ben Foster, Daniel Zhu, and Nathan Sheffield for helpful comments, including several suggestions that simplified portions of the presentation. The author is also grateful to Bernard Lidick\'y for testing whether \cref{thm:C14-turan-density} could be proven directly using flag algebras; unfortunately, no such proof was found. Lastly, the author thanks the anonymous referees for many suggestions that improved the quality of this paper.




\providecommand{\bysame}{\leavevmode\hbox to3em{\hrulefill}\thinspace}
\providecommand{\MR}[1]{}
\providecommand{\MRhref}[2]{%
  \href{http://www.ams.org/mathscinet-getitem?mr=#1}{#2}
}


\begin{aicauthors}
\begin{authorinfo}[maya]
  Maya Sankar\\
  Stanford University\\
  Palo Alto, California, USA\\
  mayars\imageat{}stanford\imagedot{}edu \\
  \url{https://mayasankar.github.io}
\end{authorinfo}
\end{aicauthors}

\end{document}